\documentclass[12pt,a4paper]{article}
\usepackage{a4wide,amsmath,amsfonts,amsthm,amscd,amssymb,latexsym,comment}
\usepackage{
graphicx,psfrag}
\usepackage{caption}
\usepackage{subcaption}
\usepackage[shortlabels]{enumitem}
\usepackage[mathscr]{eucal}
\usepackage[usenames]{color}
\usepackage{url}

\renewcommand{\a}{\alpha }
\renewcommand{\b}{\beta }
\newcommand{\G}{\Gamma }
\newcommand{\g}{\gamma }
\newcommand{\D}{\Delta }
\renewcommand{\d}{\delta }

\newcommand{\e}{\varepsilon }

\renewcommand{\i}{\iota }

\renewcommand{\l}{\lambda }

\newcommand{\s}{\sigma }
\renewcommand{\S}{\Sigma }
\renewcommand{\t}{\tau }

\newcommand{\cA}{{\cal{A}}}

\newcommand{\cF}{{\cal{F}}}

\newcommand{\cK}{{\cal{K}}}
\newcommand{\cL}{{\cal{L}}}

\newcommand{\cS}{{\cal{S}}}

\newcommand{\nul}{\emptyset }
\newcommand{\ceil}[1]{{\lceil} {#1} {\rceil}}
\newtheorem{theorem}{Theorem}[section]
\newtheorem{lemma}[theorem]{Lemma}
\newtheorem{corollary}[theorem]{Corollary}
\newtheorem{proposition}[theorem]{Proposition}

\newtheorem{definition}[theorem]{Definition}
\newtheorem*{defn*}{Definition}

\newtheorem{exam}[theorem]{Example}
\newenvironment{example}{\begin{exam} \rm}{\end{exam}}
\newtheorem{remk}[theorem]{Remark}

\numberwithin{equation}{section}
\numberwithin{figure}{section}
\newcommand{\lk}{\operatorname{lk}}
\newcommand{\supp}{\operatorname{supp}}

\newcommand{\genus}{\operatorname{genus}}

\newcommand{\ZZ}{\ensuremath{\mathbb{Z}}}

\newcommand{\RR}{\ensuremath{\mathbb{R}}}

\newcommand{\ZD}{\ZZ[\frac{1}{2}]}

\renewcommand{\cF}{\mathcal{F}}

\renewcommand{\cS}{\mathcal{S}}
\newcommand{\la}{\langle}
\newcommand{\ra}{\rangle}
\newcommand{\mbf}{\mathbf}
\newcommand{\maps}{\rightarrow}

\newcommand{\bs}{\backslash}
\newcommand{\be}{\begin{enumerate}}
\newcommand{\ee}{\end{enumerate}}
\newcommand{\bd}{\begin{description}}
\newcommand{\ed}{\end{description}}
\newcommand{\bei}{\begin{itemize}}
\newcommand{\eei}{\end{itemize}}
\newenvironment{ajd1}{\noindent\color{red} AJD }{}
\newcommand{\ajd}[1]{\begin{ajd1} #1 \end{ajd1}}
\author{Andrew Duncan, Steven Fulthorp\\
School of Mathematics and Statistics\\
Herschel Building\\
University of Newcastle upon Tyne\\
Newcastle upon Tyne
}
\title{Orientable and non-orientable genus $n$ Wicks forms over hyperbolic groups}
\date{\today}
\begin{document}
\maketitle

\begin{abstract}
In 1962 M.J.~Wicks \cite{wicks62} gave a precise description of the form a commutator could take
in a
free group or a free product and in 1973 extended this description to cover a product of two squares \cite{wicks73}. 
Subsequently, lists of ``Wicks forms'' were  
found for arbitrary products of commutators and squares in free groups and free products \cite{Culler81}, \cite{vdovina97}. 
Here we 
construct   Wicks forms for products of commutators and squares in a hyperbolic group. As applications we give 
explicit lists of forms for  a commutator and  for a square, and find bounds on the lengths of conjugating elements 
required to express a quadratic tuple of elements of a hyperbolic group  as a Wicks form. 
\end{abstract}
\section{Introduction}

In 1962 M.J.~Wicks \cite{wicks62} described  algorithms to decide, given an element $u$ of a free group or a free product of groups, 
whether or not $u$ is a commutator; and in 1973 \cite{wicks73} 
 extended the results to cover the case where $u$ is a product of two squares.   
These algorithms are based on the construction of a set of ``forms'' for commutators and squares. 
 For example, an element $u$ in the 
free group  $F(X)$
is shown to be a commutator if and only if it is conjugate to a cyclically reduced word of the form
$F=ABCA^{-1}B^{-1}C^{-1}$, with $A,B,C\in F(X)$ (where cyclically reduced means the word and 
all its cyclic permutations are reduced).
The analogous result for a word $u$ equal to a  product $a^2b^2$ 
 is that $u$ is 
conjugate to a cyclically reduced word  
$F_1=A^2BC^2B^{-1}$ or $F_2=ABACB^{-1}C$. 
 We call the words $F$, $F_1$ and $F_2$ 
\emph{Wicks forms} for commutators and products of two squares, in free groups. 
The lists of forms \cite{wicks62}, \cite{wicks73} for 
commutators and products of two squares in free products of groups are similar though
more involved.  

The question of whether an element is a commutator or product of squares is an example of an equation over a group: 
as defined in Section \ref{sec:gandvk}. 
Landmarks in the theory of equations over groups are the papers of Makanin \cite{Makanin82} showing that
 arbitrary systems of  equations over a free group are solvable, 
in the free group; and of  Razborov \cite{razborov85} giving a parametric
 description of the set of all solutions. The complexity of Makanin's algorithm is very high: it is
not known to be primitive recursive, whereas Wicks' algorithms which show that the quadratic equations $[a,b]=u$ and $a^2b^2=u$, in
variables $a,b$,  are
solvable in free groups and free products, run in time bounded by a  polynomial in the length of $u$. In addition, 
quadratic equations play an important part  in algorithms for  solution of general systems of  equations, and they are closely related
 to compact surfaces. For these reasons, among others,  the
study of quadratic equations is of independent interest. 

 An element of 
a group is said to have  \emph{orientable genus} $n$ if it may be expressed as  a product  
of $n$ commutators, and no fewer, and \emph{non-orientable genus} $n/2$ if it is a product of $n$ squares, 
and no fewer (see Section \ref{sec:gandvk} below). The corresponding quadratic equations
were shown to be solvable in a free groups by Edmunds \cite{edmunds75},
\cite{edmunds79}. Culler \cite{Culler78} and, independently, Goldstein and Turner \cite{GoldsteinTurner79}, used the theory of 
compact surfaces to produce such algorithms; in particular Culler described a  
set of Wicks forms for elements of orientable and non-orientable genus $n$ in the free group.
Algorithms to decide the solvability, and to find a solution,  of a general quadratic equation 
over a free group and over a free product of groups (where quadratic equations are solvable in the factors) were given 
by Comerford and Edmunds \cite{ComerfordEdmunds81}. In the case of the free group, Ol'shanski \cite{ols89} (see also \cite{GK92}) shows that there is such an algorithm 
which runs in time polynomial in the length of the coefficients of the equation (see Section \ref{sec:gandvk} for more detail). 
 The set of all solutions of a quadratic equation in a free group $F(X)$ can, as shown by Comerford and Edmunds \cite{ComerfordEdmunds89}, be described
in terms of a finite set of \emph{basic} solutions and certain $F(X)$-automorphisms of $F(\cA)*F(X)$; and all the parameters of this description
may be effectively computed, given the equation $w=1$. In fact Grigorchuk and Kurchanov \cite{GK92} show that the set of basic 
solutions  of the quadratic equation $w=1$ may be computed in polynomial 
time in the length of the coefficients of $w$. Moreover in \cite{GK92} a set of Wicks forms for solutions of a given quadratic equation 
in a free group 
is shown to exist. 

Quadratic equations in small cancellation groups and hyperbolic groups were studied in \cite{Schupp79}, \cite{Comerford81},   
\cite{lysionok89}. 
Ol'shanskii \cite{ols89}[Theorem 6] (see also \cite{GK92}) constructed  
an algorithm, to determine whether or not a quadratic equation 
of genus $g$ has solution in $H$, 
and if so to find one. Moreover this algorithm runs in time bounded by a polynomial in the sum of the 
lengths of coefficients. A parametric description of 
the set of solutions of a quadratic equation was 
described in  \cite{GrigorchukLysionok}. Generalising Makanin's algorithm, Rips and Sela \cite{RipsSela95} proved the
 existence of an algorithm
for the solvability of a finite system of equations over a torsion-free hyperbolic group, and  Dahmani and Guirardel \cite{DG10} have  
extended this result to cover the case of hyperbolic groups with torsion.  

In \cite{vdovina97} Vdovina described a procedure for constructing Wicks forms
for elements of any orientable genus  in a free product. 
In this paper we use similar methods to construct Wicks 
forms for elements of arbitrary genus in hyperbolic groups. Using the forms of genus $1$ and $1/2$ we then 
list all the possible forms of commutators and squares in
a hyperbolic group. Similar lists of forms could be constructed for  
elements of
higher genus. However, the number of possible extensions increases
dramatically with the increase of genus: Bacher and Vdovina show  
\cite{BacherVdovina} that the number of orientable Wicks forms of genus $g$ is $\Omega(g!)$.
 We also describe, as an 
application of our results, forms for quadratic tuples of words (defined below).

We begin by introducing a number of definitions and preliminary results in Sections \ref{sec:hyp} to \ref{sec:gandvk}. 
The key technique which we call an ``extension'' of a quadratic word is introduced and developed in Sections \ref{sec:ext} and 
\ref{sec:genlen}.  
This puts us in a
 position to state the main results of the paper: Theorem \ref{Thexp} and Theorem \ref{Thexp-}, in Section \ref{sec:mainthm}. 
Section \ref{sec:applications} contains  
some   applications of  these theorems. In Sections \ref{sec:commutators} and \ref{sec:squares}, we prove
Propositions \ref{propos} and Propositions \ref{prop:squares},  which describe the
possible forms for commutators and squares in a
$\delta$-hyperbolic group $H$.   In Section \ref{sec:sqe} we discuss quadratic equations, give bounds
on the length of elements in ``minimal'' solutions to such equations and show how the main theorems
may be used to describe forms for quadratic tuples of elements. 
The proofs of the main theorems are left to Section \ref{sec:proof}. 
In Section \ref{sec:pre-proof} we give 
preliminary results needed to construct an appropriate extension; and complete this construction in Section \ref{stext}. 
Finally, in Section \ref{sec:Rbound} we bound the length of the conjugator which appears in the 
main theorems.  

\section{Definitions and Main results}
\subsection{Hyperbolic groups}\label{sec:hyp} 
For background on hyperbolic groups the reader is referred to
\cite{AlonsoBrady,ghys90,Gromov93}. Here we use the notation of \cite{AlonsoBrady}. 
The Cayley graph of a group $G$ with respect to a generating
set $X$ will be denoted $\G_X(G)$. 
 As usual, the word
metric and makes $G$ into a geodesic metric space.   
We use $[x,y]$ to denote a geodesic path $p$ from $x$ to $y$ and $(x,y]$,
$[x,y)$ and $(x,y)$ to denote $p\backslash \{x\}$, $p\backslash \{x\}$
and $p\backslash \{x,y\}$, respectively. If $p$ is a path from 
$x$ to $y$ then we write $x=\i(p)$, $y=\t(p)$, we denote the reverse of $p$ from $y$ to $x$ by $p^{-1}$ and, if $p$ is geodesic, then we write $|p|=d(x,y)$. 
  Let $w$ be a (reduced) word in
$F(X)$. The length of $w$ as a word is denoted $|w|$.  If  $|v|\geq
|w|$ for all words $v$ in $F(X)$ such that $w= _G v$ then we say that
the word
$w$ is 
$G$-\emph{minimal}. We use the notation $|w| _G$
to denote the length of a $G$-minimal word in $F(X)$ which is equal to $w$ in $G$.
It follows that, for all $g\in G$, a $G$-minimal 
word representing $g$ is the label of a geodesic path from $1$ to $g$ 
in $\G_X(G)$. 

A \emph{geodesic triangle}, denoted $\triangle xyz$, consists of the union of 
three points $x,y,z$, of the 
the Cayley graph, and geodesic paths $[x,y]$, $[y,z]$ and $[x,z]$
joining them. 
Let $H$ be  a group generated by a set $X$ and let
 $T=\triangle xyz$ be a geodesic triangle in $\Gamma _X(H)$. Let 
 $T'=\triangle 'x'y'z'$ be  a Euclidean triangle with sides the same
length as $T$. That is
$d_E(x',y')=d(x,y)$ etc, where $d_E$ is the standard Euclidean
metric. There is a natural identification map $\phi$ from  $T$ to
$T'$. The maximum inscribed circle in $T'$ meets the side
$[x'y']$ (respectively $[y'z']$, $[z'x']$) at a point $c_z$
(respectively $c_x$, $c_y$) such that 
\begin{equation*}
d(x',c_z)=d(x',c_y), d(y',c_x)=d(y',c_z), d(z',c_y)=d(z',c_x).
\end{equation*}
Notice that we have
\begin{equation}\label{dtcon}
d(x',c_z)=\frac{1}{2}(d(x',c_z)+d(x',c_y))=\frac{1}{2}(d(x',z')+d(x',y')-d(z',y')).
\end{equation}
The preimages $\phi^{-1}(c_x)$, 
$\phi^{-1}(c_y)$
and $\phi^{-1}(c_z)$ of $c_x,c_y$ and $c_z$ 
in $T$ are called the \emph{internal points}
of $T$. 

There is a unique isometry $\iota _T$ of the triangle $T'$ onto a \emph{tripod}: that is a tree
with one vertex $p$ of degree three and vertices $x'', y'', z''$ each
of degree one, such that $d(p,z'')=d(z',c_y)=d(z',c_x)$ etcetera. 
Let $\psi$ be the composite map $\psi=\iota_T\circ\phi$.  
We say that a geodesic triangle is \emph{$\delta$-thin} if the fibres of
$\psi$ have diameter at most $\delta$ in $\Gamma_X(H)$. That is, for
all $x, y$ in $T$,
\begin{equation}\label{dt2con}
\psi(x)=\psi(y)\implies d(x,y)\leq \delta.
\end{equation} 

\begin{definition}\label{def:thintriang}
  A group $H$ is said to be $\d$\emph{-hyperbolic}, with respect to 
a finite generating set $X$, if 
all geodesic triangles 
in $\Gamma _X(H)$ are 
$\d$-thin. 
A group is \emph{(word) hyperbolic} if, for some $\d\ge 0$, it is 
$\d$-hyperbolic, with respect to some generating set $X$. 
\end{definition}
In the light of \eqref{dtcon} and \eqref{dt2con} a group is $\delta$-hyperbolic if and 
only if it satisfies the following condition. Let $w$ and $z$  be any words in $F(X)$
which are $H$-minimal,  with
$w=w_1w_2$ and $z=z_1z_2$, $|w|=|w_1|+|w_2|$ and $|z|=|z_1|+|z_2|$. 
\begin{equation}
\textrm{If 
} 
|w_2|=|z_1| 
\leq 
\frac{1}{2}(|w|+|z|-|wz|_H)
\textrm{ then } 
|w_2z_1|_H 
\leq 
\delta. 
\end{equation}
It can be shown that word hyperbolic groups are finitely presented. 
For this, and an account of many characterisations of hyperbolic groups,
see, for example, \cite{AlonsoBrady}. 
For the remainder of the paper, $H$ will denote a $\delta$-hyperbolic
group, with respect to the generating set $X$, and we shall assume that
$H$ has a finite presentation $\la X|S\ra$.

The following lemma of Gromov (see \cite{GrigorchukLysionok})
 shows that the conjugacy problem is
solvable in hyperbolic groups.  
\begin{lemma}[\cite{GrigorchukLysionok}]\label{lb}
If $H$-minimal words $h_1$ and $h_2$ are conjugate in $H$, then a word $w$ can
be found such that $h_1=_Hwh_2w^{-1}$ and 
\begin{equation*}
|w|\leq \frac{1}{2}(|h_1|+|h_2|)+M+1,
\end{equation*}
where $M$ is the number of elements of $H$ represented by words of $F(X)$ of 
length $\leq 4\delta $.
\end{lemma}
(The obvious algorithm for the conjugacy problem based on the lemma above takes  exponential
time, 
in the length of the two input words, but a closer analysis yields  
 linear bounds on the time required: see Bridson and Haefliger \cite{BridsonHaefliger}[pp. 451--454] or 
Bridson and Howie \cite{BridsonHowie05}.)

\subsection{Thin polygons}\label{prelim}
Let $H=\la X|S\ra$ be a finitely generated
hyperbolic group. 
 The following lemma is a more explicit version of \cite{Gromov87}[6.1C]; and of 
\cite{ols89}[Lemma 10].  
 Its statement requires some further notation. Let $\g=[x,y]$ be a geodesic in  $\Gamma_X(H)$ and let 
$v_0,\ldots ,v_n$ be points of $\g$ such that $v_0=x$, $d(x,v_{i-1})<d(x,v_{i})$ and $v_n=y$, and set $\g_i=[v_{i-1},v_i]$, 
$i=1,\ldots, n$. Then the sequence $\g_1,\ldots ,\g_n$ is called a \emph{partition} of $\g$. If $\g=[x,y]$ and $\g'=[x',y']$ are geodesics 
 such that $|\g|=|\g'|$ and, for all $z\in\g$ and $z'\in \g'$, if $d(x,z)=d(x',z')$ then $d(z,z')\le k$; we say that 
$\g$ and $\g'$ $k$\emph{-fellow travel}.

\begin{lemma}\label{lem:la}
Let  $q=\gamma _0\gamma _1\ldots\gamma_n$ be a closed path in $\Gamma_X(H)$, where $\gamma _i$ is a geodesic path, for $i=0,\ldots , n$, $n\ge 1$.
Then 
\bei
\item there exist  integers $r_0,\ldots, r_n$, $2\le r_i\le n$ and,  
for $i=0,\ldots n$, there exists  a partition $\g_i^{(1)},\ldots, \g_i^{(r_i)}$ of $\g_i$ and, 
\item setting $I=\cup_{i=0}^n\{(i,j)\,:\, 1\le j \le r_i\}$, there exists an involution $\s$ of $I$,  
\eei
such that the following hold.  
\be[(i)]
\item\label{it:nla1} Writing $\s(0,i)=(a_i,b_i)$, for $1\le i\le r_0$, 
geodesics $\g_0^{(i)}$ and $(\g_{a_i}^{(b_i)})^{-1}$ are $\d\lceil \log_2(n)\rceil$-fellow travellers and  
$n\ge a_1\ge \ldots \ge a_{r_0}\ge 1$; and 
\item \label{it:nla2} 
fixing $i$ with  $1\le i \le n$ and writing $\s(i,j)=(a_j,b_j)$, geodesics $\g_i^{(j)}$ and $(\g_{a_j}^{(b_j)})^{-1}$ are 
$\d(2\lceil \log_2(n)\rceil-1)$-fellow travellers and
there exists an integer $c_i$, such that $1\le c_i \le r_i -1$ and 
\[n\ge a_{c_i+1}\ge \cdots \ge a_{r_i}\ge i+1\textrm{ and }i-1\ge a_{1}\ge \cdots \ge a_{c_i}\ge 0.\]
\item  \label{it:nla3}  In addition,  $\s(0,1)=(n,n)$, $\s(i,1)=(i-1,n)$, for $1\le i \le n$,  and $\g_i^{(1)}$ and $(\g_{i-1}^{(r_{i-1})})^{-1}$ are 
$\d$-fellow travellers, for $0\le i\le n$ (subscripts modulo
$n+1$). 
\ee
(Here $\i(\g_i^{(j)})$ and $\t(\g_i^{(j)})$ are  not necessarily vertices of $\G_H(X)$. In \ref{it:nla2} if $a_j=0$ then, from \ref{it:nla1}, in fact $\g_i^{(j)}$ and $(\g_{a_j}^{(b_j))^{-1}}$ are 
$\d\lceil \log_2(n)\rceil$-fellow travellers.)
\end{lemma}
An instance of the case $n=4$ is illustrated in Figure \ref{fig:fibres}. 
\begin{figure}
\begin{center}
\psfrag{a1}{{\scriptsize $\g_0^{(1)}$}}
\psfrag{a2}{{\scriptsize $\g_0^{(2)}$}}
\psfrag{a3}{{\scriptsize $\g_0^{(3)}$}}
\psfrag{a4}{{\scriptsize $\g_0^{(4)}$}}
\psfrag{b1}{{\scriptsize $\g_1^{(1)}$}}
\psfrag{b2}{{\scriptsize $\g_1^{(2)}$}}
\psfrag{b3}{{\scriptsize $\g_1^{(3)}$}}
\psfrag{c1}{{\scriptsize $\g_2^{(1)}$}}
\psfrag{c2}{{\scriptsize $\g_2^{(2)}$}}
\psfrag{c3}{{\scriptsize $\g_2^{(3)}$}}
\psfrag{c4}{{\scriptsize $\g_2^{(4)}$}}
\psfrag{d1}{{\scriptsize $\g_3^{(1)}$}}
\psfrag{d2}{{\scriptsize $\g_3^{(2)}$}}
\psfrag{d3}{{\scriptsize $\g_3^{(3)}$}}
\psfrag{d4}{{\scriptsize $\g_3^{(4)}$}}
\psfrag{e1}{{\scriptsize $\g_4^{(1)}$}}
\psfrag{e2}{{\scriptsize $\g_4^{(2)}$}}
\psfrag{e3}{{\scriptsize $\g_4^{(3)}$}}
\includegraphics[scale=0.4]{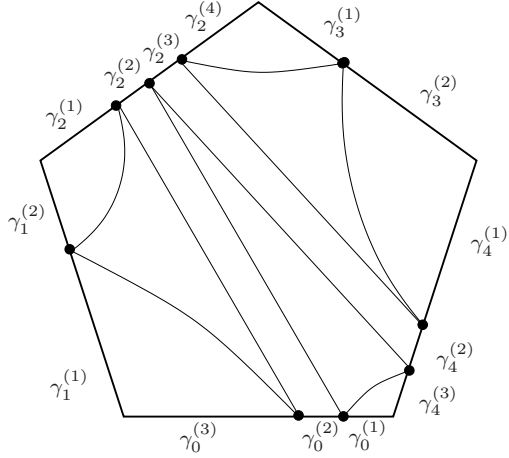}
\caption{$\g_2^{(3)}$ and $(\g_4^{(2)})^{-1}$ are $4\d$-fellow travellers and  $\g_0^{(2)}$  and $(\g_2^{(2)})^{-1}$ are
$2\d$-fellow travellers. All other paired intervals $\d$-fellow travel.}\label{fig:fibres}
\end{center}
\end{figure}
\begin{proof}
Let $k=\ceil{\log_2(n)}$. 
By adding paths of length zero between $\tau(\gamma_n)$ and
$\iota(\gamma_0)$, we may assume that $n=2^k$. Thus  we may replace 
 $ \ceil{\log_2(n)}$ with $k$, throughout the statement of the lemma.
If $k=1$ then the result follows directly from the thin triangles condition: the 
internal points of a triangle determine appropriate partitions of the sides. 

For the induction we need to be able to modify the partitions of the $\g_i$ by adding new points, so we 
begin by describing how to do this.  
The data consisting of the partitions $\g_i^{(1)},\ldots ,\g_i^{(n)}$ of the $\g_i$ and the map $\s$, as
in the lemma, is called a \emph{matching} of $q$, \emph{based} at $\g_0$. 
Given a matching $M$ as described in the statement of the lemma, satisfying \ref{it:nla1}, \ref{it:nla2} and \ref{it:nla3}, supp pose that for 
some $i,j$ we have points $p,r$ such that $\g_i^{(j)}=[p,r]$ and a point $q\in [p,r]$, $p\neq q\ne r$. 
We form a \emph{simple refinement} $M'$ of $M$ as follows. Let $\s(i,j)=(a,b)$, let $\g_a^{(b)}=[p',r']$ and 
let $q'$ be the point of $[p',r']$ such that $d(p,q)=d(q',r')$. Replace the 
partition $\g_i^{(1)},\ldots ,\g_i^{(r_i)}$ of $\g_i$ with the partition $\g_i^{\prime (1)},\ldots ,\g_i^{\prime (r_i+1)}$,
where $\g_i^{\prime (s)}=\g_i^{(s)}$, for $1\le s\le j-1$, $\g_i^{\prime (j)}=[p,q]$, $\g_i^{\prime (j+1)}=[q,r]$ and $\g_i^{\prime (s)}=\g_i^{(s-1)}$,
for $j+2\le s \le r_i+1$.  Replace the 
partition $\g_a^{(1)},\ldots ,\g_a^{(r_a)}$ of $\g_a$ with the partition $\g_a^{\prime (1)},\ldots ,\g_a^{\prime (r_a+1)}$, formed in
exactly the same way (replacing $i$, $j$, $p$, $q$ and $r$,  by $a$, $b$, $p'$, $q'$ and $r'$, throughout). Now replace 
$\s$ with an involution $\s'$ such that
\begin{align*} 
\s'(i,s)&=\s(i,s), 1\le s\le j-1,&\s'(a,t)&=\s(a,t), 1\le t\le b-1 \\
\s'(i,j)&=(a,b+1)\textrm{ and }\s'(i,j+1)=(a,b),&\s'(a,b)&=(i,j+1)\textrm{ and }\s'(a,b+1)=(i,j)\\
\s'(i,s)&=\s(i,s-1), j+2\le s\le r_i+1, & \s'(a,t)&=\s(a,t-1), b+2\le t\le r_b+1,
\end{align*}
and for $u\neq i,a$, 
\[\s'(u,v)=
\begin{cases}
(i,s+1) & \textrm{if }\s(i,s)=(u,v), \textrm{ where } j+1\le s\le r_i,\\
(a,t+1) &\textrm{if }\s(a,t)=(u,v), \textrm{ where } b+1\le t\le r_a,\\
\s(u,v) & \textrm{otherwise}.
\end{cases}
\]
Then these partitions of the $\g_i$ together with the function $\s'$ form a new matching of $q$, which again satisfies
conditions \ref{it:nla1}, \ref{it:nla2} and \ref{it:nla3}. 
A finite sequence of simple refinements of $M$ is called a \emph{refinement}. 

Assume then that $k>1$, that the result holds  for all non-negative 
integers $n$ no larger than $2^{k-1}$ and that  $n=2^k$.  
Let $\g_0^\prime$ be the geodesic path  
from  $\tau(\gamma_{2^{k-1}})$ to $\t(\g_0)$ and  
 let 
$\g_{2^{k-1}}^\prime$  be the geodesic path  from $\i(\g_0)$ to  $\tau(\gamma_{2^{k-1}})$ 
(see Figure \ref{fig:pathdivn}). 
\begin{figure}[htp]
\begin{center}
\psfrag{s}{{\normalsize{$s_1$}}}
\psfrag{z1}{{\normalsize{$\zeta _1$}}}
\psfrag{e1}{{\normalsize{$\eta_1$}}}
\psfrag{q0}{{\normalsize{$\g^\prime_0$}}}
\psfrag{q1}{{\normalsize{$\g^\prime_{2^{k-1}}$}}}
\psfrag{g0}{{\normalsize{$\gamma _0$}}}  
\psfrag{g1}{{\normalsize{$\gamma _1$}}}    
\psfrag{g2}{{\normalsize{$\gamma _2$}}}  
\psfrag{g3}{{\normalsize{$\gamma _3$}}}    
\psfrag{g4}{{\normalsize{$\gamma _{2^{k-1}}$}}}  
\psfrag{g5}{{\normalsize{$\gamma _{2^{k-1}+1}$}}}    
\psfrag{g6}{{\normalsize{$\gamma _6$}}}  
\psfrag{g7}{{\normalsize{$\gamma _7$}}}    
\psfrag{g8}{{\normalsize{$\gamma _{2^k}$}}}  
\includegraphics[scale=0.35]{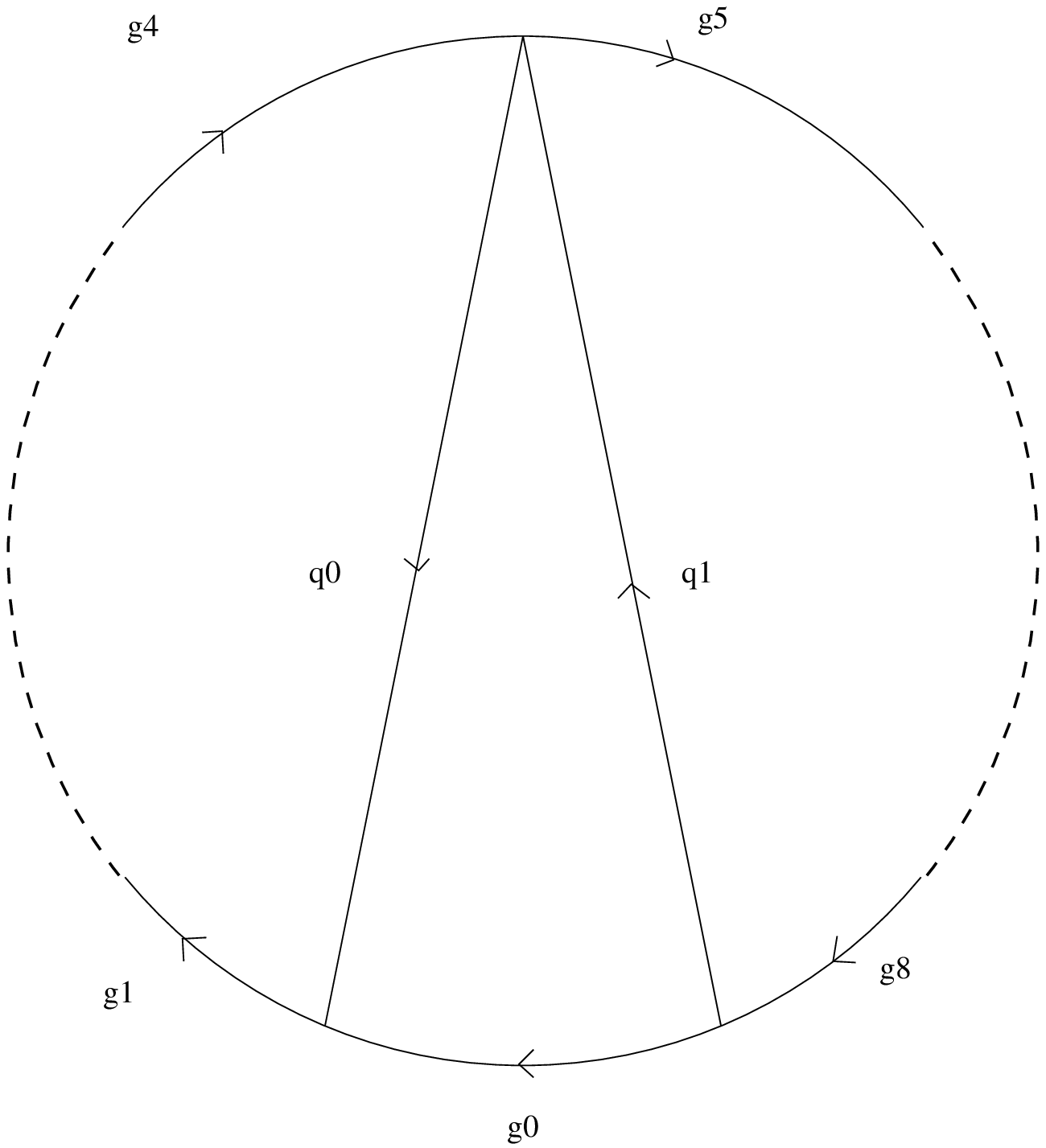}
\caption{}
\label{fig:pathdivn}                                             
\end{center}
\end{figure} 
Let $c_0$, $c'_0$ and $c''_0$ be the internal points 
of triangle $T=\g_0 {\g'}_0^{-1}{\g'}_{2^{k-1}}^{-1}$, with $c_0$ on $\g_0$, $c'_0$ on $\g'_0$ and 
$c''_0$ on $\g'_{2^{k-1}}$ and let $t_0=\i(\g_0)$, $t_1=\t(\g_0)$ and $t_2=\i(\g'_0)$ be the vertices of $T$. 
Let $M_0$ be the corresponding matching, with partitions 
$[t_0,c_0],[c_0,t_1]$ of $\g_0$; $[t_1,c'_0],[c'_0,t_2]$ of $\g_0^{\prime -1}$ and $[t_2,c''_0],[c''_0,t_0]$ 
of ${\g'}_{2^{k-1}}^{-1}$; and involution $\s_0$ such that $\s_0(0,1)=(2,2)$, $\s_0(0,2)=(1,1)$ and $\s_0(1,2)=(2,1)$. 
 Let
$q^\prime_1=\g'_0\g_1\cdots \g_{2^{k-1}}$ and 
$q^\prime_2=\g'_{2^{k-1}}\g_{2^{k-1}+1} \cdots \g_{2^{k}}$. From the 
inductive hypothesis, there exist matchings $M_1$ of 
$q'_1$ and $M_2$ of $q'_2$, based at $\g'_0$ and $\g'_{2^{k-1}}$, respectively.
Suppose that $\s_i$ is the involution associated to $M_i$ and that, under 
these matchings the partitions of $\g'_0$ and $\g'_{2^{k-1}}$ are 
${\g'}^{(1)}_0,\ldots, {\g'}^{(r')}_0$ and ${\g'}^{(1)}_{2^{k-1}},\ldots, {\g'}^{(r'')}_{2^{k-1}}$, 
respectively.  

 We shall refine the partitions $M_1$ and $M_2$, so that the intervals, lying between
$t_2$ and either $c'_0$ or $c''_0$, are of the same number and lengths.
First, single out the end points of intervals of the partitions of  $\g'_0$ and $\g'_{2^{k-1}}$ which lie between 
$[t_2,c'_0]$ and  $[c''_0,t_2]$.
That is, let $S_1=\{p\in [t_2,c'_0]\,:\, p=\t({\g'}^{(j)}_0), 1\le j\le r'\}$ and 
let  $S_2=\{p\in [c''_0,t_2]\,:\, p=\i({\g'}^{(j)}_{2^{k-1}}), 1\le j\le r''\}$. Now form sets of points to be used in the refinement: 
 let $T_1=\{q\in [t_2,c'_0]\,:\, d(q,t_2)=d(p,t_2), \textrm{ for some } p \in S_2\}$ and 
$T_2=\{q\in [c''_0,t_2]\,:\, d(q,t_2)=d(p,t_2), \textrm{ for some } p \in S_1\}$.
Refine $M_1$ by adding the points of $T_1\cup \{c'_0\}$ to $\g'_0$ and refine $M_2$ by adding the 
points of $T_2\cup\{c''_0\}$ to $\g'_{2^{k-1}}$. We shall now consider $M_1$ and $M_2$ to 
be these refined matchings and, after adjusting notation, take the partitions of $\g'_0$ and $\g'_{2^{k-1}}$ to be
${\g'}^{(1)}_0,\ldots, {\g'}^{(r')}_0$ and ${\g'}^{(1)}_{2^{k-1}},\ldots, {\g'}^{(r'')}_{2^{k-1}}$, again. 

Now we refine $M_0$, using the partitions of $\g'_0$ and $\g'_{2^{k-1}}$ under $M_1$ and $M_2$. Let
$U_1=\{p\in [t_2,t_1]\,:\, p=\t({\g'}^{(j)}_0), 1\le j\le r'\}$ and 
$U_2= \{p\in [t_0,t_2]\,:\, p=\i({\g'}^{(j)}_{2^{k-1}}), 1\le j\le r''\}$. Refine $M_0$ by adding the 
points of $U_1\cup U_2$. As before we refer to the refinement as  $M_0$ and to its involution as  $\s_0$. 
Under $M_0$ we have a partition $\g_0^{(1)}, \ldots ,\g_0^{(r_0)}$ of $\g_0$  
with the following properties. For $j$ such that $\t(\g_0^{(j)})\in [t_0,c_0]$ the interval
$\g_{2^{k-1}}^{(j)}$ has terminal point on $[t_0,c''_0]$ and $\d$-fellow travels with  $\g_0^{(j)}$.
From the inductive hypothesis, and the fact that refinements preserve properties \ref{it:nla1},
\ref{it:nla2} and \ref{it:nla3}, $\g_{2^{k-1}}^{(j)}$ and $(\g_a^{(b)})^{-1}$ are $\d(k-1)$-fellow
travellers, where $(a,b)=\s_2(2^{k-1},j)$. Hence $\g_0^{(j)}$ and $(\g_a^{(b)})^{-1}$ are 
$\d k$-fellow travellers and we may set $\s(0,j)=\s_2(2^{k-1},j)$. Similarly,
for $j$ such that $\i(\g_0^{(j)})\in [c_0,t_1]$, if $\s_1(0,r'-r_0+j)=(a,b)$ then the intervals $\g_0^{(j)}$
and $(\g_a^{(b)})^{-1}$ are $\d k$-fellow travellers and we may set $\s(0,j)= \s_1(0,r'-r_0+j)$. 
This defines $\s$ on all pairs $(0,j)$, with $1\le j\le r_0$. If $1\le i\le 2^{k-1}$ and 
$j$ is such that $(i,j)\neq \s(0,j')$, for all $j'$, it follows from construction of $M_1$ that 
$\s_1(i,j)=(a,b)$ where either $a=0$ and ${\g'}_0^{(b)}$ has initial point in $[t_2,c'_0]$ or 
$1\le a\le 2^{k-1}$. In the former case ${\g'}_0^{(b)}$ is a subinterval of $[t_2,c'_0]$,  
which $\d(k-1)$-fellow travels with $(\g_i^{(j)})^{-1}$, 
$\s_0(0,b)=(2^{k-1},l)$, 
and the interval ${\g'}_{2^{k-1}}^{(l)}$ of $\g'_{2^{k-1}}$ is a subinterval of 
$[c''_0,t_2]$ and $\d$-fellow travels with $({\g'}_0^{(b)})^{-1}$. Furthermore, $\s_2(2^{k-1},l)=(c,d)$, for 
some $c$ with $2^{k-1}+1\le c\le 2^k$, and   ${\g'}_{2^{k-1}}^{(l)}$ and $(\g_c^{(d)})^{-1}$ are $\d(k-1)$-fellow travellers.
Hence, in this case, $\g_i^{(j)}$ and $(\g_c^{(d)})^{-1}$ are $2\d(k - 1)$-fellow travellers; and we may set
$\s(i,j)=(c,d)$. In the latter case $\s_1(i,j)=(a,b)$, where $1\le a\le 2^{k-1}$, so $\g_i^{(j)}$ and 
$(\g_a^{(b)})^{-1}$ are $\d(2(k-1)-1)$-fellow travellers, so we may set $\s(i,j)=(a,b)$. This defines 
$\s$ on all pairs $(i,j)$, where $1\le i\le 2^{k-1}$. A similar argument allows us to define
$\s$ on the remaining pairs $(i,j)$ where $2^{k-1}+1\le i \le 2^k$. Therefore we have
a matching $M$ which satisfies \ref{it:nla1} and \ref{it:nla2}. If this matching does
not already satisfy  \ref{it:nla3} then a simple modification results in one that does. 
\end{proof}
\subsection{Quadratic equations, genus and quadratic words}
\label{sec:gandvk}
We write $\ZD$ for the set $\{n/2:n\in \ZZ\}$. We use Stallings' 
definition of the genus of a surface: a connected sum of $g$ torii has
 \emph{genus} $g$, while a connected sum of $g$ projective planes has 
\emph{genus} $g/2$. (Hence the Euler characteristic of a compact, connected,
closed surface $\D$, of genus $g$,  is $\chi(\D)=2-2g$, irrespective of 
orientability.)

By a \emph{graph} we mean a finite directed graph. If $e$ is
an edge of a graph we write $\i(e)$ and $\t(e)$ for the initial 
and terminal vertices of $e$. We use $e^{-1}$ to denote the 
edge with $\i(e^{-1})=\t(e)$ and $\t(e^{-1})=\i(e)$. If $p$ is 
a path in a graph with edge sequence $e_1,\ldots ,e_n$ then we 
denote by $p^{-1}$ the path with edge sequence $e_n^{-1},\ldots 
,e_1^{-1}$. A \emph{labelling} of a graph is a function $\l$ from
directed edges of the graph to a set $L$ such that
if $e$ is an edge with $\l(e)=A$ then $A^{-1}\in L$ and 
$\l(e^{-1})=A^{-1}$. 
 We routinely identify the edge $e$ with its label $\l(e)$. 

Let $\mathcal{A}$ be an countably infinite alphabet. 
The \emph{support} of a word $w$ in the free monoid 
$(\cA\cup \cA^{-1})^*$ is the smallest subset $\cS$ of $\cA$ such that 
$w$ belongs to $(\cS\cup \cS^{-1})^*$, written $\supp(w)$. 
As we are concerned here only with groups, 
we refer to elements  of $(\cA\cup \cA^{-1})^*$ as  \emph{words over} $\cA$. 
A word in $(\cA\cup \cA^{-1})^*$ 
is \emph{(freely) reduced} if it contains no subword $x^\e x^{-\e}$, where $x\in \cA\cup \cA^{-1}$,
 $\e=\pm 1$,  and 
\emph{(freely) cyclically reduced}
if $w$ and all its cyclic permutations are reduced. 
We  use the term {\emph{cyclic word}} to mean the equivalence class $[w]$ of a
word $w$ under the relation which relates two words if one is a cyclic
permutation of the other. 
When we talk about a \emph{ cyclic word} $w$ we mean any element $v$ of $[w]$.

\begin{definition}\label{def:qw}
Let $w_1,\ldots, w_t$ be a $t$-tuple of words over $\mathcal{A}$ 
and let $W=w_1\cdots w_t$ (as a word in $(\cA\cup \cA^{-1})^*$). 
Then  $w_1,\ldots, w_t$ 
is said to be \emph{quadratic} if
each element of $\supp(W)$  appears exactly twice in 
$W$.  
In this case, let $x\in \supp(W)$ and 
let $W=W_0 x^\e W_1 x^\d W_2$, where $\e,\d\in \{\pm 1\}$. Then the
\emph{signature} $\s(x)$ of $x$ \emph{in} $w_1,\ldots, w_t$ is 
$\s(x)=(\e,\d)$. Define $o(x)=-\e\d$, where $\s(x)=(\e,\d)$. 
If $o(x)=1$  then we say that the letter
 $x$ is \emph{alternating}, in $w_1,\ldots, w_t$.  If 
every 
element $x$ of $\supp(W)$  is alternating then $w_1,\ldots, w_t$
is said to be \emph{orientable}.  The \emph{signature} of a quadratic tuple $w_1,\ldots, w_t$ is
the map $\s$ from $\supp(W)$ to  $\{\pm 1\}\times\{\pm 1\}$.
\end{definition}

Let $X$ and $\cA$ be disjoint sets and let $w$ be an element of $F(X)*F(\cA)$. The expression
$w=1$ is called an \emph{equation}, with variables $\cA$.  
 If $H$ is a group with presentation $\la X|S\ra$ then a homomorphism 
$\phi$ from $F(X)*F(\cA)$ to $H$ is  an $H$\emph{-map} if the restriction of $\phi$ to $F(X)$ induces the identity
map on $H$.
 A \emph{solution} to the equation $w=1$ over $H$ is an $H$-map $\phi: F(X)*F(\cA)\maps H$ such that $\phi(w)=_H 1$. 
Given $w\in F(X)*F(\cA)$, write $w=c_0a_1c_1\ldots c_{n-1}a_n$, 
where $a_i\in F(\cA)$ and $c_i\in F(X)$, with $c_i\neq 1$, $i>0$,
 and $a_i\neq 1$, $i<n$. Then the equation $w=1$ is said to have \emph{coefficients} $c_0,\ldots c_{n-1}$. 
The equation $w=1$, with $w=c_0a_1c_1\ldots c_{n-1}a_n$ as above, is \emph{quadratic} if $a_1,\ldots, a_n$ 
is a quadratic $n$-tuple over $\cA$, \emph{orientable} if   $a_1,\ldots, a_n$ is orientable, and 
\emph{non-orientable} otherwise. 

To a quadratic equation $w=1$, with coefficients $c_0,\ldots, c_{n-1}$,  we associate a surface  
$\S$ as follows. 
 Take a copy $D$ of the disk $D^2$ and   divide its
boundary into $|w|$ segments (the length taken in $F(\cA\cup X)$). Write $w$ anti-clockwise around the
boundary of a disk, labelling segments consecutively with the letters
of $\supp(w)$, and directing segments anti-clockwise when a letter occurs with
exponent $1$, and clockwise when the labelling letter has exponent $-1$. For example 
if $w=ABCA^{-1}B^{-1}C^{-1}$, where $A, B, C\in \cA$,  we obtain the labelling at the top left  of Figure \ref{gammaU}. 
 Now identify the
segments labelled by the same letters,
respecting orientation. 
We obtain a compact surface $\S_w$ of genus $g$,
for some $g\ge 0$, $g\in \ZD$,  with $n$ boundary components, 
which is orientable if and only if $w=1$ is orientable. 
After identification the oriented segments on the boundary
of the disk become the edges of  a directed, labelled graph $\G_w$ on the surface $\S_w$;  
with boundary if $n-1>0$.
We call this graph the \emph{graph associated to} $w$. 
 (See Example \ref{ex:wicks}.)  
Every edge $e$ of $\G_w$ is labelled with an element of $\supp(w)$ and 
directed edges are in one to one correspondence with $\supp(w)$; so we
identify edges with $\supp(w)$. Edges on the boundary of $\S$ are labelled
by elements of $X$ while two-sided edges are labelled by elements of $\cA$. 
The equation $w=1$ is said to have \emph{genus} $g$ if $\S$ is a surface 
of genus $g$. 

Now consider the following types of quadratic equation.  
The orientable, genus $g$, 
quadratic 
 equation  
\begin{equation}\label{eq:qgo}
\prod_{i=1}^{t}v_i^{-1}c_iv_i\prod_{i=1}^g[x_i,y_i]=1,
\end{equation}
with variables $v_i,x_i,y_i$, 
 and the non-orientable, genus $g$, 
quadratic 
 equation  
\begin{equation}\label{eq:qgn}
\prod_{i=1}^{t}v_i^{-1}c_iv_i\prod_{i=1}^{2g}z_i^2=1,
\end{equation}
with variables $v_i,z_i$, where the 
 coefficients are the $c_i$ in both cases, are called \emph{standard} quadratic equations. A quadratic equation
$w=1$ is equivalent to a standard quadratic equation: in the  sense that there is an efficient procedure to transform 
$w=1$ into one of \eqref{eq:qgo} or \eqref{eq:qgn}, and to transform a solution $\phi$ of $w=1$ into a solution of the resulting equation;  and
vice-versa (see for example \cite{GrigorchukLysionok}[Proposition 2.2]). 

Now let $w=1$ be the equation \eqref{eq:qgo} or \eqref{eq:qgn} and suppose $\psi$ is a solution to
$w=1$ over $H$. In the 
 construction of the surface $\S_w$ above, we labelled the boundary of a disk $D$ with the word 
$w$. Replacing the label $a$ of each directed edge of $D$ with $\psi(a)$ we obtain a disk
with boundary label $\psi(w)$. As $\psi(w)=_H 1$,   there exists a van Kampen diagram on $D$, with boundary label $\psi(w)$, over 
$H$; and  this gives  a van Kampen diagram on $\S_w$. 

Conversely, consider a van Kampen diagram $K$ on a  genus $g$, surface $\S$  with $t$ boundary components. 
To construct a solution to one of \eqref{eq:qgo} or \eqref{eq:qgn} we 
allow van Kampen diagrams to contain \emph{null-cells}: that is $2$-cells with boundary  
labels of the form $1\cdot x \cdot 1\cdot x^{-1}$, $x\in X$. Given a $2$-cell $R$ of $K$ we may replace 
each directed edge of $R$ with boundary label $x$ by a null-cell; as on the left hand side of Figure 
\ref{gammaU}. We call the result a \emph{thickening} of $R$. Whenever we do this it
 gives us a new van Kampen diagram, on $\S$ with the same boundary labels. 
Returning to $K$, we wish to show that we may cut $\S$ 
along a closed curve $c$ which  separates $\S$ into two surfaces, such that capping each surface
 off with a disk glued along its copy of $c$, one surface $\S_g$ is of closed of genus $g$ and the
other is a sphere $S$ with $t$ boundary components, labelled  $c_1,\ldots ,c_t$, when read with 
the appropriate orientation of $S$. Moreover, we would like to choose $c$ to be a subgraph of  the $1$-skeleton  $K_1$ of $K$. 
Since $\S$ is the $2$-skeleton of $K$, the graph $K_1$ is connected, and we may choose
a minimal connected subgraph $T$ of $K_1$ which contains every boundary component. Thicken every
$2$-cell which meets $T$. The union of the null-cells of (this thickened version of) 
$K$ meeting $T$ is then a genus $0$ sub-surface $\D$ 
of $\S$ containing all  boundary components. The boundary of $\D$ has no self intersections and
is a path in $K_1$. 
If 
$\S$ is orientable and the boundary components are  labelled 
$c_1,\ldots ,c_t$,  when read with some chosen orientation of $\S$  then 
the boundary components of $\D$ are also labelled $c_1,\ldots ,c_t$; and so we may take $c=\d\D$. 
 
In the case that $\S$ is non-orientable and $t>0$, our assumption on  the boundary 
components means that each boundary component may be assigned an orientation 
in such a way that, reading each boundary component's label, with its chosen
orientation, from a suitable base point $b_i$, the boundary labels are $c_1,\ldots ,c_t$. 
This means in turn that, reading   the labels of the boundary components of the disk $\D$ constructed above,
according to some fixed orientation of $\D$, we obtain the list  
 $c_1^{\e_1}, \ldots , c_t^{\e_t}$, with $\e_i=\pm 1$ (ignoring the boundary label of $c$ itself). 
To  modify this disk so that the $\e_i$ are the same, for all $i$, 
 we use the non-orientability of $\S$. 
First choose a vertex $*$ of $K$ on the boundary $c$ of the disk $D$, and a base point $b_i$ on  the boundary
component labelled $c_i$, for each $i$ 
(so that the boundary label may be read from $b_i$). Then, after thickening some $2$-cells if necessary,  we may choose 
paths $p_i$  in $D\cap K_1$, from $*$ to $p_i$, for $i=1,\ldots, t$,  so that the $p_i$ have disjoint interiors and 
the boundary component of the surface obtained by cutting $\S$ along the $p_i$ is  $\prod_{i=1}^t w_ic_i w_i^{-1}$, where
$w_i$ is the label of $p_i$, as in Figure \ref{fig:spray0}. 

\begin{figure}[htp]
\begin{center}
\begin{subfigure}[b]{.3\columnwidth}
\begin{center}
\psfrag{p1}{{\scriptsize $p_1$}}
\psfrag{pi}{{\scriptsize $p_i$}}
\psfrag{pt}{{\scriptsize $p_t$}}
\psfrag{c1}{{\scriptsize $c_1$}}
\psfrag{ci}{{\scriptsize $c_i$}}
\psfrag{ct}{{\scriptsize $c_t$}}
\psfrag{v1}{{\scriptsize $v_1$}}
\psfrag{v2}{{\scriptsize $v_2$}}
\psfrag{q1}{{\scriptsize $q_1$}}
\psfrag{q2}{{\scriptsize $q_2$}}
\psfrag{q}{{\scriptsize $q$}}
\psfrag{D}{{\scriptsize $D$}}
\includegraphics[scale=0.6]{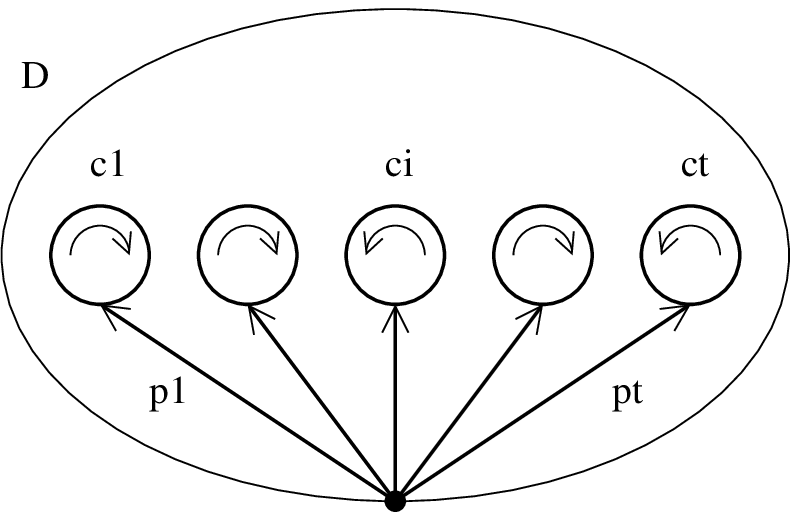}
\end{center}
\caption{}\label{fig:spray0}
\end{subfigure}
\quad
\begin{subfigure}[b]{.3\columnwidth}
\begin{center}
\psfrag{p1}{{\scriptsize $p_1$}}
\psfrag{pi}{{\scriptsize $p_i$}}
\psfrag{pt}{{\scriptsize $p_t$}}
\psfrag{c1}{{\scriptsize $c_1$}}
\psfrag{ci}{{\scriptsize $c_i$}}
\psfrag{ct}{{\scriptsize $c_t$}}
\psfrag{v1}{{\scriptsize $v_1$}}
\psfrag{v2}{{\scriptsize $v_2$}}
\psfrag{q1}{{\scriptsize $q_1$}}
\psfrag{q2}{{\scriptsize $q_2$}}
\psfrag{q}{{\scriptsize $q$}}
\psfrag{D}{{\scriptsize $D$}}
\psfrag{.}{{\scriptsize $\cdots$}}
\includegraphics[scale=0.6]{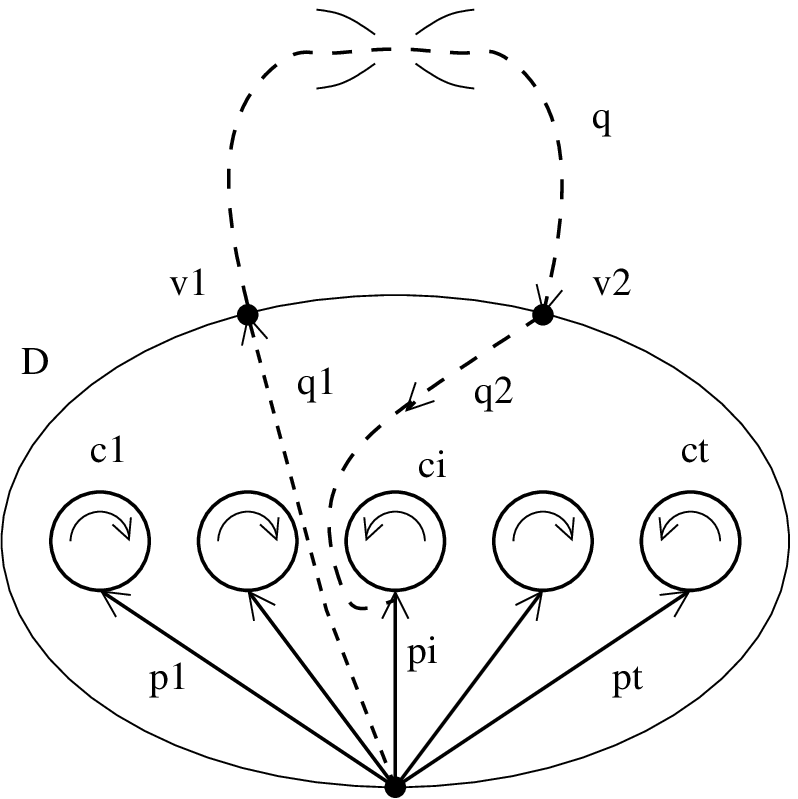}
\end{center}
\caption{}\label{fig:spray1}
\end{subfigure}
\quad
\begin{subfigure}[b]{.3\columnwidth}
\begin{center}
\psfrag{p1}{{\scriptsize $p_1$}}
\psfrag{pt}{{\scriptsize $p_t$}}
\psfrag{c1}{{\scriptsize $c_1$}}
\psfrag{ci}{{\scriptsize $c_i$}}
\psfrag{ct}{{\scriptsize $c_t$}}
\psfrag{v1}{{\scriptsize $v_1$}}
\psfrag{v2}{{\scriptsize $v_2$}}
\psfrag{q1}{{\scriptsize $q_1$}}
\psfrag{q2}{{\scriptsize $q_2$}}
\psfrag{q}{{\scriptsize $q$}}
\psfrag{D}{{\scriptsize $D$}}
\includegraphics[scale=0.6]{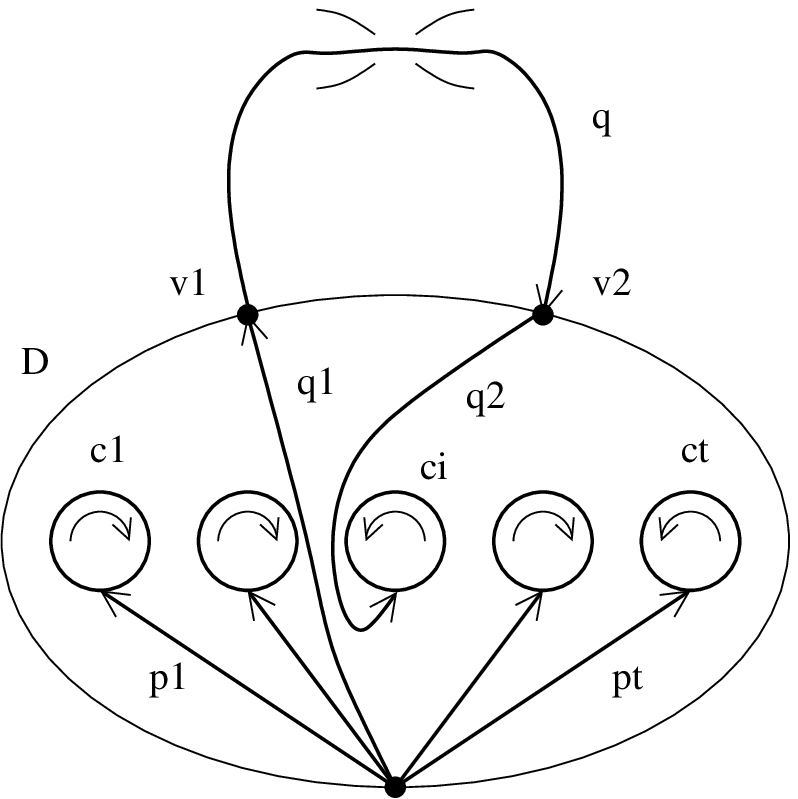}
\end{center}
\caption{}\label{fig:spray2}
\end{subfigure}
\caption{Reversing orientation of a boundary component}\label{fig:spray}
\end{center}
\end{figure}

Suppose that $\e_1=1$ and $i$ is minimal such that $\e_i=-1$. 
Then choose  vertices $v_1$ and $v_2$ on $c$ and paths $q_1$, from $*$ to $v_1$,  and $q_2$, from $v_2$ to  $b_i$, so that 
the path $q_1^{-1}p_iq_2^{-1}$ cuts a disk from $D$, not meeting any boundary component $c_j$ and so that the $q_1$ 
meets $p_j$ only at $*$, for all $j$, and $q_2$ meets $p_i$ only at $b_i$ and meets no other $p_j$; 
as in Figure \ref{fig:spray1}. Now let $q$ be a simple, orientation reversing,  
path  on $\S$ from $v_1$ to $v_2$ such that $q$ meets $D$ only at $v_1$ and $v_2$, as in Figure \ref{fig:spray2}. 
Such a path exists since $\S$ is non-orientable. 
After thickening the $2$-cells meeting 
these paths we may assume that $q$ and $q_i$ are edge paths in $K_1$. 
Replace the path $p_i$ with $p'_i=q_1qq_2$. Now repeat the construction of $D$ using the tree $p_1\cup \cdots \cup p'_i \cup \cdots\cup p_t$ 
instead of $T$. This time the boundary components have labels  $c_1^{\e_1}, \ldots , c_t^{\e_t}$, where 
$\e_1=\cdots =\e_i=1$. Continuing this way, eventually all the $\e_i$ are equal.

Thus, in all cases, there is a closed edge path $c$ in $K_1$ with the required properties. Let $C$ be
 the label of $c$. The disk $\D$ gives  a solution to the equation $\prod_{i=1}^t v_i^{-1}c_iv_i=C$ and 
the surface $\S\backslash\textrm{int}(\D)$ gives a solution to the equation $\prod_{i=1}^g[x_i,y_i]=C^{-1}$, if 
$\S$ is orientable, and $\prod_{i=1}^{2g} z_i^2=C^{-1}$, otherwise. Combining these we obtain a solution to 
\eqref{eq:qgo} or \eqref{eq:qgn} as appropriate.
Thus there exists a solution to $w=1$ if and only if there exists a van Kampen diagram on a surface $\S$, of  genus $g$, 
with $t$ boundary components, labelled (with appropriate orientation) $c_1,\ldots ,c_t$. 
 This motivates the following definition.

\begin{definition} Let $G$ be a group and let $c_1,\ldots ,c_t$ be 
elements of $G$. 
\be
\item The \emph{orientable genus} of $(c_1,\ldots ,c_t)$, denoted $\genus^+_G(c_1,\ldots
,c_t)$, is the
smallest integer $k$ such that there exists a solution to the equation \eqref{eq:qgo} over $G$. 
\item The \emph{non-orientable genus} of $(c_1,\ldots ,c_t)$, 
denoted $\genus^-_G(c_1,\ldots
,c_t)$, is the
smallest 
element $k\in \ZD$ such that  
 there exists a solution to the equation \eqref{eq:qgo} over $G$. 
\ee
In all cases, if there is no $k$ satisfying
the given conditions then the corresponding genus is defined 
to be infinite. 
\end{definition}
When no ambiguity arises, or we wish to make statements covering both orientable and non-orientable
genus we use $\genus_G(c_1,\ldots ,c_t)$  instead of  $\genus^+_G(c_1,\ldots ,c_t)$ or  $\genus^-_G(c_1,\ldots ,c_t)$.
\eqref{eq:qgo} holds, and similarly for $\genus_G^-$. 
\subsection{Wicks forms}\label{sec:wicks}

Consider a quadratic  word $U$ over   $\mathcal{A}$ and the surface $\S_U$ obtained as above. 
 The signature $\s$ and orientability $o$ of letters of $U$ induce a signature 
and orientability on the graph $\G_U$: namely if $\l(e)=A\in \supp(U)$ then
$\s(e)=\s(A)$ and  $o(e)=-\e\d$, where $\s(A)=(\e,\d)$.  
The edge $e$ is \emph{alternating} if $o(e)=1$.
 To obtain a van Kampen diagram over $F(\cA)$ (as distinct from a van Kampen diagram over $H$ as above) on $\S_U$  with
boundary label $U$, 
 thicken the $2$-cell $D$ to obtain a disk with an interior region $D_0$ labelled $U$, as in 
Figure \ref{gammaU}. 
 Remove the interior of $D_0$ from $\S_U$ leaving  a van Kampen diagram,
$\D_U$, 
over $F(\cA)$ on a surface of genus $g$, with one boundary
component labelled $U$.  
It follows 
from \cite{Culler81} 
that $\genus_{F(\cA)}(U)=g$. 

\begin{example}\label{ex:wicks}
Suppose that we have the quadratic orientable word $U=ABCA^{-1}B^{-1}C^{-1}$ we
construct $\Gamma_U$  and $\D_U$ as shown in Figure \ref{gammaU}.
\end{example}
\begin{figure}[htp]
\begin{center}
\psfrag{A}{{\scriptsize $A $}}
\psfrag{B}{{\scriptsize $B$}}
\psfrag{C}{{\scriptsize $C$}}
\psfrag{X1}{{\scriptsize $X_1$}}
\psfrag{X2}{{\scriptsize $X_2$}}
\psfrag{U}{{\scriptsize $\Gamma_U $}}
\psfrag{D}{{\scriptsize $\D_U $}}
\psfrag{SU}{{\scriptsize $\S_U $}}
\includegraphics[scale=0.55]{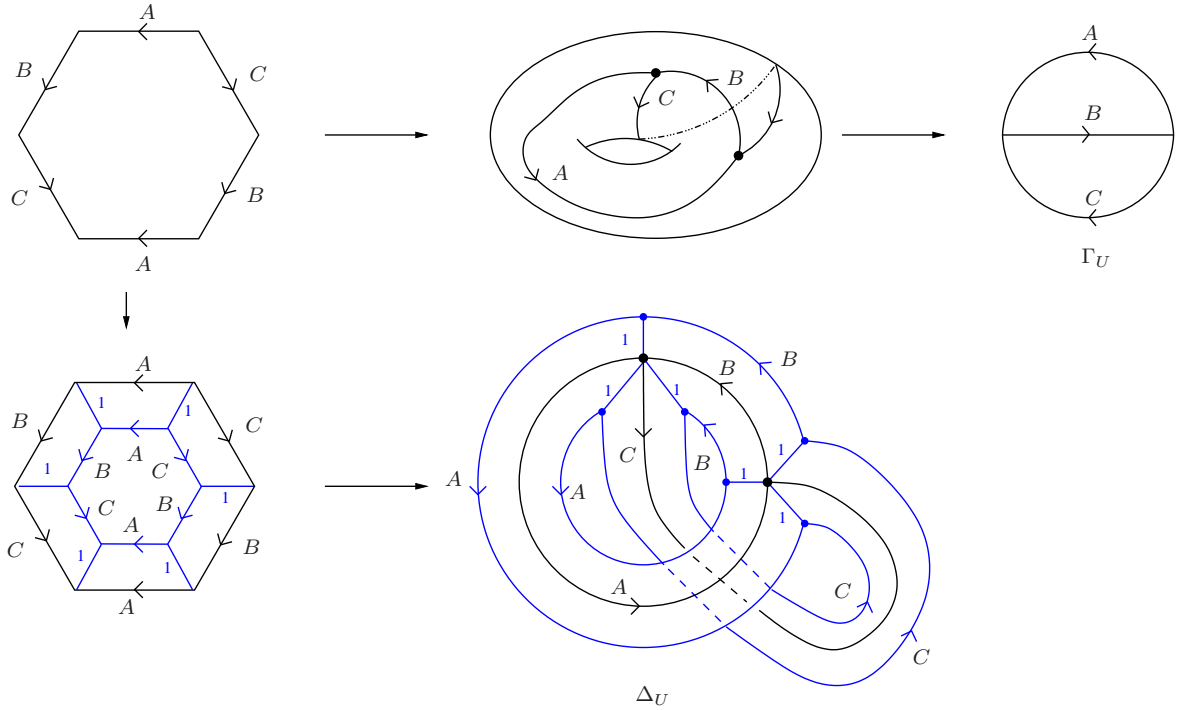}
\caption{Constructing $\Gamma_U$ and $\D_U$}\label{gammaU}
\end{center}
\end{figure}

This construction may be carried out in the same way starting with
a  quadratic  $t$-tuple of words $U_1,\ldots, U_t$. This
time let $U$ be the quadratic word 
$U=U_1\cdots, U_t$, label the boundary of $D$ with $U$, as before, 
define  $\S_{U_1,\ldots, U_t}$ to be $\S_U$ and the associated 
graph $\G_{U_1,\ldots,U_t}$ to 
be $\G_U$. To  construct a van Kampen diagram on a surface 
 with $t$ boundary components labelled  $U_1$, ..., $U_t$ 
take $t$ letters, $X_1$, ..., $X_t$, of $\cA\bs(\cup_{i=1}^t\supp(U_i))$, and 
let $V=X_1^{-1}U_1X_1\cdots X_t^{-1}U_tX_t$. As before let $D_0$ be 
a disk embedded in $\S_{U_1,\ldots, U_t}$ and not meeting $\G_{U_1,\ldots,U_t}$. 
Label the boundary 
of $D_0$ with $V$ and join the ends of corresponding
letters of $U_i$ on the boundaries of $D$ and $D_0$ with disjoint, 
properly embedded arcs, labelled $1$. 
Remove the interior of $D_0$ and identify edges labelled $X_i$ with each
other, respecting orientation.  (See Example \ref{ex:wicks1}.) The result
is the required van Kampen diagram $\D_{U_1,\ldots, U_t}$.
  It follows once more,
from \cite{ols89}[Sections 2.3 and 2.4],  
that $\genus_{F(\cA)}(U_1,\ldots, U_t)=\genus(\S_{U_1,\ldots, U_t})$. 

\begin{example}\label{ex:wicks1}
Suppose that we have the quadratic orientable pair of  words $U_1,U_2$, where 
$U_1=AB$, and $U_2=CA^{-1}B^{-1}C^{-1}$. Then $U=ABCA^{-1}B^{-1}C^{-1}$ as in 
Example \ref{ex:wicks}.   We
construct $\S_{U_1,U_2}=\S_{U}$ and $\Gamma_{U_1,U_2}=\G_U$ as in 
shown in Figure \ref{gammaU}. The van Kampen diagram $\D_{U_1,U_2}$ is 
obtained by identifying the edges around the outside hexagon on the left
of Figure \ref{gammaU1}, as well as those labelled $X_i$ on the inside
decagon. 
 \begin{figure}[htp]
\begin{center}
\psfrag{A}{{\scriptsize $A $}}
\psfrag{B}{{\scriptsize $B$}}
\psfrag{C}{{\scriptsize $C$}}
\psfrag{U}{{\scriptsize $\Gamma_U $}}
\psfrag{DU}{{\scriptsize $\D_U $}}
\psfrag{SU}{{\scriptsize $\S_U $}}
\psfrag{X1}{{\scriptsize $X_1$}}
\psfrag{X2}{{\scriptsize $X_2$}}
\includegraphics[scale=0.75]{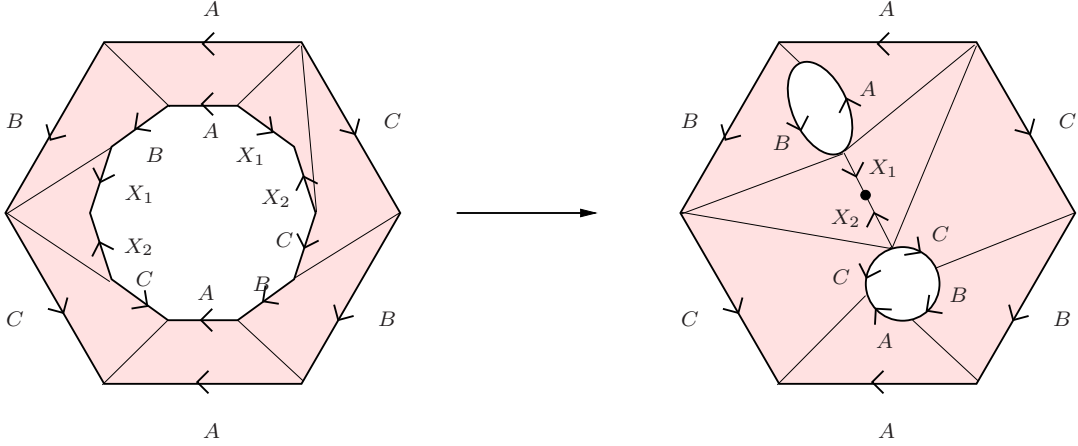}
\caption{$U_1=AB$, $U_2=CA^{-1}B^{-1}C^{-1}$ and $\D_{U_1,U_2}$ is obtained from the right hand diagram by
identifying edges on the outside boundary according to labels.}\label{gammaU1}
\end{center}
\end{figure}
The result is a van Kampen diagram on the torus, with two boundary
components, labelled $AB$ and $CA^{-1}B^{-1}C^{-1}$.
\end{example}

Let $\Gamma$ be any directed, connected graph with signature $\s$.
An \emph{Eulerian circuit} in $\G$ is 
 a
circuit $C$ such that, 
\begin{itemize}
\item $C$ 
traverses every edge of $\G$ exactly twice; and 
\item 
for all edges $e$, if $\s(e)=(\e,\d)$ then (some cyclic permutation of) the 
  edge sequence of $C$  contains the
subsequence $e^\e, e^\d$. 
\end{itemize}
By construction the graph $\G_U$ of a quadratic word $U$ contains an
Eulerian circuit, namely the word $U$, and a similar statement applies
to a quadratic tuple of words.

~\\

Let $\Gamma$ be any graph with an
 Eulerian circuit, $T$
say. Write $T$ around the boundary of a disk and identify the edges,
respecting orientation, to obtain a surface $S$. Then the Euler
characteristic $\chi(S)$ of $S$ is given by the formula $v-e+1$, where $v$
and $e$ are the  number of vertices and edges, respectively, of 
$\Gamma$ and $\chi(S)$ is equal to $2-2\genus(S)$. We 
define $\genus(\Gamma)$, the  \emph{genus} of $\G$, 
to be equal to $\genus(S)$: that is
\begin{equation}\label{gengra}
\genus(\Gamma)=\frac{1-v+e}{2}.
\end{equation}

As noted above, from \cite{ols89},  if $U_1,\ldots, U_t$ is a quadratic $t$-tuple of words
 over $F(\cA)$ then $\genus_{F(\cA)}(U_1,\ldots, U_t)=\genus(\S_{U_1,\ldots, U_t})$, so in addition 
$\genus_{F(\cA)}(U_1,\ldots, U_t)=\genus(\G_{U_1,\ldots, U_t})$.

\begin{definition}\label{def:redun}
 A  quadratic 
$t$-tuple $w_1,\ldots, w_t$ of  words  is said to be \emph{redundant}
if there are $x$ and $y$ in $(\cup_{i=1}^t\supp(w_i))^{\pm 1}$ and 
disjoint (no overlap) subwords $s_1$ and $s_2$ of the cyclic words
$w_i$ and $w_j$, $1\le i\le j\le t$,  
such that $s_1=xy$ and $s_2=(xy)^{\pm 1}$. 
Otherwise $w_1,\ldots, w_t$ is said to be
\emph{irredundant}.
\end{definition}

\begin{definition}[Wicks Form] Let $W$ be a  quadratic 
 word in $F(\cA)$.  
Then $W$ is called a \emph{Wicks
    form over $\cA$} if  the following conditions hold.
\begin{enumerate}[(i)]
\item\label{it:wicks1} $W$ is freely cyclically reduced and
\item\label{it:wicks2} $W$ is irredundant.
\end{enumerate}
\end{definition}

Let $H$ be a group with presentation $\la X| S\ra$. A 
(monoid) homomorphism
$\lambda$ 
from $(\cA\cup \cA^{-1})^*$ to $F(X)$, 
such that $\lambda(a^{-1})=_H \lambda(a)^{-1}$, is called a \emph{labelling function}. 
If $H$ is the free group on $X$ and $h$ is an element of $H$, of genus $g$,
 then, as shown in \cite{Culler81},   there is a Wicks form $W=A_1\ldots A_n$ over $\cA$, 
and a labelling function $\psi$, such that $h$ is conjugate to $\psi(W)$ and
$\psi(W)=\psi(A_1)\cdots \psi(A_n)$ is reduced as written. Indeed Culler 
(\emph{loc. cit.}) 
uses this fact to give a description of all solutions of either
equation \eqref{eq:qgo} or \eqref{eq:qgn}, when $t=1$ and 
$G=F(X)$.

If $W$ is a Wicks form of genus not equal to $1/2$ then 
 the graph $\Gamma _W$, associated to $W$,  
contains no vertices of degree $1$ or $2$ 
(as if it did then either condition  \ref{it:wicks1} or \ref{it:wicks2}, in the
definition of a Wicks form, would be violated).  There is only one 
Wicks form of genus $1/2$, namely $W=X^2$, and it has graph $\G_W$
consisting of a single vertex and a single edge, lying on the projective plane 
$\S_W$. 
Consequently,  there is a bound on the maximal length of a Wicks form 
of genus $n$, given explicitly in the following lemma of M.~Culler \cite{Culler81}.
\begin{lemma}[{\cite[Theorem 3.1]{Culler81}}]\label{cull}
Let $V$ be a Wicks form over $\mathcal{A}$ such
that $\genus_{F(\mathcal{A})}(V)=m>\frac{1}{2}$. Then the length of $V$ is at most $12m-6$.
\end{lemma}

\begin{lemma}\label{lem:wicks}
Let $h_1,\ldots ,h_t$ be  elements of $F(X)$, $t\ge 1$. Then $\genus_H(h_1,\ldots ,h_t)\le m$ if and only if there exists
 a Wicks form $V$ over $\cA$, of genus $m$, and a labelling function
$\psi$ such that some product $r_1h_1r_1^{-1}\cdots r_th_tr_t^{-1}$, with $r_i\in F(X)$,  is equal in $H$ to $\psi(V)$. 
\end{lemma}
\begin{proof}
Suppose that $\genus_H(h_1,\ldots ,h_t)=g$. Then there exist elements $r_j\in H$, $j=1,\ldots ,t$, and 
either (i) elements 
$a_i, b_i\in H$, $i=1,\ldots, g$, 
such that $\prod_{j=1}^tr_jh_jr_j^{-1}=\prod_{i=1}^g[a_i,b_i]$; or (ii) elements $c_i\in H$,
such that $\prod_{j=1}^tr_jh_jr_j^{-1}=\prod_{i=1}^{2g} c_i^2$. Let $A_i, B_i, C_i\in \cA$.
In case (i), for 
all integers $m\ge g$, 
$V=\prod_{i=1}^m[A_i,B_i]$ is a Wicks form, of genus $m$. Define 
$\psi$ by $\psi(A_i)=a_i$, $\psi(B_i)=b_i$, for $i=1,\ldots, g$,  
and $\psi(X)=1$, if $X\in \cA$, $X\neq A_i$ or
$B_i$, with $1\le i\le g$. Then $\psi(V)=\prod_{j=1}^tr_jh_jr_j^{-1}$, as required.
In case (ii), for all integers  $m\ge 2g$, 
$V=\prod_{i=1}^{m}C_i^2$ is a Wicks form, of genus $m/2$. Define 
$\psi$ by $\psi(C_i)=c_i$, for $i=1,\ldots, 2g$,  
and $\psi(X)=1$, if $X\in \cA$, $X\neq C_i$, with  $1\le i\le 2g$. 
Again $\psi(V)=\prod_{j=1}^tr_jh_jr_j^{-1}$, as required.

Conversely, suppose that $V$ is a Wicks form of genus $m$, that 
$\psi$ is a labelling function  and that $\prod_{j=1}^tr_jh_jr_j^{-1}$ is equal in $H$ to
$\psi(V)$, for some $r_j\in H$. Without loss of generality, 
we may assume that $h_j$ is freely cyclically reduced. As $\prod_{j=1}^tr_jh_jr_j^{-1}$ is
equal to $\psi(V)$ there is a van Kampen diagram $D_1$, 
over $H$, on a surface of genus $0$, with $t+1$ boundary components $\b_1,\ldots ,\b_t, \b$, with $\b_j$ labelled 
$h_j$ and    $\b$, 
labelled $\psi(V)$.
Moreover, as in Section \ref{sec:gandvk}, there is a  
van Kampen diagram $\D_V$ over $F(\cA)$, on a surface 
$\S_V$, of genus $m$, with one boundary component $\b^\prime$, 
labelled $V$. 
Relabelling $\D_V$, using the labelling function $\psi$, 
we may construct a van Kampen diagram over $F(X)$ with boundary
label $\psi(V)$. In more detail:
 every non-trivial region $R$ of $\D_V$ has boundary label 
a cyclic word $a\cdot 1\cdot  a^{-1}\cdot 1$, for some $a\in \cA$. Divide the
edge labelled $a$ into $|\psi(a)|$ segments and label the resulting edge
path $\psi(a)$. The region $R$ can now also be  divided into $\psi(a)$
parallel regions, each with boundary label  $x\cdot 1 \cdot x^{-1}\cdot 1$,
for some $x\in \supp(\psi(a))$. Repeating over all edges and regions results
in a van Kampen diagram, over $F(X)$, on $\S_V$, with boundary 
component $\b^\prime$, labelled  
$\psi(V)$. 
As this is a diagram over $F(X)$ it is
 {\it a fortiori} a diagram  over $H$.  
 Thus, attaching $D_1$ to $D_2$ by gluing the boundary components $\b$ and $\b^\prime$
together, using their labels, we obtain a van Kampen diagram over 
$H$, on a surface of genus $m$, with boundary label $h$. 
Hence $\genus_H(h_1,\ldots ,h_t)\le m$. 
\end{proof}

\subsection{Extension of a quadratic  word over a Hyperbolic group $H$}\label{sec:ext}
Let $U$ be a non-empty quadratic  word of genus $k$ (not necessarily reduced or irredundant) over the  countably infinite 
alphabet
$\mathcal{A}$ and let $\Gamma _U$ be its associated 
graph on the surface $\S_U$. 
Then $\G_U$ has genus $k$ and signature $\s$,  
and its edges are labelled  with
letters of $U$. It is notationally convenient 
to identify a directed 
edge $e$ of $\G_U$ with its label $\l(A)\in \cA^{\pm 1}$  (and
 $e^{-1}$ with $\l(A)^{-1}$) and we shall 
do so from now on. Thus  
 $U$ itself is an Eulerian
circuit of $\G_U$. Let $H=\la X|S\ra$ be a hyperbolic group in
which geodesic triangles in $\Gamma _X(H)$ are $\delta$-thin. We 
now describe a procedure which, when applied to $\Gamma_U$, results 
in a new graph 
called an \emph{extension} of the quadratic  word $U$ over $H$. This
will be carried out in three steps, but first we need to set up 
notation. 

Roughly speaking, we define the star of a vertex $v$ of $\G_U$ to be a sufficiently small closed 
disk, on $\S_U$, containing $v$. 
To make this precise,  let 
$K$ be the second barycentric subdivision of   
the cell complex consisting of $\G_U$ and $D$ on
$\S_U$. The \emph{star} of $v$ is the convex-hull of the star of $v$
as a $0$-cell of the simplicial complex $K$ (the subcomplex of $K$
consisting 
of all simplices meeting $v$, together with their faces).  
As $\S_U$ is a surface the star of $v$ is a disk and we define
the \emph{link} of $v$ to be the boundary of this disk. 
%
  Then we may view $\lk(v)$  as a cycle graph, with
vertices the points where it meets edges of $\G_U$ (necessarily incident to $v$) 
and edges the closures of connected components  
of $\lk(v)\backslash E(\G_U)$. Let $x$ be a point of intersection
of $\lk(v)$ and an edge $e$ of $\G_U$. If $v=\t(e)$ (and $[x,v]$ is contained
in the star of $v$) then we denote the vertex 
$x$ of $\lk(v)$ by $e$ and if $v=\i(e)$ we denote $x$ 
by $e^{-1}$. (When $e$ is a loop then both $e$ and $e^{-1}$ are
vertices of $\lk(v)$.) Thus the union, over all vertices $v$ of $\G_U$, 
of the graphs $\lk(v)$ is isomorphic to the star graph of the word $U$ 
(the graph with  vertices $A^{\pm 1}$, for all $A\in \supp(U)$, 
and an edge joining 
$A^\e$ to $B^{-\d}$ for each subword $A^\e B^\d$ of the cyclic word $U$). 
  
 Suppose that a vertex $v$ of $\G_U$ has degree $d$. Fix an orientation of $v$ and 
renumber the edges 
incident to $v$ (temporarily, for the current purposes) so that the link of $v$, read
according to the chosen orientation,  has vertex
sequence $e_1^{\e_1}, e_2^{\e_2},\ldots ,e_d^{\e_d}$. 
Then, for all $i$, the cyclic word $U$ contains the 
subword $e_i^{\e_i}e_{i+1}^{-\e_{i+1}}$ or $e_{i+1}^{\e_{i+1}}e_i^{-\e_i}$. 
\begin{definition}\label{def:vinc}
Let $v$ be a vertex of $\G_U$, with a chosen orientation, and assume that the 
link of $v$ has vertex
sequence $e_1^{\e_1}, e_2^{\e_2},\ldots ,e_d^{\e_d}$. 
Define the \emph{incidence sequence of} $v$ to be $O_1(v), \ldots ,O_d(v)$, 
where 
\[
O_1(v)=
\begin{cases}
1 & \textrm{ if $U$ contains $e_1^{\e_1}e_2^{-\e_2}$}\\
-1 &\textrm{ otherwise}
\end{cases}
\]
and 
\[O_{i+1}(v)=O_i(v) o(e_{i+1}), 
\]
for 
$i=2,\ldots d$ (see Definition \ref{def:qw}). 
\end{definition}
When the vertex $v$ is fixed we write $O_q$ instead of $O_q(v)$.  
Thus $U$ contains $(e_i^{\e_i}e_{i+1}^{-\e_{i+1}} )^{O_i}$, for 
$i=1,\ldots ,d$. It follows that $v$ is incident to an even number
of non-alternating edges. Hence, 
 once the orientation of $v$ has been fixed and the 
first edge $e_1$ chosen, the incidence sequence is  both 
well-defined 
  and uniquely determined.
 (Note that, in the case $U$ is orientable, $o(e_i)=1$, for all $i$, so 
$O_1=\cdots =O_d$ and the orientation of $v$ may be chosen so that 
$O_i=1$, for all $i$.)

 To keep track of the orientation of letters we define 
 book-keeping functions: $\mu$ and $\nu$, from $\{\pm 1\}$ to $\{1,2\}$, and 
$l$ and $r$, from $ \{\pm 1\}^2$ to $\{\pm 1\}$, as follows.
\[
\begin{array}{ll}
\mu(1)=1, & \mu(-1)=2\\
\nu(1)=2, & \nu(-1)=1\\
l(x,y)=\nu(xy), & r(x,y)=\mu(xy).
\end{array}
\]

\begin{definition}\label{def:vori}
Assume $v$ has  link with vertex sequence $e_1^{\e_1}, e_2^{\e_2},\ldots ,e_d^{\e_d}$, as in the previous definition. Define $\mu_q=\mu_q(v)$, 
$\nu_q=\nu_q(v)$, $l_q=l_q(v)$ and $r_q=r_q(v)$ as follows, 
for $q=1,\ldots ,d$.
\be[(i)]
\item
If $o(e_q)=1$ then 
\[
\mu_q=\mu(\e_q), \nu_q=\nu(\e_q), l_q=l(O_q,\e_q) \textrm{ and } 
r_q=r(O_q,\e_q).  
\]
\item\label{it:vori2}
Assume that $o(e_q)=-1$ and  $\s(e_q)=(\e,\e)$. In this case $O_{q-1}=-O_q$. 
\be
\item 
 If the first occurrence in $U$ of the 
letter $e_q^\e$ is the one occurring in the subword $(e_{q-1}^{\e_{q-1}}
e_q^{-\e_q})^{-O_q}$ then $l_q=1$ and $r_q=2$. 
\item
Otherwise the first occurrence in $U$ of the 
letter $e_q^\e$ is the one occurring in the subword $(e_{q}^{\e_{q}}
e_{q+1}^{-\e_{q+1}})^{O_q}$ and $l_q=2$, $r_q=1$. 
\item \label{it:voric}
If $O_q=1$ then 
$\mu_q=r_q$ and 
$\nu_q=l_q$. 
\item  \label{it:vorid}
Otherwise $\mu_q=l_q$ and $\nu_q=r_q$.  
\ee
(If $d=1$ then $e_1$ is necessarily alternating; $o(e_1)=-1$, and $O_1$ is $1$ if and only if the chosen orientation is that 
induced by reading  $U$.  
In  this case $l_1=\mu(O_1)$, $r_1=\nu(O_1)$, with $\mu_1$ and $\nu_1$ as in 
\ref{it:voric} and \ref{it:vorid} above). 
\ee 
\end{definition}

\begin{example}\label{ex:vori}
Let $U=ABA^{-1}B$, so $\G_U$ is a graph with one vertex and two edges, labelled $A$ and $B$, on the
Klein bottle. We have 
 $\s(A)=(1,-1)$, 
$\s(B)=(1,1)$, $o(A)=1$ and $o(B)=-1$. The star of the vertex $v$ intersects $\G_U$ as shown in Figure \ref{fig:vori}(\subref{fig:voria}). 
Orienting anti-clockwise, $v$ has vertex sequence $A^{-1}$, $B^{-1}$, $A$, $B$. 
\begin{figure}[htp]
\begin{center}
\psfrag{A1}{{\scriptsize $A_1$}}
\psfrag{A2}{{\scriptsize $A_2$}}
\psfrag{B1}{{\scriptsize $B_1$}}
\psfrag{B2}{{\scriptsize $B_2$}}
\psfrag{C1}{{\scriptsize $C_1$}}
\psfrag{C2}{{\scriptsize $C_2$}}
\psfrag{A}{{\scriptsize $A$}}
\psfrag{B}{{\scriptsize $B$}}
\psfrag{C}{{\scriptsize $C$}}
\psfrag{v}{{\scriptsize $v$}}
\begin{subfigure}[b]{.3\columnwidth}
\begin{center}
\includegraphics[scale=0.7]{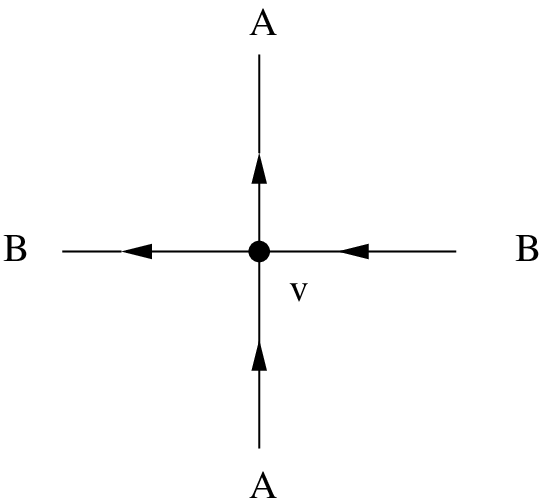}
\caption{}
\label{fig:voria}
\end{center}
\end{subfigure}
\begin{subfigure}[b]{.3\columnwidth}
\begin{center}
\includegraphics[scale=1]{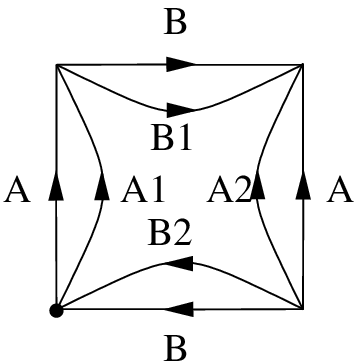}
\caption{}
\label{fig:vorib}
\end{center}
\end{subfigure}
\begin{subfigure}[b]{.3\columnwidth}
\begin{center}
\includegraphics[scale=0.7]{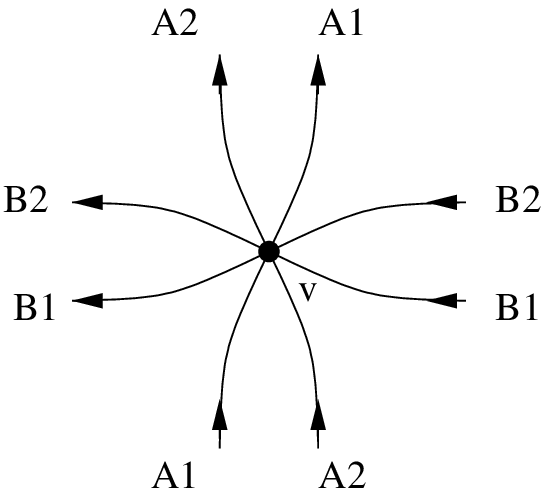}
\caption{}
\label{fig:voric}
\end{center}
\end{subfigure}
\\
\caption{}\label{fig:vori}  
\end{center}
\end{figure}
The following table shows the values of the functions of Definition \ref{def:vori} in this case.
\begin{center}
\begin{tabular}{c|cccc}
$q $ &$1 $&$2 $&$3 $&$4 $\\\hline\\[-.5em]
$e_q $&$A $&$B $&$A $&$B $\\
$\e_q $&$-1 $&$-1 $&$1 $&$1 $\\
$o(e_q) $&$1 $&$-1 $&$1 $&$-1 $\\
$O_q $&$1 $&$-1 $&$-1 $&$1 $\\
$u_q $&$2 $&$2 $&$1 $&$2 $\\
$v_q $&$1 $&$1 $&$2 $&$1 $\\
$l_q $&$1 $&$2 $&$1 $&$1 $\\
$r_q $&$2 $&$1 $&$2 $&$2 $
\end{tabular}
\end{center}
\end{example}

For non-orientable words, in order to construct an extension of $U$ we
first fix a particular representative $U$ of $[U]$. For each 
vertex of $\G_U$ we 
also fix an orientation and initial incident edge, so that the incidence
sequence is defined, for all vertices.  
\begin{enumerate}[label=\textbf{Step \arabic*.},leftmargin=0cm,itemindent=1.6cm,labelwidth=\itemindent,labelsep=0cm,align=left]
\item 
Let $e$ be a directed edge  in $\Gamma _U$ with $u=\iota(e)$
  and $v=\tau(e)$ (note that $u$ may be the same vertex as $v$). In 
this step we  replace $e$ by
  two new edges $e_1$ and $e_2$, where 
$e_1$, $e_2$ are in $(\mathcal{A}\bs\supp(U))$, 
  such that $\iota(e_1)=\iota(e_2)=u$ and $\tau(e_1)=\tau(e_2)=v$; as 
follows. 
As in Section \ref{sec:gandvk}, consider a disk $\D$ with
boundary divided into $|U|$ consecutive, directed intervals, $I_1, \ldots, I_{|U|}$,
 labelled with elements of $\supp(U)$, such that, read in
 a clockwise direction, from the point $I_1\cap I_{|U|}$,  the boundary label of $\D$ is $U$. 
Let $a_1,\ldots , a_{|U|}$ be properly embedded directed arcs 
in $\D$, such that $\i(a_i)=\i(I_i)$ and $\t(a_i)=\t(I_i)$; 
and the interiors of $a_i$ and
$a_j$ are disjoint, for $i\neq j$. There are precisely two 
indices $i,j$ such that 
$I_i$ and $I_j$, read
 with their given directions, are labelled 
$e$. Assume first that  $o(e)=1$.  In this case one of $I_i$, $I_j$ is
oriented clockwise and the other anti-clockwise. If  
$I_i$ is directed clockwise then the directed 
 arc $a_i$ is labelled $e_1$ and $a_j$ is labelled $e_2$. 
Otherwise $a_i$ is labelled $e_2$ and $a_j$ labelled $e_1$. 
(See Figure \ref{fig:extstep1}(\subref{fig:extstep1a}).)
\begin{figure}[htp]
\begin{center}
\psfrag{U}{{\scriptsize $U$}}
\psfrag{e1}{{\scriptsize $e_1$}}
\psfrag{e2}{{\scriptsize $e_2$}}
\psfrag{e11}{{\scriptsize $e_1^1$}}
\psfrag{e12}{{\scriptsize $e_2^1$}}
\psfrag{e21}{{\scriptsize $e_1^2$}}
\psfrag{e22}{{\scriptsize $e_2^2$}}
\psfrag{ed2}{{\scriptsize $e_2^d$}}
\psfrag{ed1}{{\scriptsize $e_1^d$}}
\psfrag{ed}{{\scriptsize $e^{d}$}}
\begin{subfigure}[b]{.45\columnwidth}
\begin{center}
\psfrag{e}{{\scriptsize $e$}}
\psfrag{e1}{{\scriptsize $e_1$}}
\psfrag{e2}{{\scriptsize $e_2$}}
\includegraphics[scale=0.3]{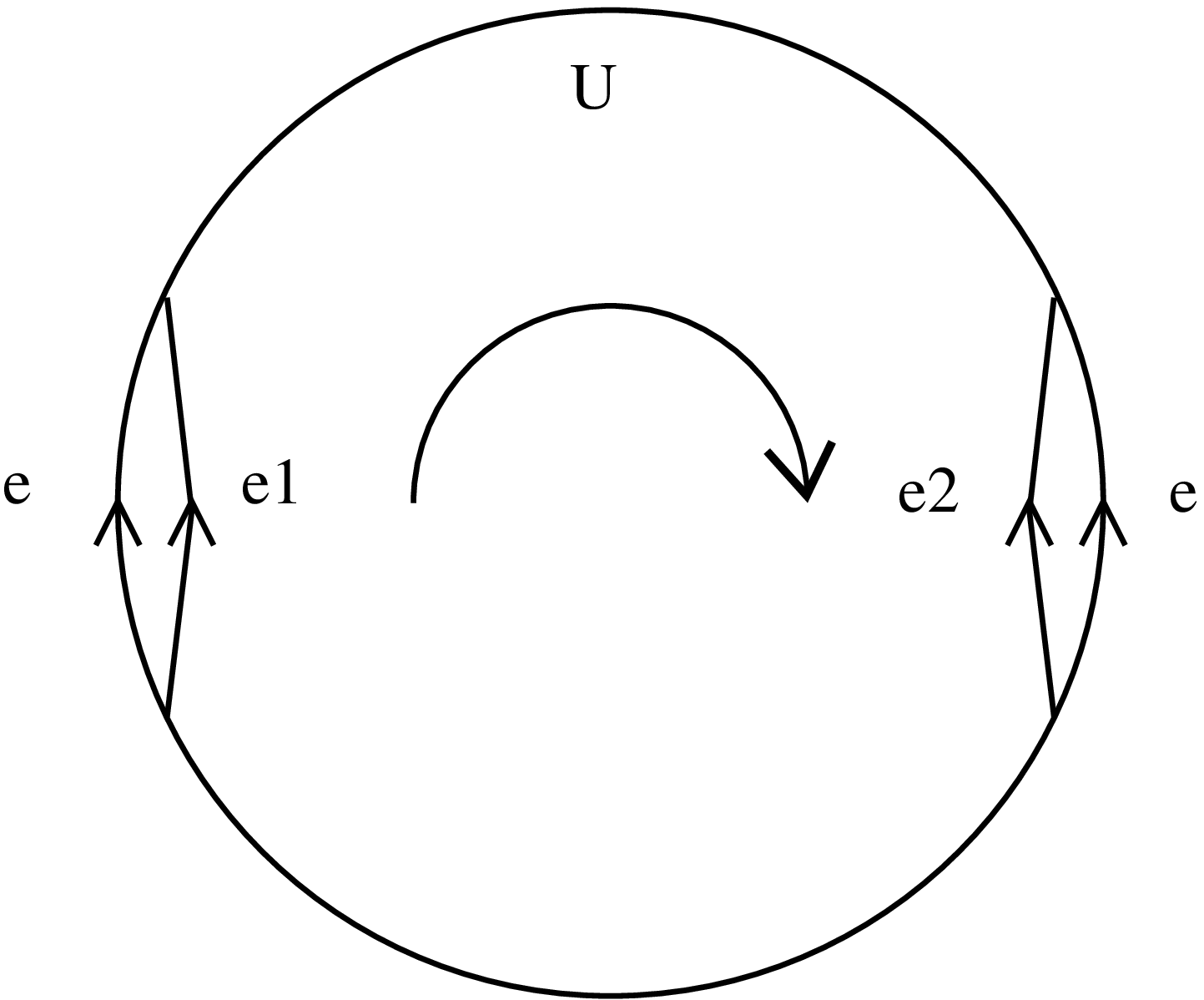}
\caption{}
\label{fig:extstep1a}
\end{center}
\end{subfigure}
\quad\quad
\begin{subfigure}[b]{.45\columnwidth}
\begin{center}
\psfrag{e}{{\scriptsize $e^\e$}}
\psfrag{e1}{{\scriptsize $e_1^\e$}}
\psfrag{e2}{{\scriptsize $e_2^\e$}}
\includegraphics[scale=0.3]{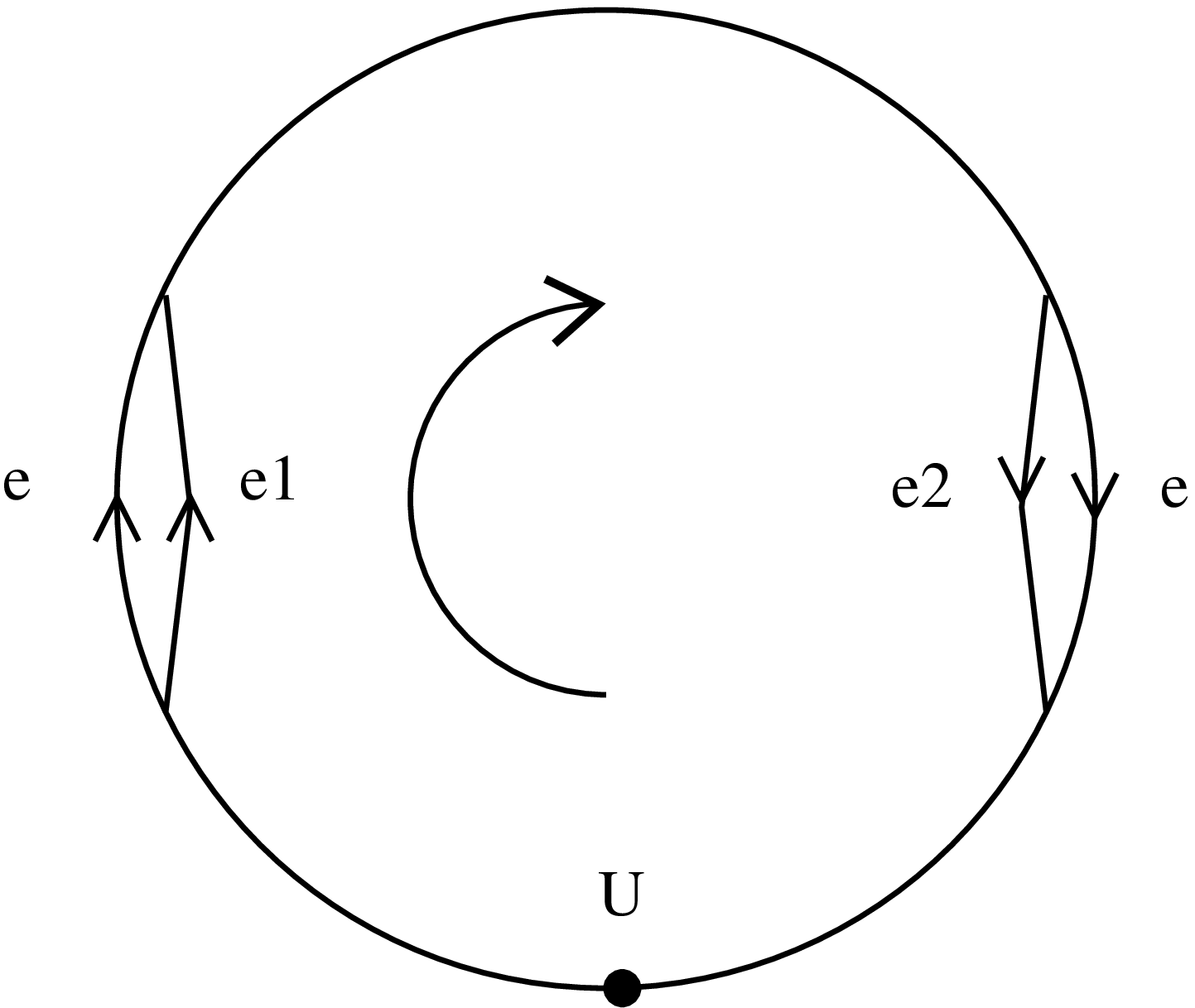}
\caption{}
\label{fig:extstep1b}
\end{center}
\end{subfigure}
\caption{}\label{fig:extstep1}
\end{center}
\end{figure}
 Now suppose that $o(e)=-1$.  If $i<j$ then 
$a_i$ is labelled     
with $e_1$. Otherwise $a_i$ is labelled $e_2$.  
(See Figure \ref{fig:extstep1}(\subref{fig:extstep1b}).)
  We do this for every edge in
  $\Gamma_U$ and form a new graph $\G^\prime_{U}$ on $\S_U$ by deleting
the edges of $\G_U$. 
We denote by $U'$  the word obtained by reading the label of $a_i$ instead
of that of $I_i$, when reading round the boundary of $\D$,  
and note that 
$U'$ is a circuit in $\G^\prime_{U}$. 
For this circuit we have the following lemma. 
\begin{lemma}\label{lem:vind}
Let $v$ be a vertex of $\G_U$ with link which has vertex sequence $e_1^{\e_1}, \ldots ,e_d^{\e_d}$
and incidence sequence $O_1,\ldots , O_d$. Then 
$v$ is a vertex of $\G'_U$ of degree $2d$, 
 the link of $v$ in $\G'_U$ has 
vertex sequence $e_{1,l_1}^{\e_1},e_{1,r_1}^{\e_1} , \ldots 
,e_{d,l_d}^{\e_d},
e_{d,r_d}^{\e_d}$ 
and (the cyclic word) 
$U'$ contains the 
subwords $(e_{q,r_q}^{\e_q}e_{q+1,l_{q+1}}^{-\e_{q+1}})^{O_q}$, 
for $1\le q\le d$. (See Figure \ref{fig:extstep2}(\subref{fig:extstep2a}) and (\subref{fig:extstep2b}).)
\end{lemma} 
\begin{proof} 
 ~This follows from the  definition and the fact that $U$ contains 
$(e_{q}^{\e_q}e_{q+1}^{-\e_{q+1}})^{O_q}$.
\end{proof}
\begin{example}\label{ex:vori-cont}
Applying Step 1 to the word $U$ of example \ref{ex:vori} we construct the graph shown in 
Figure \ref{fig:vori}(\subref{fig:vorib}); 
where edges labelled $A$ and $B$ are identified. Deletion of edges $A$ and $B$ then results in  a graph $\G_{U'}$ 
with one vertex $v$, the link of 
which intersects $\G_{U'}$ as shown in Figure \ref{fig:vori}(\subref{fig:voric}). Here $U'=A_1B_1A_2^{-1}B_2$. 
\end{example}
\item In this step we replace vertices of $\G^\prime_U$ by cycle graphs,
as follows. 
For each vertex $v$ of $\G_U$ choose a subset $\cA_v$ of $\cA$ such that 
\begin{itemize}
\item ~$|\cA_v|=\deg_{\G_U}(v)$,
\item ~$\cA_u\cap \cA_v=\nul$, if $u\neq v$, and 
\item ~$\cA_v\cap(\supp(U)\cup \supp(U^\prime))=\nul$.
\end{itemize}
For each $v$, choose a cyclic word $C_v=c_1\cdots c_d$, where 
$d=\deg_{\G_U}(v)$ and  
$\cA_v=\{c_1,\ldots, c_d\}$. We shall regard $C_v$ either as a cyclic
word or as a directed, labelled, cycle graph, with edges $c_i$, as expedient. 
Now   
let $v$ be  a
  vertex of $\Gamma_U$ of  degree
  $d$, 
 with link which has vertex sequence $e_1^{\e_1}, \ldots ,e_d^{\e_d}$
and incidence sequence $O_1,\ldots , O_d$. 
Then the link of $v$ in $\G'_U$ has 
vertex sequence $e_{1,l_1}^{\e_1},e_{1,r_1}^{\e_1} , \ldots 
,e_{d,l_d}^{\e_d},
e_{d,r_d}^{\e_d}$. 
Now 
assume that the cycle $C_v$ is $c_1\cdots c_d$, with vertices 
$\i(c_i)=v_i$ and $\t(c_i)=v_{i+1}$, subscripts modulo $d$.  
 Remove the vertex $v$ from $\G^\prime_{U}$ and replace it with
$C_v$; according to the following scheme.
 If $d\ge 2$ then 
\begin{eqnarray*}
\tau(e_{d,r_d}^{\e_d})  = \tau(e_{1,l_1}^{\e_1}) & = & v_1\\
\textrm{and}\qquad \tau(e_{q-1,r_{q-1}}^{\e_{q-1}})  =  
\tau(e_{q,l_q}^{\e_q}) & = & v_q,\qquad
\textrm{for }q=2,\ldots ,d.
\end{eqnarray*}
 See Figure  \ref{fig:extstep2}.  
 \begin{figure}[htp]
\begin{center}
\psfrag{w1}{{\scriptsize $c_1$}}
\psfrag{w2}{{\scriptsize $c_2$}}
\psfrag{e11}{{\scriptsize $e_1^1$}}
\psfrag{e12}{{\scriptsize $e_2^1$}}
\psfrag{e21}{{\scriptsize $e_1^2$}}
\psfrag{e22}{{\scriptsize $e_2^2$}}
\psfrag{ed2}{{\scriptsize $e_2^d$}}
\psfrag{ed1}{{\scriptsize $e_1^d$}}
\psfrag{wd}{{\scriptsize $c_{d}$}}
\begin{subfigure}[b]{.45\columnwidth}
\begin{center}
\psfrag{v}{{\scriptsize $v$}}
\psfrag{e1}{{\scriptsize $e_1^{\e_{1}}$}}
\psfrag{e2}{{\scriptsize $e_2^{\e_{2}}$}}
\psfrag{e3}{{\scriptsize $e_3^{\e_{3}}$}}
\psfrag{ed}{{\scriptsize $e_d^{\e_{d}}$}}
\psfrag{ed1}{{\scriptsize $e_{d-1}^{\e_{d-1}}$}}
\includegraphics[scale=1.2]{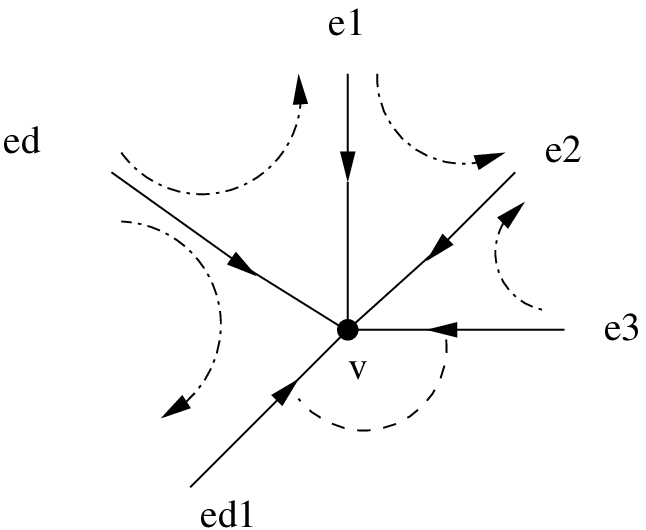}
\caption{}
\label{fig:extstep2a}
\end{center}
\end{subfigure}
\raisebox{9em}{~
$\longrightarrow$~
}
\begin{subfigure}[b]{.45\columnwidth}
\begin{center}
\psfrag{v}{{\scriptsize $v$}}
\psfrag{e1}{{\scriptsize $e_{1,l_1}^{\e_{1}}$}}
\psfrag{e2}{{\scriptsize $e_{2,l_2}^{\e_{2}}$}}
\psfrag{e3}{{\scriptsize $e_{3,l_3}^{\e_{3}}$}}
\psfrag{ed}{{\scriptsize $e_{d,l_d}^{\e_{d}}$}}
\psfrag{ed1}{{\scriptsize $e_{d-1,l_{d-1}}^{\e_{d-1}}$}}
\psfrag{e1r}{{\scriptsize $e_{1,r_1}^{\e_{1}}$}}
\psfrag{e2r}{{\scriptsize $e_{2,r_2}^{\e_{2}}$}}
\psfrag{e3r}{{\scriptsize $e_{3,r_3}^{\e_{3}}$}}
\psfrag{edr}{{\scriptsize $e_{d,r_d}^{\e_{d}}$}}
\psfrag{ed1r}{{\scriptsize $e_{d-1,r_{d-1}}^{\e_{d-1}}$}}
\includegraphics[scale=1.2]{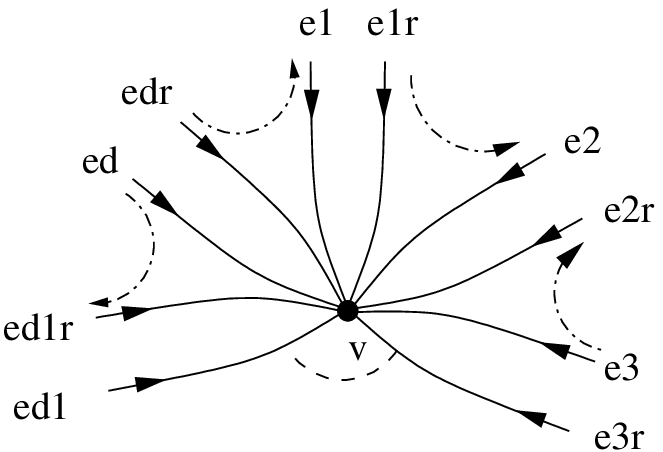}
\caption{}
\label{fig:extstep2b}
\end{center}
\end{subfigure}
\begin{subfigure}[b]{.9\columnwidth}
\begin{center}
\psfrag{v1}{{\scriptsize $v_1$}}
\psfrag{v2}{{\scriptsize $v_2$}}
\psfrag{v3}{{\scriptsize $v_3$}}
\psfrag{v4}{{\scriptsize $v_4$}}
\psfrag{vd1}{{\scriptsize $v_{d-1}$}}
\psfrag{vd}{{\scriptsize $v_d$}}
\psfrag{c1}{{\scriptsize $c_1$}}
\psfrag{c2}{{\scriptsize $c_2$}}
\psfrag{c3}{{\scriptsize $c_3$}}
\psfrag{cd}{{\scriptsize $c_d$}}
\psfrag{e1}{{\scriptsize $e_{1,l_1}^{\e_{1}}$}}
\psfrag{e2}{{\scriptsize $e_{2,l_2}^{\e_{2}}$}}
\psfrag{e3}{{\scriptsize $e_{3,l_3}^{\e_{3}}$}}
\psfrag{ed}{{\scriptsize $e_{d,l_d}^{\e_{d}}$}}
\psfrag{ed1}{{\scriptsize $e_{d-1,l_{d-1}}^{\e_{d-1}}$}}
\psfrag{e1r}{{\scriptsize $e_{1,r_1}^{\e_{1}}$}}
\psfrag{e2r}{{\scriptsize $e_{2,r_2}^{\e_{2}}$}}
\psfrag{e3r}{{\scriptsize $e_{3,r_3}^{\e_{3}}$}}
\psfrag{edr}{{\scriptsize $e_{d,r_d}^{\e_{d}}$}}
\psfrag{ed1r}{{\scriptsize $e_{d-1,r_{d-1}}^{\e_{d-1}}$}}
~\\~\\
\hspace*{9em}$\swarrow$\\
\includegraphics[scale=1.25]{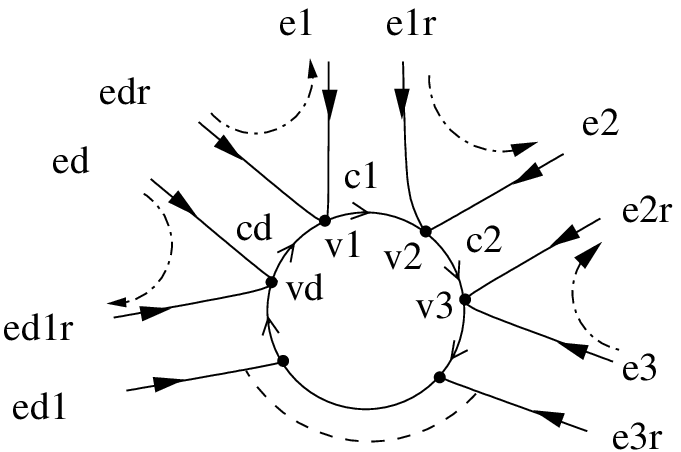}
\caption{}
\label{fig:extstep2c}
\end{center}
\end{subfigure}
\end{center}
\caption{$o(e_1)=1$, $o(e_2)=o(e_d)=-1$, $O_1=1$}
\label{fig:extstep2}
\end{figure} 
If $v$ is a vertex of $\G_U$ of degree one, then the
edge incident to $v$ must be alternating  and $U$ must
contain the subword $e^\e e^{-\e}$, $\e=\pm 1$. In this case   add a
loop  with the label $c_1$ to $v$, as in  
 Figure \ref{extstep2b-1}. On the surface $\S_U$ we require that
this loop lies in the disk which is bounded by $e_1$ and $e_2$ and  which
meets no other edge of $\G^\prime_U$, as illustrated. 
 \begin{figure}[htp]
\begin{center}
\psfrag{w}{{\scriptsize $c_1$}}
\psfrag{e11}{{\scriptsize $e_{\mu(\e)}^\e$}}
\psfrag{e12}{{\scriptsize $e_{\nu(\e)}^\e$}}
\includegraphics[scale=0.5]{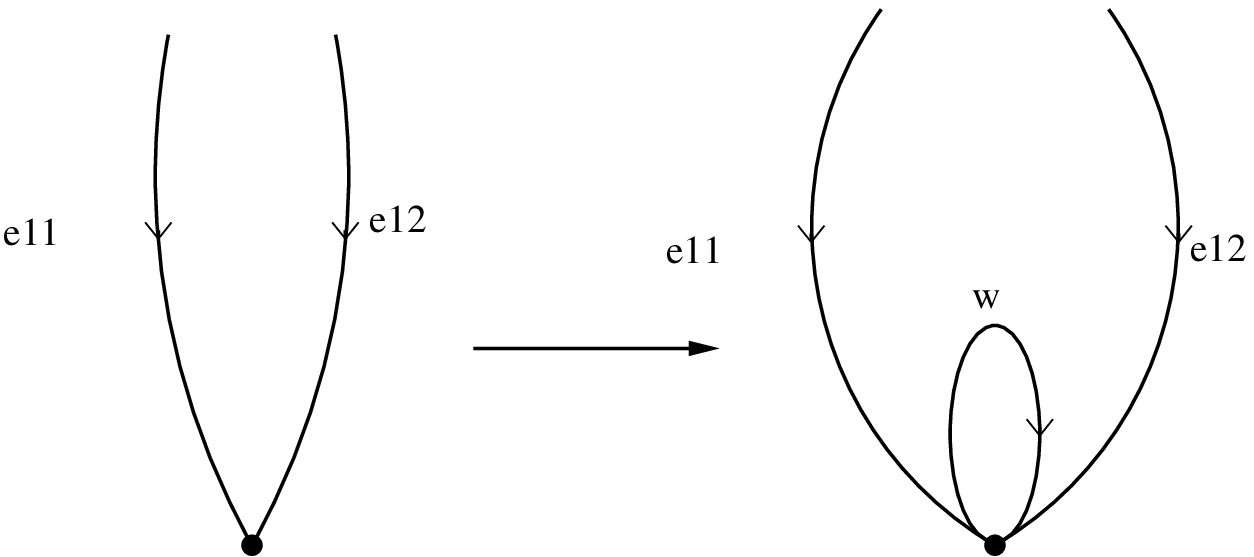}
\\
\vspace{0.38cm}                                                                  
\refstepcounter{figure}\label{extstep2b-1}                                             
Figure \thefigure                  
\end{center}
\end{figure} 
In all cases we say that $v$ has been \emph{extended by} $C_v$. 
After each vertex has been extended 
we
obtain a  new graph which we shall call $\Gamma^{\prime\prime}_{U}$.  We
can still read the closed path $U'$ in this graph. In fact it is now a Hamiltonian
cycle. We shall call $U'$ the \emph{Hamiltonian cycle associated with $U$}.
\item In this step we choose a labelling function for the new graph, which
allows us (eventually) to regard it as a van Kampen diagram over $H$.  
Consider a directed edge $e$ of the original graph $\Gamma _U$. In Step
  1  this was 
  replaced by a pair of edges $(e_1,e_2)$. Then in Step 2 we extended the end
  points of these edges to a  subgraph of $\Gamma^{\prime\prime}
  _{U}$ of the form shown in Figure \ref{fig:extstep3}, 
 \begin{figure}[htp]
\begin{center}
\psfrag{x}{{\scriptsize $x$}}
\psfrag{y}{{\scriptsize $y$}}
\psfrag{e1}{{\scriptsize $e_1$}}
\psfrag{e2}{{\scriptsize $e_2$}}
\includegraphics[scale=0.5]{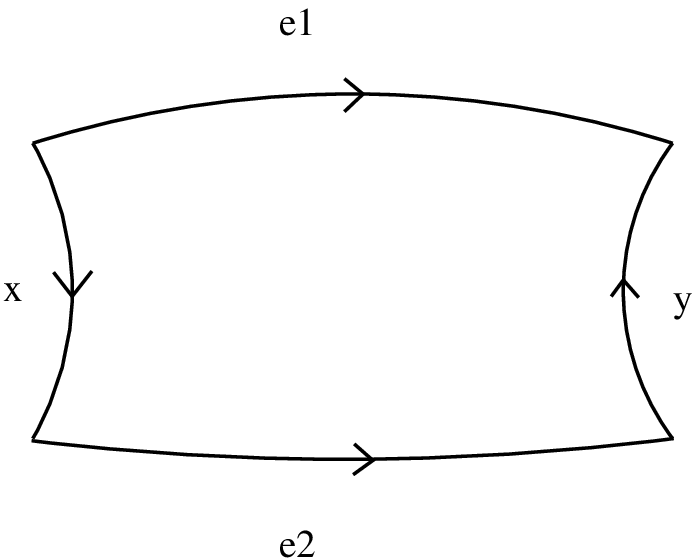}
\\
\vspace{0.38cm}                                                                  
\refstepcounter{figure}\label{fig:extstep3}                                             
Figure \thefigure                  
\end{center}
\end{figure}   
where $x$ and $y$ are 
letters occurring in $C_{\i(e)}$ and $C_{\t(e)}$, respectively. 
Let $\psi:(\cA\cup\cA^{-1})^*\maps F(X)$ be a labelling function such that 
\be[labelsep=.5em]
\item $\psi(a)$ is $H$-minimal, for all $a\in \cA$,
\item $\psi(e_1)=_H\psi(x)\psi(e_2)\psi(y)$, for all pairs of 
directed edges
$(e_1,e_2)$ corresponding to an edge $e$ of $\G_U$, with $x$ and $y$ as in Figure \ref{fig:extstep3}, and 
\item 
if $U^\prime=a_1\cdots a_r$ is the Hamiltonian cycle associated to $U$,
where  $a_i\in \cA^{\pm 1}$, then $\psi(a_1)\cdots \psi(a_r)$ is freely
 cyclically reduced (as written). 
\ee
Then the resulting pair $(\G^{\prime\prime}_U,\psi)$  
is called an {\emph{extension
    of $U$ over the group $H$}}. 
\end{enumerate}
\subsection{Genus and length of an extension of $U$ over $H$}\label{sec:genlen}
Suppose that we have an extension  $(\G^{\prime\prime}_U,\psi)$ 
 of a non-empty quadratic  word $U$ of genus
$k$. We define the {\emph{length of the extension}} to be the sum of
the lengths, in $F(X)$, of the labels of the cyclic words added in Step 2 of the extension. That
is, if $v_1,v_2,\ldots ,v_m $ are the vertices of $\Gamma _{U}$,  and these
are extended by cyclic words $C_1,C_2, \ldots ,C_m$ respectively, then the
length of the extension is  
\begin{equation*}
\sum_{i=1}^m|\psi(C_i)|.
\end{equation*}
\begin{definition}\label{defn:ext}
Let $(\G^{\prime\prime}_U,\psi)$  be an extension of $U$ over $H$, 
let $v_1,\ldots ,v_t$ be a subset of the vertices of $\Gamma_U$  such that 
$v_i$ is  
 extended by a cyclic word $C_i$ and  let $w_i=\psi(C_i)$, for $i=1,\ldots , t$. Then
we say that $(\G^{\prime\prime}_U,\psi)$ is a {\emph{genus $g$ joint extension}} 
\emph{on}  $(v_1,\ldots ,v_t)$ \emph{by words} $(w_1,\ldots ,w_t)$, if 
\be[(i)]
\item\label{it:ext1} 
$g=\genus_H(w_1,\ldots, w_t)+t-1$ and  
\item\label{it:ext2} if $t=1$ and $v_1$ has degree $1$ or $2$ then 
either 
\be
\item $g\ge \frac{1}{2}$ or 
\item $U=A^{\pm 2}$, for some $A\in \cA$.
\ee 
\ee 
If $\genus_H^+(w_1,\ldots ,w_t)=g-t+1$, in \ref{it:ext1} above, then
we say that  $(\G^{\prime\prime}_U,\psi)$ is an \emph{orientable} 
extension on $(v_1,\ldots, v_t)$ and otherwise that it is \emph{non-orientable}
 on $(v_1,\ldots, v_t)$. 
\end{definition}
(Note that condition \ref{it:ext2} implies that if $U\neq A^{\pm 2}$ and 
$g=0$, which means that
$t$ must be $1$, then $v_1$ has degree at least $3$.) 

\begin{definition}
Partition the vertices of $\Gamma_U$ into $p$ sets
$V_1,\ldots ,V_p$.  If 
\be[(i)]
\item 
$(\G^{\prime\prime}_U,\psi)$ is  
 a genus $g_i$ joint extension 
on the vertices in the set $V_i$, 
for all $i=1,\ldots ,p$, and 
\item 
$\sum_{i=1}^pg_i=g$, 
\ee 
then $(\G^{\prime\prime}_U,\psi)$ is called a \emph{genus $g$ extension of $U$ over $H$}. If $U$ is orientable and $(\G^{\prime\prime}_U,\psi)$ is 
orientable on $V_i$, for $1\le i\le p$, then this extension is called
\emph{orientable}. Otherwise it is said to be \emph{non-orientable}.
\end{definition}
\begin{example}\label{ex:gg3}{\bf{(A genus $3$ extension)}}\\
Consider the quadratic orientable word $U=ABCC^{-1}B^{-1}A^{-1}$ of genus $0$. Let
$u_1,u_2,v_1,v_2$ be the vertices of the associated graph
$\Gamma_U$, as shown in Figure \ref{extex}. We  construct a  genus $2$ joint extension on the vertices
$u_1$ and $u_2$, by words $w_1$ and $w_2$ respectively, and a genus $1$ joint
extension on the vertices
$v_1$and $v_2$, by words $z_1$ and $z_2$ respectively (see Figure \ref{extex}).
\begin{figure}[htp]
\begin{center}
\psfrag{A1}{{\scriptsize $A_1$}}
\psfrag{A2}{{\scriptsize $A_2$}}
\psfrag{B1}{{\scriptsize $B_1$}}
\psfrag{B2}{{\scriptsize $B_2$}}
\psfrag{C1}{{\scriptsize $C_1$}}
\psfrag{C2}{{\scriptsize $C_2$}}
\psfrag{A}{{\scriptsize $A$}}
\psfrag{B}{{\scriptsize $B$}}
\psfrag{C}{{\scriptsize $C$}}
\psfrag{u1}{{\scriptsize $u_1$}}
\psfrag{u2}{{\scriptsize $u_2$}}
\psfrag{v1}{{\scriptsize $v_1$}}
\psfrag{v2}{{\scriptsize $v_2$}}
\psfrag{w1}{{\scriptsize $c_1$}}
\psfrag{w21}{{\scriptsize $c_{21}$}}
\psfrag{w22}{{\scriptsize $c_{22}$}}
\psfrag{z11}{{\scriptsize $d_{11}$}}
\psfrag{z12}{{\scriptsize $d_{12}$}}
\psfrag{z2}{{\scriptsize $d_2$}}
\includegraphics[scale=0.6]{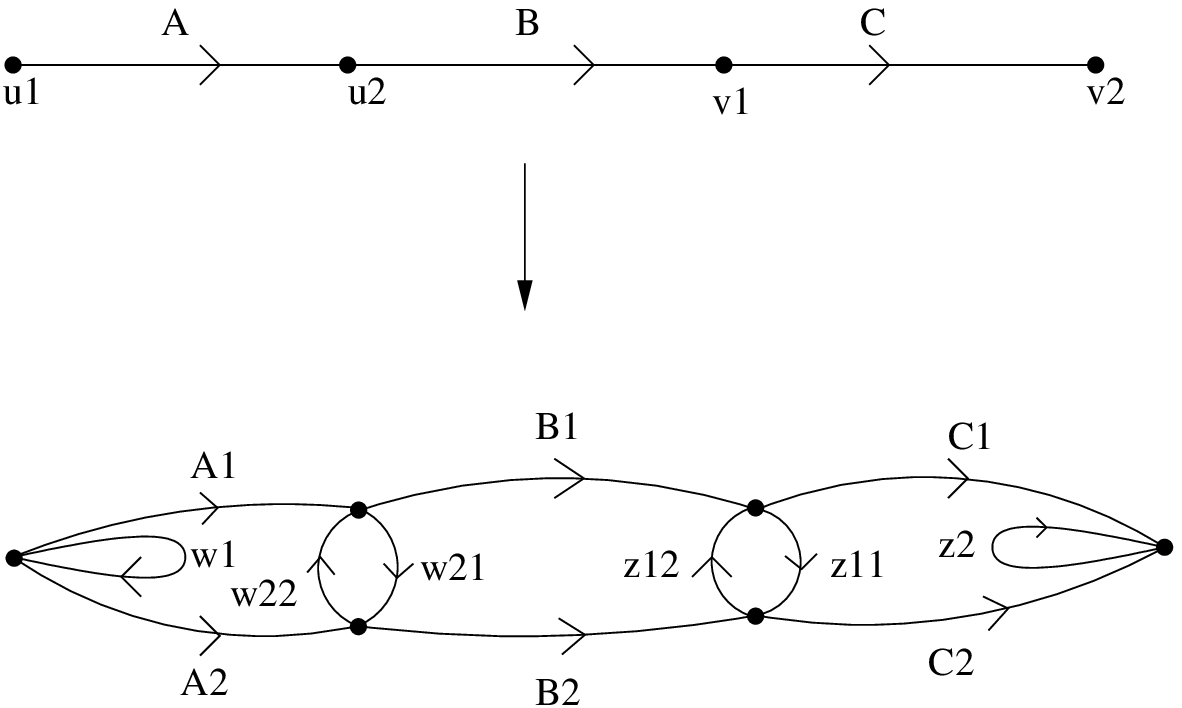}
\\
\vspace{0.38cm}                                                                  
\refstepcounter{figure}\label{extex}                                             
Figure \thefigure                  
\end{center}
\end{figure}
Here all words in the generators of $H$ which are mentioned are assumed to  be $H$-minimal, 
 $w_2=w_{21}w_{22}$ and $z_1=z_{11}z_{12}$. The labelling function
$\psi$ is such that $\psi(c_1)=w_1$, $\psi(d_2)=z_2$, $\psi(c_{ij})=
w_{ij}$, $\psi(d_{ij})=z_{ij}$, 
\begin{eqnarray*}
\psi(A_1) & =_H & w_1\psi(A_2)w_{22},\\
\psi(B_1) & =_H & w_{21}\psi(B_2)z_{12},\\
\psi(C_1) & =_H & z_{11}\psi(C_2)z_2,
\end{eqnarray*} 
$\psi(A_1)\psi(B_1)\psi(C_1)\psi(C_2^{-1})\psi(B_2^{-1})\psi(A_2^{-1})$
is freely cyclically reduced, 
$w_1w_2$ is a commutator in $H$ and  $z_1z_2=1$  in $H$. 
\end{example}

\begin{lemma}\label{lem:ggext}
Let $(\G^{\prime\prime}_U,\psi)$ be a genus $g$ extension of the quadratic
 word $U$ of genus $k$ and let $U^\prime$ be the Hamiltonian
cycle associated to $U$. Let $h_1,\ldots, h_s$ be elements of $H$  
such that $\prod_{j=1}^tr_jh_jr_j^{-1}=_H \psi(U^\prime)$, for some $r_j\in H$.  
If  $(\G^{\prime\prime}_U,\psi)$ 
is orientable, then $\genus_H^+(h_1,\ldots, h_s)\le g+k$ and otherwise 
 $\genus_H^-(h_1,\ldots, h_s)\le g+k$. 
\end{lemma}

\begin{proof} 
The graphs $\G_U$ and $\G^{\prime\prime}_U$ are embedded on a closed surface $\S_U$, by construction. As $\S_U$ is constructed by identifying segments of the
boundary of the disk $D$, 
the Hamiltonian circuit $U^\prime$, associated to $U$, labels
 a closed path in $\G^{\prime\prime}_U$ which
 bounds an open disk  embedded in $\S_U$ and not meeting $\G_U$. 
Remove this disk from $\S_U$, to leave  a closed surface $\S_{U,0}$, with one boundary competent $\b_0$. Then $\S_{U,0}$ is the
convex-hull of the  $2$-complex, which has $1$-skeleton $\G_{U''}$ and $2$-cells the closures of components of 
 $\S_{U,0}\bs \G_{U''}$, and moreover $\b_{0}$ is  labelled
$\psi(U^\prime)$.  As $\prod_{j=1}^sr_jh_jr_j^{-1}=_H\psi(U^\prime)$ 
there exists a van Kampen diagram $\D_H$, over $H$,  
on surface of genus $0$, 
with $s+1$ boundary components $\b_1,\ldots ,\b_s, \b$, with $\b_j$ labelled 
$h_j$ and    $\b$ labelled $\psi(U^\prime)$. 
Identifying $\b^\prime$ with $\b_{0}$, along $\psi(U^\prime)$, results in a surface $\S_U^{\prime}$, 
with $s$ boundary components $\b_1,\ldots ,\b_s$, labelled 
 with the $h_j$'s. Indeed we may regard $\S'_U$ as  
obtained from $\S_U$ by removing the interior of $s$ disks. Moreover 
$\G^{\prime\prime}_U$ is embedded in $\S_U^{\prime}$ and does not meet $\b$. 

Now assume the vertices of $\G_U$ were $v_1,\ldots, v_t$ and assume that $v_i$ has been extended by the 
cycle graph $C_i$. For each $i$, the cycle $C_i$ bounds a 
disk $D_i$ in $\S'_U$, the interior of which does not meet $\G^{\prime\prime}_U$. 
For all $i$, remove the interior of $D_i$ from $\S'_U$ 
to leave a surface $\S_U^{\prime\prime}$ with
$t+s$ boundary components labelled $h_1,\ldots, h_s$ and $\psi(C_1),\ldots, \psi(C_t)$. From the conditions
imposed on the labelling function $\psi$ it follows that the labelled graph $\G_U^{\prime\prime}$, together
with the diagram $\D_H$, form a van Kampen diagram over $H$, on the surface  $\S_U^{\prime\prime}$. 

Now suppose that the vertices of $\G_U$ are partitioned into sets $V_1$, ..., $V_p$ in such a way that
$(\G_U^{\prime\prime},\psi)$ forms a genus $g_i$ joint extension on $V_i$, with $\sum_{i=1}^p g_i=g$. Let 
$w_i=\psi(C_i)$, for all $i$. Then, if $V_i=\{v_{i,1}\ldots, v_{i,t_i}\}$, 
by definition we have $\genus_H(w_{i,1}\ldots, w_{i,t_i})=g_i-t_i+1$. Hence, there exists a van Kampen diagram, 
over $H$,  on a
compact surface $T_i$ of genus $g_i-t_i+1$, with $t_i$ boundary components labelled $w_{i,1}\ldots, w_{i,t_i}$.
Attach the boundary component of $T_i$ labelled $w_{i,j}$ to the boundary component of $\S_U^{\prime\prime}$
labelled $w_{i,j}$, for $j=1,\ldots , t_i$. In this way we attach a handle, 
or cross-cap, of genus 
$t_i-1+\genus(T_i)= g_i$ to $\S_U^{\prime\prime}$. Repeat the process for  all
sets $V_i$ of the  partition of the 
 vertices of $\G_U$. As $\S_U$
has genus $k$, so does $\S_U^{\prime\prime}$, so the result is a surface 
$\S$ of genus
$k+\sum_{i=1}^p g_i=k+g$. We have now constructed a van Kampen diagram, over $H$, on $\S$, with 
$s$ boundary components, labelled $h_1,\ldots , h_s$. Hence 
$\genus_H(h_1,\ldots , h_s)\le g+k$.
\end{proof}
\begin{example}\label{ex:gg3_cont}
In Example \ref{ex:gg3} we constructed a genus $3$ extension consisting of a joint genus $2$ extension
by words $w_1,w_2$, and a joint genus $1$ extension by $z_1$ and $z_2$. To construct the corresponding
van Kampen diagram on a surface of genus $3$ first embed the graph $\G^{\prime\prime}$ in the disk, with
boundary labelled $A_1B_1C_1C_2^{-1}B_2^{-1}A_2^{-1}$. Then remove the interior of the 
disks labelled $c_1$, $c_{21}c_{22}$, $d_{11}d_{12}$ and $d_2$. Finally, 
attach a torus with two boundary components,
labelled $c_1$ and $c_{21}c_{22}$ and a sphere with two boundary components labelled $d_2$ and $d_{11}d_{12}$,
in the obvious way: see Figure \ref{fig:exextvk}.
\begin{figure}[htp]
\begin{center}
\psfrag{A1}{{\scriptsize $A_1$}}
\psfrag{A2}{{\scriptsize $A_2$}}
\psfrag{B1}{{\scriptsize $B_1$}}
\psfrag{B2}{{\scriptsize $B_2$}}
\psfrag{C1}{{\scriptsize $C_1$}}
\psfrag{C2}{{\scriptsize $C_2$}}
\psfrag{A}{{\scriptsize $A$}}
\psfrag{B}{{\scriptsize $B$}}
\psfrag{C}{{\scriptsize $C$}}
\psfrag{u1}{{\scriptsize $u_1$}}
\psfrag{u2}{{\scriptsize $u_2$}}
\psfrag{v1}{{\scriptsize $v_1$}}
\psfrag{v2}{{\scriptsize $v_2$}}
\psfrag{w1}{{\scriptsize $c_1$}}
\psfrag{w21}{{\scriptsize $c_{21}$}}
\psfrag{w22}{{\scriptsize $c_{22}$}}
\psfrag{z11}{{\scriptsize $d_{11}$}}
\psfrag{z12}{{\scriptsize $d_{12}$}}
\psfrag{z2}{{\scriptsize $d_2$}}
\psfrag{SU0}{{\scriptsize $\G_{U,0}$}}
\psfrag{Dh}{{\scriptsize $\D_{h}$}}
\psfrag{h}{{\scriptsize $h$}}
\psfrag{+}{{$\mathbf{+}$}}
\begin{subfigure}[b]{.9\columnwidth}
\begin{center}
\includegraphics[scale=0.6]{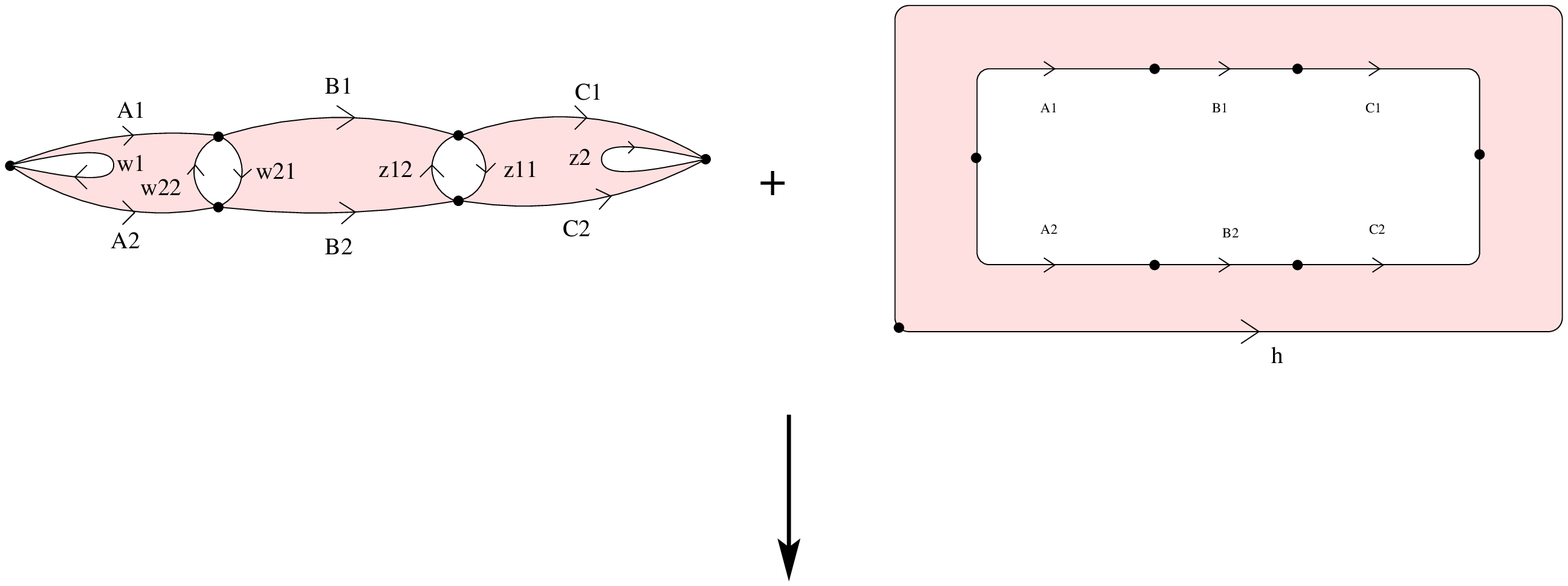}
~\\[1em]
\end{center}
\end{subfigure}
~\\[1em]
\begin{subfigure}[b]{.9\columnwidth}
\begin{center}
\includegraphics[scale=0.6]{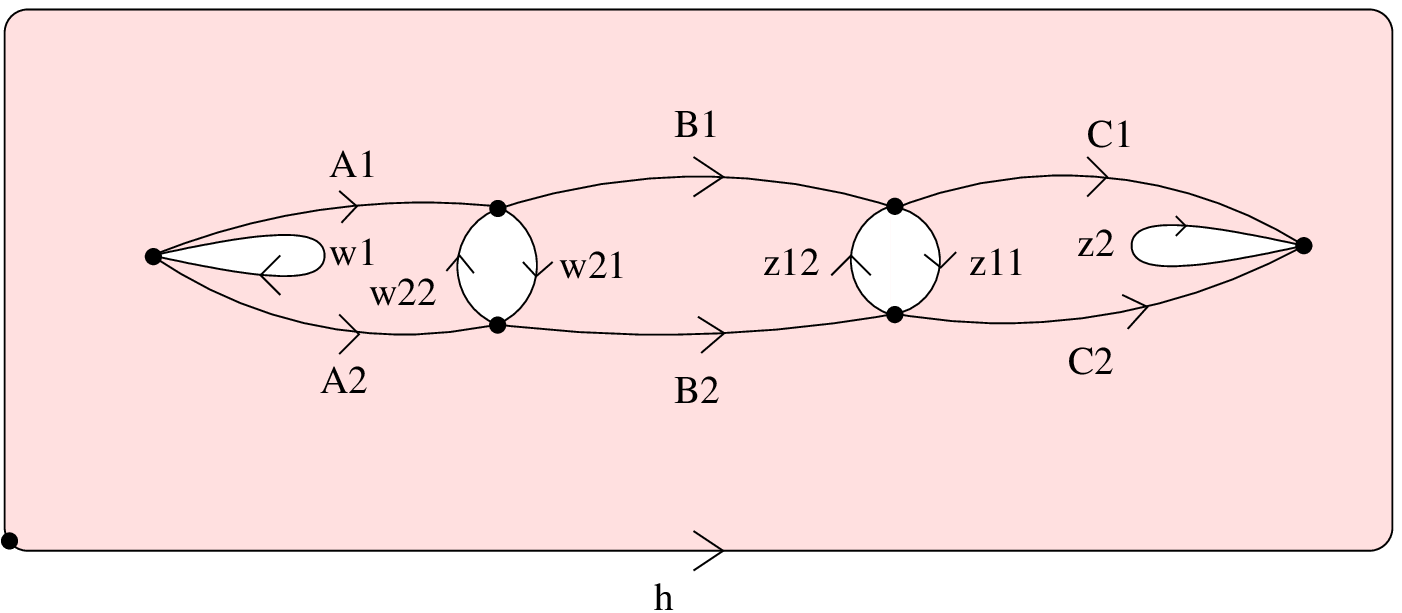}
\end{center}
\end{subfigure}
~\\[1em]
\begin{subfigure}[b]{.45\columnwidth}
\begin{center}
\includegraphics[scale=0.45]{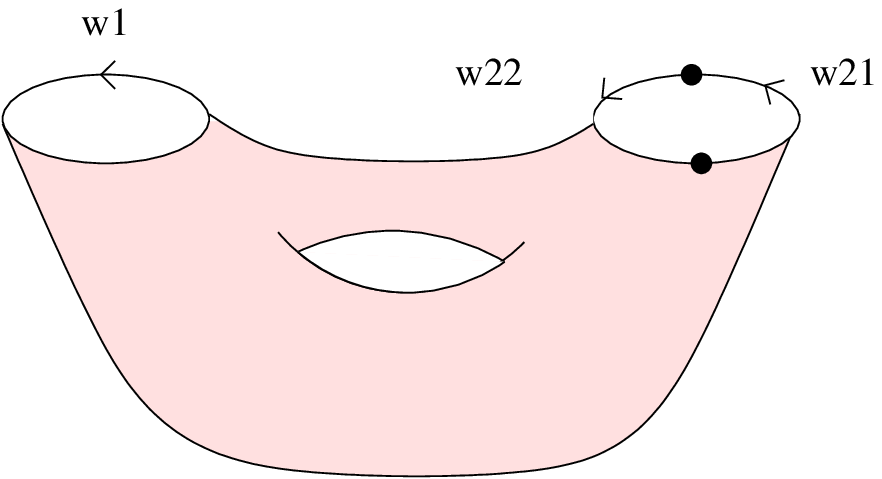}
\end{center}
\end{subfigure}
\begin{subfigure}[b]{.45\columnwidth}
\begin{center}
\includegraphics[scale=0.45]{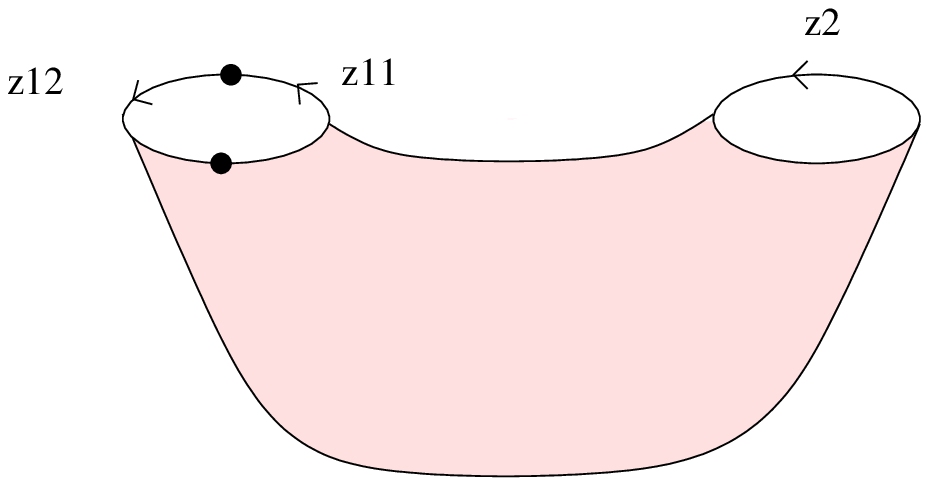}
\end{center}
\end{subfigure}
\caption{A van Kampen diagram constructed from an extension of $U$}
\label{fig:exextvk}                                             
\end{center}
\end{figure}
\end{example}
\subsection{The main theorems}\label{sec:mainthm}
We are now in a position to state the main results of this paper, which give a
method of constructing forms for elements of (orientable or non-orientable) genus $n$ in $H$.
\begin{definition}\label{defn:constants}
Let $H$ be a $\d$-hyperbolic group, with respect to the generating set $X$. 
For an integer $k\ge 0$, let $B_H(k)$ be the set of elements of $H$
represented by  words of length at most
$k$ in $F(X)$.
Define $M=|B_H(4\d)|$: and,   for a positive element $n\in \ZZ[\frac{1}{2}]$, define 
$K(n)=12n-6$, $n\ge 1$, $K(1/2)=2$, $K(0)=0$ and 
$l(n)=\delta(\log _2(K(n))+1)$.
\end{definition}
Assume that $h\in H$, that  the orientable genus of 
$h$ is  $\genus_H^+ (h)=g$,  that $U$ is an orientable Wicks form, 
over $\cA$,  
of genus $g$, and that $\phi$ is  a labelling
 function such that $\phi(U)$ is conjugate to $h$.  
Then $\phi(U)$ has orientable genus $g$ in $F(X)$. 
(It cannot have orientable genus less than $g$, otherwise
$h$ would, as well.) It follows,  from \cite{Culler81}, that 
there exists an orientable Wicks form $W$ of 
genus $g$ over $\cA$, and a labelling function
$\psi$ such that $\psi(W)$ is conjugate, in $F(X)$, to $\phi(U)$, 
and the cyclic word obtained by  
substituting $\psi(X)$ for $X$, for all  $X\in \supp(W)$, is reduced as written, as 
a word in $F(X)$. That is $\psi(W)$ is freely
cyclically reduced and $|\psi(W)|=
2\sum_{X\in\supp(W)}|\psi(X)|$. Call such a labelling function \emph{cancellation free} on $W$.  
The same holds for the non-orientable genus: replacing  ``orientable'' by ``non-orientable'' in the paragraph
above. 

Let $W$ be a Wicks form and let $D$ be a proper subset of $\supp(W)$. The word $U$ obtained from $W$ by
setting all elements of $D$ equal to $1$ is called a \emph{specialisation} of $W$. For example, up to permutation
of letters in the support, the Wicks form $ABCA^{-1}B^{-1}C^{-1}$ has specialisations $AA^{-1}$, $ABA^{-1}B^{-1}$ and $ABCA^{-1}B^{-1}C^{-1}$. 
\begin{theorem}\label{Thexp}
Let $h$ be a word in $X\cup X^{-1}$, let $n$ be a positive integer and let
  $l=l(n)$. Then  $\genus_H^+(h)=n$ if and only 
if $n$ is the minimal integer such that the following holds. There exist 
words $F,R\in F(X)$ such that  
$h=_H RFR^{-1}$, and an orientable Wicks form $W$ over $\cA$, of genus $n$, satisfying   either 
\ref{it:exp1} or \ref{it:exp2} below.
\begin{enumerate}
\item\label{it:exp1} $F=\theta(W)$, where 
 $\theta$ is a cancellation free labelling on $W$, 
 $|\theta(E)|\leq 12l+M+4$, for all $E\in \supp(W)$ (so
$|F|\leq (12n-6)(12l+M+4)$) and $|R|\le |h|/2+6l+3M/2+2\d+7/2$. 
\item\label{it:exp2}  $F$ is $H$-minimal and $F=_H\psi(U^\prime)$, where $U^\prime$ 
is the Hamiltonian cycle of an orientable, genus
  $g$  extension $(\G^{\prime\prime}_U,\psi)$, of length at most $2(12n-6)(12l+M+4)$, of some 
  specialisation $U$ of $W$, of  genus $k$ over $\cA$, where $n=g+k$. 
In this
case $|R|\le |h|/2+2\d$. 
\end{enumerate}
\end{theorem}

\begin{theorem}\label{Thexp-}
Let $h$ be a word in $X\cup X^{-1}$, let $n\in \ZD$, $n>0$ and  let  
  $l=l(n)$. Then  $\genus_H^-(h)=n$ if and only 
if $n$ is the minimal positive element  of $\ZD$ such that the following hold. There exist 
words $F,R\in F(X)$ such that 
$h=_H RFR^{-1}$, and a non-orientable Wicks form $W$ over $\cA$, of genus $n$, satisfying either 
\ref{it:exp-1} or \ref{it:exp-2} below.
\begin{enumerate}
\item\label{it:exp-1} $F=\theta(W)$, where 
  $\theta$ is a  cancellation free labelling on $W$, 
$|\theta(E)|\leq 12l+M+4$, for all $E\in \supp(W)$; so
$|F|\leq K(n)(12l+M+4)$. In this case $|R|\le |h|/2+6l+3M/2+2\d+7/2$.
\item\label{it:exp-2}  $F$ is $H$-minimal and $F=_H\psi(U^\prime)$, where $U^\prime$ 
is the Hamiltonian cycle of a non-orientable, genus
  $g$  extension $(\G^{\prime\prime}_U,\psi)$, of length at most $2K(n)(12l+M+4)$, of some
 specialisation   
 $U$ of $W$ of  genus $k$ over $\cA$, where  $n=g+k$. 
  In this
case $|R|\le |h|/2+2\d$. 
\end{enumerate}
\end{theorem}
The proof of these two theorems is the content of Section \ref{sec:proof}.

\section{Applications}\label{sec:applications}
\subsection{Forms for Commutators}\label{sec:commutators}
We use Theorem \ref{Thexp} to obtain a full list of all
possible forms for commutators in $H=\la X|S\ra$. Since in this
case $n=1$, it follows that  $l=\delta(\log _2(6)+1)$.
\begin{proposition}\label{propos}
An element $h\in H$ is a commutator  if and only if  there are  
words
$R$ and $F$ in $F(X)$, such that $h=_H RFR^{-1}$,  where 
 $F$ takes one of the following forms,  $|R|\le |h|/2+6l+3M/2+2\d+7/2$  in case \ref{it:propos1}, 
and $|R|\le |h|/2+2\d$ in cases \ref{it:propos2}, \ref{it:propos3} and \ref{it:propos4}.
\begin{enumerate}
\item\label{it:propos1} $F= XYZX^{-1}Y^{-1}Z^{-1}$, for 
$X,Y,Z\in F(X)$ with  
$|X|$, $|Y|$, $|Z|\leq 12l+M+4$.
\item\label{it:propos2} $F=A_1A_2^{-1}$ with $A_1=_H \xi _1^{-1}A_2 \xi _2$, where $|\xi
_1|+|\xi _2|\leq 12(12l+M+4)$ and $\xi _1$ is conjugate to $\xi _2$ in
$H$.  
\item\label{it:propos3} $F=A_1B_1A_2 ^{-1}B_2 ^{-1}$ with $A_1=_H \xi _1A_2\xi _3$, $B_1=_H
\xi _4B_2\xi _2$, where $|\xi _1|+|\xi _2|+|\xi _3|+|\xi _4|\leq
12(12l+M+4)$   and
$\xi _1\xi _2\xi _3\xi _4=_H 1$.
\item\label{it:propos4} $F=A_1B_1C_1A_2^{-1}B_2^{-1}C_2^{-1}$ with $A_1=_H \xi _1A_2\rho
  _1$,  $B_1=_H
\rho _2B_2\xi _2$ and $C_1=_H \xi _3C_2\rho _3$, where $|\xi _1|+|\xi
_2|+|\xi _3|+|\rho _1|+|\rho _2|+|\rho _3|\leq 12(12l+M+4)$ and   $\xi _1\xi
_2\xi _3 =_H\rho _1\rho _2\rho _3=_H 1$. 
\end{enumerate}
\end{proposition}
\begin{proof}
By Theorem \ref{Thexp}, $h$
is conjugate to a word $F$ which
either has form $1$
above or is $H$-minimal and obtained by an orientable, genus $g$ extension, of length at most $12(12l+M+4)$,
 of some specialisation $U$ of the Wicks form $ABCA^{-1}B^{-1}C^{-1}$, 
of genus $k$, 
where  $g+k=1$. The bounds
on $|R|$ also follow directly from the theorem. 
There are only
three  specialisations 
which can have a suitable
extension. 
\begin{enumerate}[(i)]
\item If $k=0$ and $g=1$,  take a quadratic orientable word $U=AA^{-1}$, of length $2$ and genus $0$, with a joint genus
  $1$ extension constructed on the two vertices of $\Gamma_U$.
\item If $k=1$ and $g=0$ there are two possibilities.
\be[(a)]
\item
The orientable word $U=ABA^{-1}B^{-1}$ of genus $1$ with a genus
  $0$ extension constructed on the only vertex of $\Gamma_U$.
\item The orientable word $U=ABCA^{-1}B^{-1}C^{-1}$ of genus $1$ with a
  genus $0$ extension constructed on both of the vertices of $\Gamma_U$.
\ee
\end{enumerate}
{\flushleft{(i)}} We extend the graph $\Gamma_U$ as shown in Figure \ref{form2pic}. 
\begin{figure}[htp]
\begin{center}
\psfrag{A}{{\scriptsize $A$}}
\psfrag{A1}{{\scriptsize $A_1$}}
\psfrag{A2}{{\scriptsize $A_2$}}
\psfrag{w1}{{\scriptsize $\xi _1$}}
\psfrag{w2}{{\scriptsize $\xi _2$}}
\includegraphics[scale=0.6]{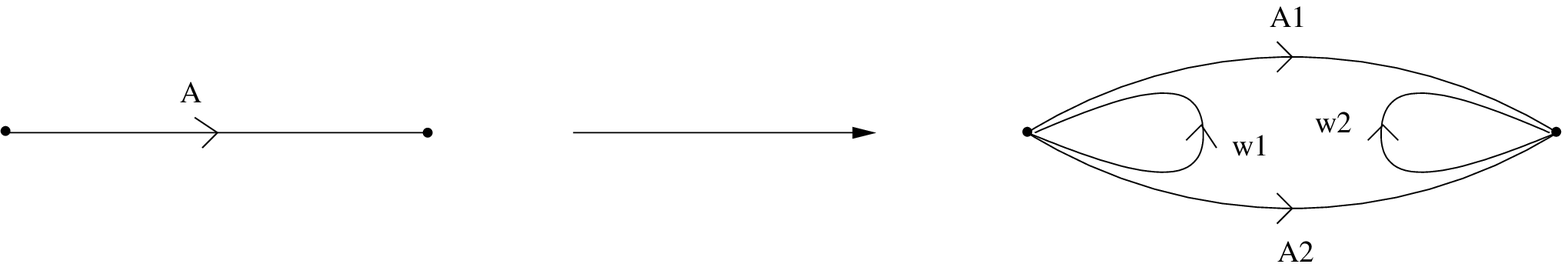}      
\caption{}\label{form2pic}
\end{center}
\end{figure} 
By Theorem \ref{Thexp}, $F$ takes the form of the Hamiltonian cycle
$A_1A_2^{-1}$ in the extended graph and from the nature of the
extension constructed on $U$, it is clear that we have form $2$. 

{\flushleft{(ii)(a)}} We extend the graph $\Gamma_U$ as shown in Figure
\ref{form3pic}.
\begin{figure}[htp]
\begin{center}
\psfrag{A}{{\scriptsize $A$}}
\psfrag{A1}{{\scriptsize $A_1$}}
\psfrag{A2}{{\scriptsize $A_2$}}
\psfrag{B}{{\scriptsize $B$}}
\psfrag{B1}{{\scriptsize $B_1$}}
\psfrag{B2}{{\scriptsize $B_2$}}
\psfrag{w1}{{\tiny $\xi _1$}}
\psfrag{w2}{{\tiny $\xi _2$}}
\psfrag{w3}{{\tiny $\xi _3$}}
\psfrag{w4}{{\tiny $\xi _4$}}
\includegraphics[scale=0.6]{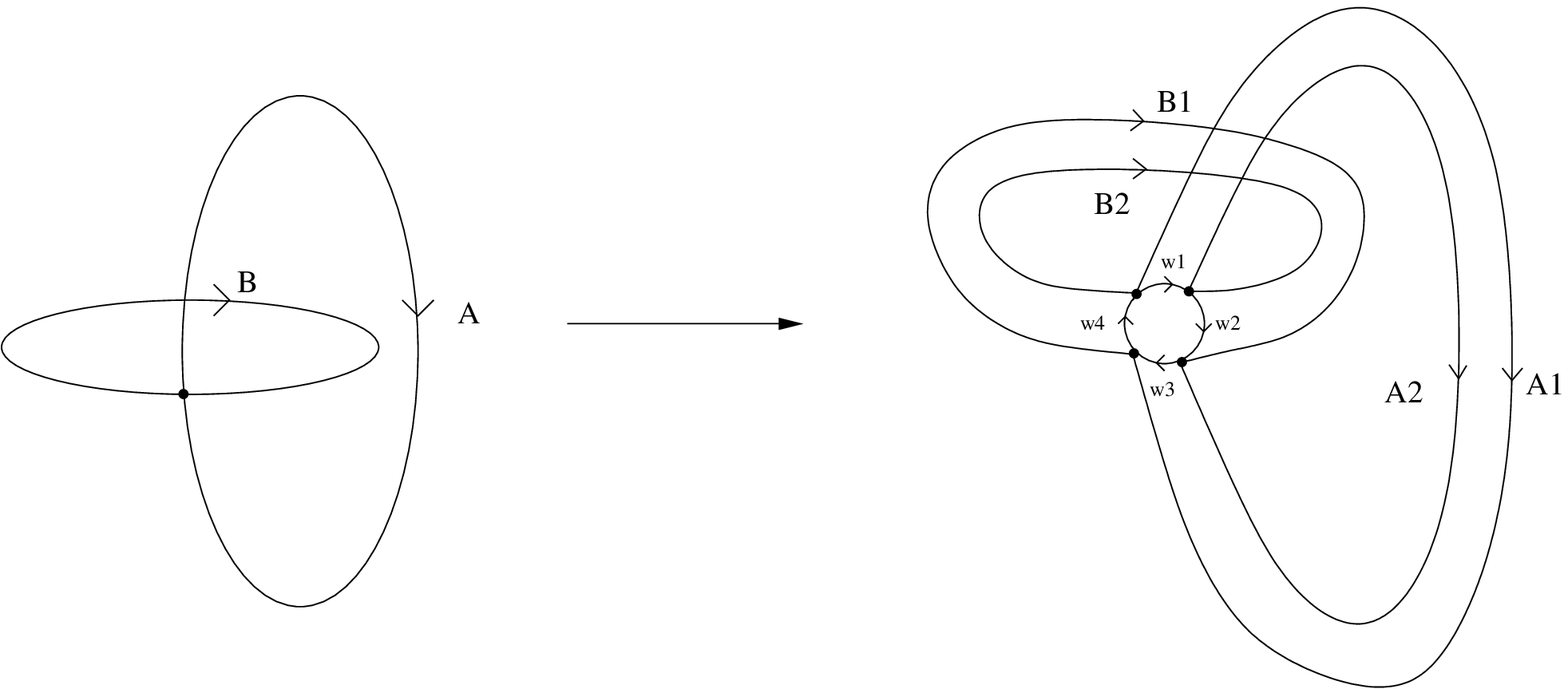}
\\
\vspace{0.38cm}                                                                  
\refstepcounter{figure}\label{form3pic}                                             
Figure \thefigure                  
\end{center}
\end{figure} 
By Theorem \ref{Thexp}, $F$ takes the form of the Hamiltonian cycle
$A_1B_1A_2^{-1}B_2^{-1}$ in the extended graph and from the nature of the
extension constructed on $U$, it is clear that we have form $3$. 

{\flushleft{(ii)(b)}}  Finally, we extend the graph $\Gamma_U$ as shown in Figure
\ref{form4pic}.
\begin{figure}[htp]
\begin{center}
\psfrag{A}{{\scriptsize $A$}}
\psfrag{A1}{{\scriptsize $A_1$}}
\psfrag{A2}{{\scriptsize $A_2$}}
\psfrag{B}{{\scriptsize $B$}}
\psfrag{B1}{{\scriptsize $B_1$}}
\psfrag{B2}{{\scriptsize $B_2$}}
\psfrag{C}{{\scriptsize $C$}}
\psfrag{C1}{{\scriptsize $C_1$}}
\psfrag{C2}{{\scriptsize $C_2$}}
\psfrag{z1}{{\tiny $\rho _1$}}
\psfrag{z2}{{\tiny $\rho _2$}}
\psfrag{z3}{{\tiny $\rho _3$}}
\psfrag{w1}{{\tiny $\xi _1$}}
\psfrag{w2}{{\tiny $\xi _2$}}
\psfrag{w3}{{\tiny $\xi _3$}}
\includegraphics[scale=0.35]{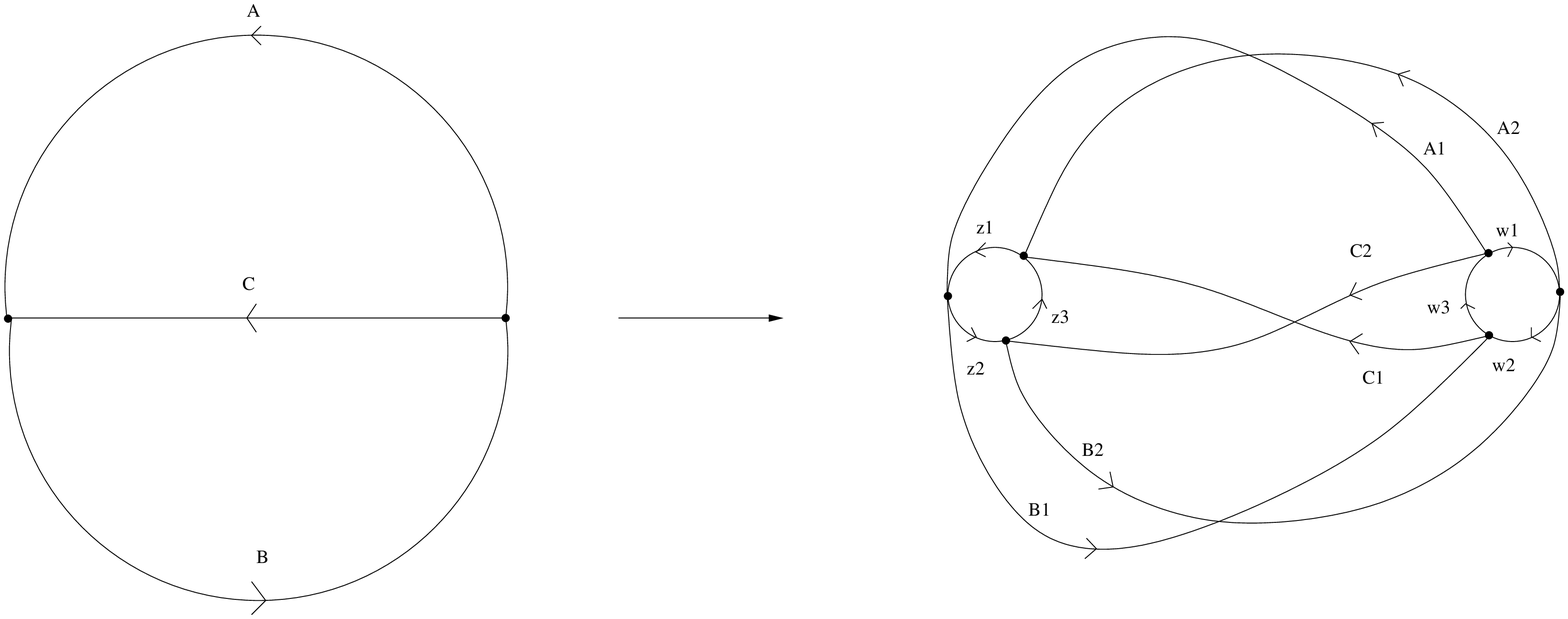}
\\
\vspace{0.38cm}                                                                  
\refstepcounter{figure}\label{form4pic}                                             
Figure \thefigure                  
\end{center}
\end{figure} 
Again, by Theorem \ref{Thexp}, $F$ takes the form of the Hamiltonian cycle
$A_1B_1C_1A_2^{-1}B_2^{-1}C_2^{-1}$ in the extended graph and from the
nature  of the
extension constructed on $U$, it is clear that we have form $4$. 
Hence $h$ is conjugate to some $F$ which takes one of the required forms. 

\end{proof}

\subsection{Forms for Squares}\label{sec:squares}
For a list of squares we use Theorem \ref{Thexp-}. 
 In this
case $n=1/2$, so   $l=3\delta$.
\begin{proposition}\label{prop:squares}
An element $h\in H$ is a square if and only if  there are  
words
$R$ and $F$ in $F(X)$, such that $h=_H
RFR^{-1}$,  where 
either 
\begin{enumerate}
\item $F= X^{2}$, for 
$X\in F(X)$ with  
$|X|\leq 12l+M+4$ and $|R|\leq |h|/2+12l +3M/2+ 2\delta+7/2$; or
\item $F=A_1A_2$ with $A_1=_H \xi A_2 \xi$, where 
$|\xi|\leq 5l+M+4$  and $|R|\le |h|/2 +2\d$. 
\end{enumerate}
\end{proposition}
\begin{proof}
By Theorem \ref{Thexp-}, $h$
is conjugate to a 
word $F$ which
either has form $1$
above or is obtained by a non-orientable, genus $g$ extension, of length at most $4(12l+M+4)$, 
 of some specialisation of the Wicks form $A^2$, 
 of genus $k$, 
where $g+k=1/2$.
There is 
one possible specialisation which can have a suitable
extension, namely the word $U=A^{2}$, of genus $k=1/2$. Therefore we
have a genus $0$ extension  over the unique vertex $v$ of $\G_U$. 
We extend the graph $\G_U$ as shown in Figure \ref{form5pic}. 
\begin{figure}[htp]
\begin{center}
\psfrag{A}{{\scriptsize $A$}}
\psfrag{A1}{{\scriptsize $A_1$}}
\psfrag{A2}{{\scriptsize $A_2$}}
\psfrag{w1}{{\scriptsize $\xi _1$}}
\psfrag{w2}{{\scriptsize $\xi _2$}}
\includegraphics[scale=0.6]{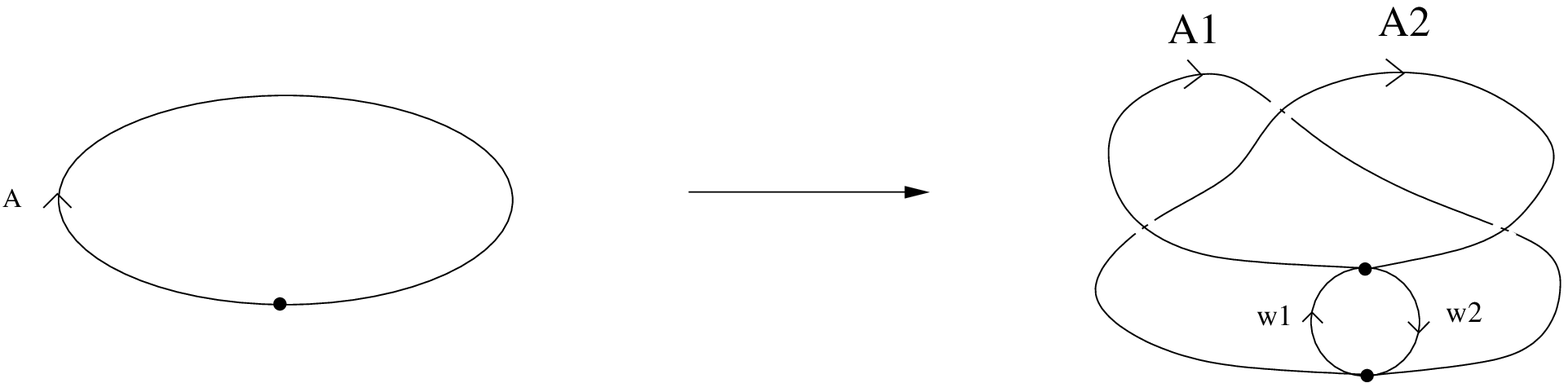}      
\caption{}\label{form5pic}
\end{center}
\end{figure} 
Following through the proof of Proposition \ref{prop:ext}, we see
that the length of this extension is at most $2(5l+M+4)$, since in this
case $v$ is an internal vertex and 
the words $b_1(A)=a_2(A)=1$, while $a_1(A)=b_2(A)$ has
length at most $5l+M+4$.  Also, as the extension is of genus $0$ we
have $\xi=\xi_2=_H \xi_1^{-1}$, and $\xi A_2 \xi A_1^{-1}=_H 1$, as 
required. 
\end{proof}
\subsection{Solutions of quadratic equations in hyperbolic groups}\label{sec:sqe}

In this section we extend Theorems \ref{Thexp} and \ref{Thexp-} to apply to quadratic tuples of elements 
of $H$. That is we show that, if $\genus_H(c_1,\ldots ,c_t)\le g$ then there exist conjugators $r_i$, linearly bounded
in terms of the sum of lengths of the $c_i$, such that $\prod_{i=1}^t c_i^{r_i}$ has one of the forms given in 
 Theorem \ref{Thexp} or \ref{Thexp-}, as appropriate. 
 In the process we give bounds on the lengths of images of variables in (minimal) solutions of quadratic equations
over hyperbolic groups. These are  similar to those found over free groups by Kharlampovich and Vdovina \cite{KV} and by Lysenok and Myasnikov \cite{LysenokMyasnikov}. For a torsion-free hyperbolic group $H$, Kharlampovich, Mohajeri, Taam and Vdovina \cite{KMTV} founds such bounds which are in somewhat different from 
those we give below. In fact  they compute a number $N$, 
dependent only $H$, such that 
if a quadratic equation $Q$ is solvable in $H$ then there is a solution in
which the variables are replaced by words of length at most $NL^d$, where 
$L$ is the length of the the equation $Q$,  as a word in the generators of 
$H$ and the  variables of $Q$, and $d=3$, if $Q$ is orientable, and $4$ otherwise. In contrast the bounds given belown depend on the genus of the equation,
but are linear (and torsion  is not excluded). 
 As usual, suppose that $H$  is $\d$-hyperbolic with respect to the presentation $\la X|S\ra$. 
  First we use a result from \cite{ols89} to bound the length of solutions of quadratic equations
over $H$.  
\begin{theorem}[\emph{cf.} {\cite{KMTV}[Theorem 1]}, {\cite{LysenokMyasnikov}[Theorem 1.1]}]\label{thm:qesol}
Let $H$ be a $\d$-hyperbolic group, with respect to the presentation $\la X|S\ra$, let $g\in \ZZ[\frac{1}{2}]$ and    
let $c_1,\ldots ,c_t$ be elements of $F(X)$. Then there exist constants $A$ and $B$, dependent only on $\d$, $|X|$, $g$ 
and $t$ such that the following hold. 
\be
\item\label{it:qesol1} If $\genus^+(c_1,\ldots ,c_t)_H\le g$ then there exists a solution $\phi$ to equation \eqref{eq:qgo} such that 
\[|\phi(v_l)|,|\phi(x_i)|,|\phi(y_i)|\le A\sum_{j=1}^t|c_j| +B,\] 
$i=1,\ldots, g$, $l=1,\ldots ,t$. 
\item\label{it:qesol2} If $\genus^-(c_1,\ldots ,c_t)_H\le g$ then there exists a solution $\phi$ to equation \eqref{eq:qgn} such that 
\[|\phi(v_l)|,|\phi(z_i)|\le A\sum_{j=1}^t|c_j| +B,\] 
$i=1,\ldots, 2g$, $l=1,\ldots ,t$.
\ee    
\end{theorem}
\begin{proof}
First note that we may assume that $R$ contains elements of length at most $8\d$ (see for example \cite{AlonsoBrady}). 
Let $c_1,\ldots ,c_t$ be 
orientable and let $X_i=\phi(x_i)$, $Y_i=\phi(y_i)$ and $V_l=\phi(v_l)$ be a solution to the equation \eqref{eq:qgo}.  
As in
  Section \ref{sec:gandvk} 
there is a van Kampen
diagram $\cK$ on a surface $\S$ of genus $g$, with boundary components labelled $c_i$, obtained from a disk 
with boundary labelled $\prod_{i=1}^{t}V_i^{-1}c_iV_i\prod_{i=1}^g[X_i,Y_i]$. We assume that the number of regions
of $\cK$ is minimal, among all such van Kampen diagrams. From \cite{ols89}[Lemma 11] we may 
cut $\cK$ along edges of $2$-cells to obtain a van Kampen diagram on the disk; in such a way that the 
 boundary label of the disk has length
at most $\sum_{j=1}^t|c_j|+C$, for some constant $C$ dependent only on $\d$, $|X|$, $g$ 
and $t$. As $H$ is hyperbolic it has a linear Dehn function $f(n)=Dn+E$, for some constants $D$ and $E$, dependent 
only on $\d$,  so
the van Kampen diagram on the disk has at most $D(\sum_{j=1}^t|c_j|+C)+E=D\sum_{j=1}^t|c_j|+E'$ regions. By construction
the same is true of $\cK$. Each of $X_i$, $Y_i$ and $V_l$ on $\cK$ corresponds to a simple path on $\cK$, formed
from a union of boundary components of $2$-cells. As the $2$-cells of $\cK$ have boundaries of length at most $8\d$, we may
assume that 
 each of  $V_l,X_i$ and $Y_i$ has length bounded by $4\d(D\sum_{j=1}^t|c_j|+E')=A\sum_{j=1}^t|c_j|+B$, with 
$A=4D\d$, $B=4E'\d$. This proves \ref{it:qesol1} and \ref{it:qesol2} follows similarly.  
\end{proof}

A \emph{system} of equations is a set of equations $(w_i=1)_{i\in I}$, where $I$ is an indexing set, and a 
 \emph{solution} to this system of equations over $H$ is an $H$-map $\phi: F(X)*F(\cA)\maps H$ such that $\phi(w_i)=_H 1$, 
for all $i\in I$. 
\begin{corollary}[\emph{cf.} {\cite{KV}[Theorem 2]}]\label{cor:qsys}
Let $\mbf w =w_1,\ldots ,w_t$ be a quadratic tuple of words over $\cA$ of genus $g\in \ZZ[\frac{1}{2}]$ and
let $c_1,\ldots, c_t$ be elements of $F(X)$ such that the system of equations $(w_ic_i=1)_{i=1}^t$ has solution in $H$. 
Let $d=\sum_{i=1}^t|c_i|.$ Then there exist constants $A$ and $B$, dependent only on $\d$, $|X|$, $g$ and $t$, such 
that the following hold. 
\be
\item \label{it:qsys1}
If $\mbf w$ is orientable then there exists a solution $\phi$ to the equation \eqref{eq:qgo}  over $H$ such that 
$|\phi(a)|\le Ad+B$, for all $a\in \{x_i,y_i\,|\, 1\le i\le g\}\cup \{v_i\,|\,1\le i\le t\}$.   
\item \label{it:qsys2} 
If $\mbf w$ is non-orientable then there exists a solution $\phi$ to the equation \eqref{eq:qgn}  over $H$ such that 
$|\phi(a)|\le Ad+B$, for all $a\in \{z_i\,|\, 1\le i\le 2g\}\cup \{v_i\,|\,1\le i\le t\}$.   
\ee
\end{corollary}
\begin{proof}
Let $\psi$ be a solution to $(w_ic_i=1)_{i=1}^t$. For each $i$ label the boundary of a disk $D_i$ with $w_iv_i^{-1}c_iv_i$, where $v_i\in \cA$, 
$v_i\notin \supp(w_j)$, for all $j$, and $v_i\neq v_j$, $i\neq j$. Setting $\psi(v_i)=1$, for all $i$, the system of equations 
$(w_iv_i^{-1}c_iv_i=1)_{i=1}^t$ also has solution $\psi$. 
 Let 
$D$ be the disjoint union of these disks. 
The quotient of $D$ obtained by identifying directed edges, according to their label 
in $\cA$, is a surface $\S$ (not necessarily connected) with  $t$ boundary components.
Collapsing the arcs labelled $v_i^{-1}c_iv_i$, on the boundary of each $D_i$, to a point collapses all boundary components of
$\S$ to  a point and, as noted above Definition \ref{def:redun}, as $\mbf w$ has genus $g$,   the surface obtained from $\S$, by this collapse, has genus $g$. Therefore
 $\S$ also has genus $g$. 
 If $\mbf w$ is orientable then so is  $\S$ and, read with an appropriate orientation of $\S$, the  
 the boundary labels of $\S$ are $c_1,\ldots ,c_t$ (see Section \ref{sec:gandvk}). 
If  $\mbf w$ is non-orientable, then with an appropriate choice of orientation of each boundary component, the 
 boundary labels of $\S$  are again $c_1,\ldots ,c_t$. 

Now relabel each directed edge on $D_i$ by applying $\psi$: that is an edge labelled  $a\in \cA$ is relabelled  
 $\psi(a)$. This results in a disk with boundary label $\psi(w_i)c_i=_H 1$, for each $i$. Hence there
is a van Kampen diagram over $H$ on each disk $D_i$. Thus there is also a van Kampen diagram on $\S$. 
Therefore there is a corresponding solution
$\phi$ to the equation \eqref{eq:qgo} in the case where $\mbf w$ is orientable and to \eqref{eq:qgn} otherwise. 
That is $\genus_H^{\pm}(c_1,\ldots ,c_t)\le g$ and the result follows from Theorem \ref{thm:qesol}.
\end{proof}
\begin{theorem}\label{thm:mgen}
With the hypotheses of Theorem \ref{thm:qesol},
 there exist constants $A'$ and $B'$, dependent only on $\d$, $|X|$, $g$ 
and $t$ such that the following hold.   If $\genus^{\pm}(c_1,\ldots ,c_t)_H = g$ there exist 
$F, r_1,\ldots ,r_t\in F(X)$ such that 
\[\prod_{i=1}^t r_i^{-1}c_ir_i=F\,\textrm{ and }\,  |r_i|\le  A'\sum_{j=1}^t|c_j|+B',\] 
for $i=1,\ldots ,t$, where  
\be
\item\label{it:mgen1} if $\genus^+(c_1,\ldots ,c_t)_H = g$ then 
$F$ satisfies \ref{it:exp1} or \ref{it:exp2} of Theorem \ref{Thexp} and
\item\label{it:mgen2} if $\genus^-(c_1,\ldots ,c_t)_H= g$ then $F$ 
satisfies \ref{it:exp1} or \ref{it:exp2} of Theorem \ref{Thexp-}.
\ee    
\end{theorem}
\begin{proof}
Let $c_1,\ldots ,c_t$ be 
orientable and let $\phi$ be a solution to equation \eqref{eq:qgo}. From Theorem \ref{thm:qesol}, there exist
constants  $A$ and $B$  such that $|\phi(a)|\le A\sum_{i=1}^t |c_i|+B$, 
for all variables $a$ in the equation. Then 
$h=\prod_{i=1}^t V^{-1}_ic_iV_i$ has orientable genus $g$, where $V_i=\phi(v_i)$, and 
 $|h|$ is bounded by $(2tA+1)\sum_{i=1}^t|c_i|+2tB$. From Theorem \ref{Thexp} there
exist $R,F\in F(X)$ such that $h=RFR^{-1}$ and $F$ and $R$  satisfy  \ref{it:exp1} or \ref{it:exp2} of the theorem. 
Here $R$ is bounded by $|h|/2 +C'$, for some constant 
$C'$ dependent only on  $\d$, $|X|$ and $g$, so we have $F=\prod_{i=1}^t r_i^{-1}c_ir_i$, where $r_i= R^{-1}V_iR$; and 
$|r_i|\le A'\sum_{i=1}^t|c_i|+B'$, for constants $A'$ and $B'$ dependent only on $\d$, $|X|$, $g$ and $t$. This proves
\ref{it:mgen1} and \ref{it:mgen2} follows similarly. 
\end{proof}
\section{Proof of Theorems \ref{Thexp} and \ref{Thexp-}}\label{sec:proof}
\subsection{Decomposition over short and long edges}\label{sec:pre-proof}
The proofs of Theorems \ref{Thexp} and \ref{Thexp-}
 involve consideration of  a genus $n$ Wicks form $W$ over $\cA$  and a homomorphism 
$\theta$ from $F(\cA)$ to $F(X)$ such that 
\begin{itemize}
\item $\theta(W)$  is conjugate to $h$ in $H$ and 
\item  amongst all such  Wicks forms (of the appropriate
orientation) and maps, $W$ and $\theta$ are minimal, in a sense
made precise below.
\end{itemize}
In outline: we shall  first establish some 
preliminary lemmas involving Wicks forms and then show 
that $W$ and $\theta$ either satisfy condition 1 of our theorems,  or $W$ can be
reduced to a quadratic  word $U$, by setting certain
letters equal to $1$, and from  $U$ an extension satisfying condition 2 may be constructed. 

We abbreviate notation to allow discussion of both theorems simultaneously. 
We write $\genus_H$ to mean $\genus_H^+$ or $\genus_H^-$, and 
``genus'' instead of ``orientable genus'' or ``non-orientable genus'', whenever possible. 
We  omit ``orientable'' and 
``non-orientable'' when talking of Wicks forms; the context should make
it clear which is meant. \\[1em]
\begin{proof}[Proof of Theorems \ref{Thexp} and \ref{Thexp-}.]\label{startthm}
From Lemmas \ref{lem:wicks} and \ref{lem:ggext} it follows that if the conditions of either  Theorem hold then
$\genus_H(h)=n$. It remains to prove the converse. 
In the orientable case (Theorem \ref{Thexp})  
we 
note that, from Lemma \ref{lem:wicks}, there exists no genus $m$ Wicks form $V$, $m<n$, with a
labelling function $\psi$ such that $\psi(V)$ is conjugate to $h$ in $H$.   
By hypothesis, the genus of $h$ is equal to $n$ in $H=\la X|S\ra$. Therefore, in the case of Theorem \ref{Thexp}, $n\in \ZZ$ and 
there exist  words $a_i,b_i\in F(X)$, for $i=1,\ldots , n$, such that 
\begin{equation*}
h=_H [a_1,b_1][a_2,b_2]\cdots [a_n,b_n].
\end{equation*}
The quadratic orientable word $U=[A_1,B_1]\cdots
[A_n,B_n]$, over $\cA$, is a genus $n$ Wicks form. Let
$\phi:F(\cA)\to F(X)$
be the labelling function defined by
 $\phi(A_i)=a_i$ and $\phi(B_j)=b_j$, for all $i,j=1,\dots ,n$, and $\phi(X)=1$, otherwise. 
Then $\phi(U)=h$. Moreover, $\phi(U)$ has genus $n$ in $F(X)$. 
 In this (orientable) case let $\mathcal{F}=\mathcal{F}(h)$ be the set of pairs $(U,\phi)$ where $U$ is an orientable
 genus
$n$ Wicks form and $\phi$ is a labelling function 
such that
$\phi(U)$ is conjugate to $h$ in $H$. 

In the non-orientable case, Theorem \ref{Thexp-}, 
$n\in \ZZ[1/2]$ and 
there exist  words $c_i\in F(X)$, for $i=1,\ldots , 2n$, such that 
\begin{equation*}
h=_H c_1^2\cdots c_{2n}^2.
\end{equation*}
The non-orientable  quadratic word $U=C_1^2\cdots
C_{2n}^2$, over $\cA$, is a genus $n$ Wicks form. In this case let 
$\phi:F(\cA)\to F(X)$
be the labelling function given by 
 $\phi(C_i)=c_i$, $i = 1,\ldots ,2n$, and $\phi(X)=1$, otherwise. 
Then $\phi(U)=h$ and  $\phi(U)$ has genus $n$ in $F(X)$. 
 In this (non-orientable) case, 
let $\mathcal{F}=\mathcal{F}(h)$ be the set of pairs $(U,\phi)$ where $U$ is a 
non-orientable
 genus
$n$ Wicks form and $\phi$ is a labelling function 
such that
$\phi(U)$ is conjugate to $h$ in $H$. 

In both orientable and non-orientable cases,
we call  a pair $(W,\theta)\in \cF$ \emph{minimal}  if
$|\theta(W)|$ is minimal amongst all pairs in
$\mathcal{F}$. Since we have shown that, in both cases, 
there is at least one pair in
$\mathcal{F}$ there is always a minimal pair. As above, from \cite{Culler81}, 
if $(U,\phi)$ is a minimal pair then we 
may choose a cancellation free pair $(W,\psi)$ such that $\psi(W)$ is conjugate to $\phi(U)$ in $F(X)$. Then
$(W,\psi)$ is necessarily minimal. Hence $\cF$ contains a cancellation free minimal pair. 
Moreover, if $(W,\psi)$ is minimal and cancellation free then, 
for each element $E$ of $\supp(W)$, it follows that $\psi(E)$ is $H$-minimal.

Now,  let $(W,\theta)$ be a minimal cancellation free pair in $\mathcal{F}$. 
If $p$ is a geodesic path between $2$ vertices of the Cayley graph 
$\G_X(H)$ and $p$ has label $\a$ then, by  abuse of notation,
we refer to the path $p$ as $\a$. In particular, for $(W,\theta)\in \cF$, 
if $W=A_1\cdots A_n$, where $A_i\in \cA^{\pm 1}$, then we 
consider $\theta({W})=\theta(A_1)\cdots \theta(A_n)$ 
as the concatenation of paths $\theta(A_i)$  in the Cayley graph 
$\Gamma_X(H)$. This path contains a geodesic subpath with label
$\theta(A_i)$, which we refer to as the path $\theta(A_i)$. On the
other hand there is a geodesic path in $\G_X(H)$ from $1$ to $\theta(W)$,
with label an $H$-minimal word $F$ such that $F=_H \theta(W)$. We
call such a  path a \emph{geodesic for} $(W,\theta)$ and refer to it as $F$. 
 
Now suppose that, for each letter $E$ of the Wicks form $W$, 
we have  $|\theta(E)|\leq L=12l+M+4$. In this case, in the light of Lemma \ref{cull},  
condition 1 of the appropriate Theorem is satisfied, apart from the bound on 
the length of $R$, which is deferred to Section \ref{sec:Rbound}. 
  Therefore, from now on, we shall assume
that there is at  least one letter of $W$ which is labelled by a word
of length
greater  than $L$ in $F(X)$. (This
 implies that there are two since each letter appears twice.)

For a pair $(W, \theta)\in \cF$ and $A\in \supp(W)$, we say that the occurrences
of  $A$ and $A^{-1}$ in $W$ are \emph{long edges of} $(W,\theta)$ (or just
$W$) if 
$|\theta(A)|>L$, as a word in $F(X)$. 
All other letters of $W$ are called {\emph{short edges}}.
  For
convenience in the proof, we may, without loss of generality, 
assume that the 
last letter
of  ${W}$ is a long edge.

 We 
state a number of preliminary lemmas, which we need later. The first
 is a version of Lemma 14 of  \cite{lysionok89}, 
and   establishes bounded length properties of subwords of $\theta(A)$,
where $A$ is any letter in  $\supp(W)$. 
Let $\Gamma_W$ be the genus $n$ graph associated to $W$.
\begin{lemma}\label{lem:l1}
Let $(W,\theta)$ be a minimal cancellation free pair in $\cF$,
let $A\in \supp(W)$, with signature $\s(A)=(\e,\d)$, and let 
$E_1,\ldots ,E_r,A^\e,E_{r+1},\ldots ,E_s,A^{\d},E_{s+1},\ldots ,E_t$
be the cyclic  sequence of letters in the  Eulerian circuit
$W$ in $\Gamma _W$. Let $\theta(A^\e)=a_1a_2$ and $\theta(A^\d)=a'_1a'_2$,
 where $a_1, a'_1, a_2$ and $a'_2$ are
subwords of the $H$-minimal words
$\theta(A^\e)$ and $\theta(A^\d)$. 
 Similarly, let $\theta(E_i)=e_{i1}e_{i2}$, 
for $1\leq i\leq t$. Then 
\begin{enumerate}[(i)]
\item\label{it:l1-1} $|a_1|\leq |e_{i2}\theta(E_{i+1}\cdots E_r)a_1|_H$, for
  $1\leq i\leq r$;
\item\label{it:l1-2} $|a_2|\leq |a_2\theta(E_{r+1}\cdots E_{j-1})e_{j1}|_H$,
  for $r+1\leq j\leq s$; 
\item\label{it:l1-3} $|a_1|\leq |a_2\theta(E_{r+1}\cdots
  E_sA^{\d}E_{s+1}\cdots E_{k-1})e_{k1}|_H$, for $s+1\leq k\leq t$;  
\item\label{it:l1-4} if $\d=\e$ and $|a_1|\ge |a'_1|$ then 
$|a_2|\le |a_2\theta(E_{r+1}\cdots E_s)a'_1|_H$ and 
\item\label{it:l1-5} if $\d=\e$ and $|a_2|\ge |a'_2|$ then 
  $|a_1|\le |a_2\theta(E_{r+1}\cdots E_s)a'_1|_H$.
\end{enumerate}   
\end{lemma} 
\begin{proof}
Without loss of generality we may assume that $\e=1$ so $\s(A)=(1,\d)$. 
First we shall prove statement \ref{it:l1-2}. 
 As in Section \ref{sec:gandvk} consider a disk $D$ with
its boundary divided into $t+2$ segments, 
labelled by the word $W=E_1\cdots E_t$. 
Suppose that  
 $E_j=X^\g$, for some $X\in \cA$ and $\g=\pm 1$. Then there is 
(unique) $i\neq j$
 such that $E_i=X^{\pm 1}$; say $E_i=E_j^{\xi}$, where $\xi =\pm 1$.  
Bisect $E_j$ into
two new edges, 
the  first new edge  denoted $E_{j1}$ and the
second  $E_{j2}$, where $E_{j1},E_{j2}$ are elements of
$\mathcal{A}^{\pm 1}$ not occurring in $W$.  Also replace the edge  
$E_i$ by the two edge path 
$E_{i1}E_{i2}$, where $E_{i1}=E_{j1}$ and $E_{i2}=E_{j2}$, if $\xi=1$, 
and  $E_{i1}=E_{j2}^{-1}$ and $E_{i2}=E_{j1}^{-1}$,
if $\xi=-1$. That is, 
$E_{i1}=E^{\xi}_{j\mu(\xi)}$ and $E_{i2}=E^{\xi}_{j\nu(\xi)}$. 

Now consider a properly embedded arc in $D$, 
with end points $\i(A)$ and $\t(E_{j1})$. Direct this arc 
from $\i(A)$ to
$\t(E_{j1})$ and label it $A^\prime$. (See Figure \ref{fig:ocut1a}.) 
\noindent
\begin{figure}[htp]
\begin{center}
\psfrag{E1}{{\scriptsize $E_1$}}
\psfrag{A}{\scriptsize $A$}
\psfrag{Ad}{\scriptsize $A^\d$}
\psfrag{A'}{\scriptsize $A'$}
\psfrag{Er}{{\scriptsize $E_{r}$}}
\psfrag{Er1}{{\scriptsize $E_{r+1}$}}
\psfrag{Ej-1}{{\scriptsize $E_{j-1}$}}
\psfrag{Ej1}{{\scriptsize $E_{j1}$}}
\psfrag{Ej2}{{\scriptsize $E_{j2}$}}
\psfrag{Ej+1}{{\scriptsize $E_{j+1}$}}
\psfrag{Es}{{\scriptsize $E_{s}$}}
\psfrag{Es+1}{{\scriptsize $E_{s+1}$}}
\begin{subfigure}[b]{.3\columnwidth}
\begin{center}
\includegraphics[scale=0.3]{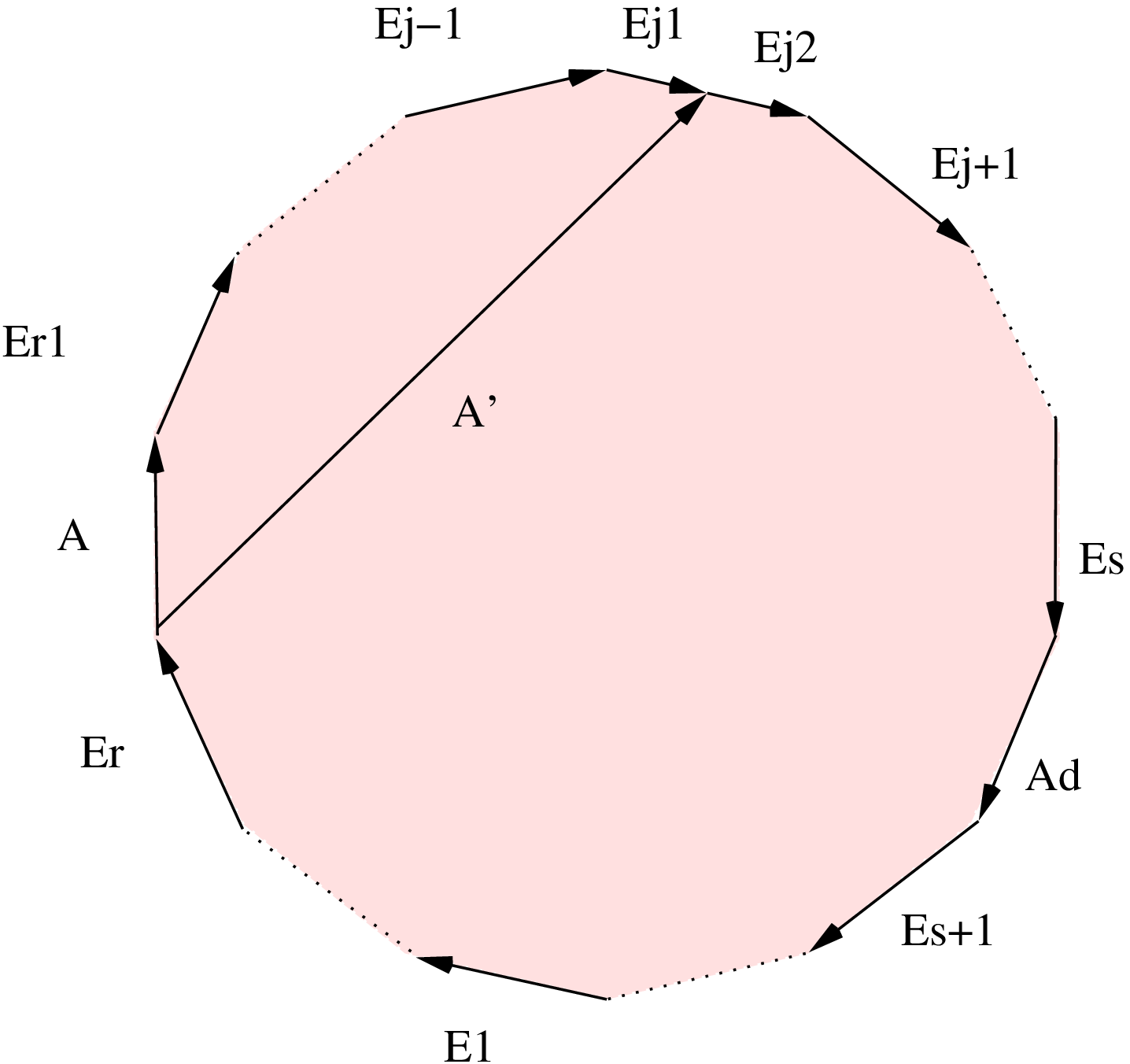}
\caption{}
\label{fig:ocut1a}
\end{center}
\end{subfigure}
%
\begin{subfigure}[b]{.3\columnwidth}
\begin{center}
\includegraphics[scale=0.3]{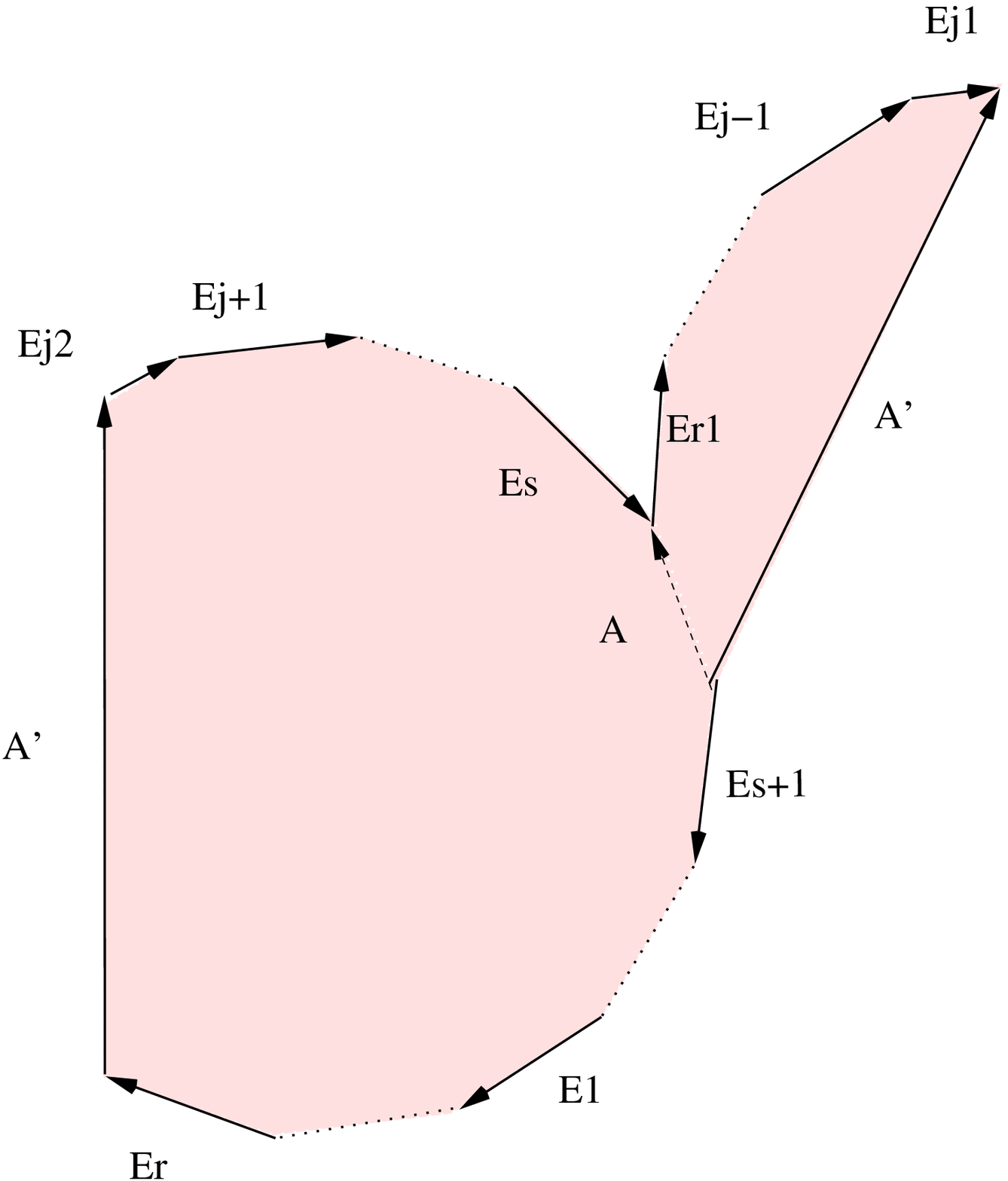}
\caption{$\d=-1$}
\label{fig:ocut1b}
\end{center}
\end{subfigure}
\begin{subfigure}[b]{.3\columnwidth}
\begin{center}
\includegraphics[scale=0.3]{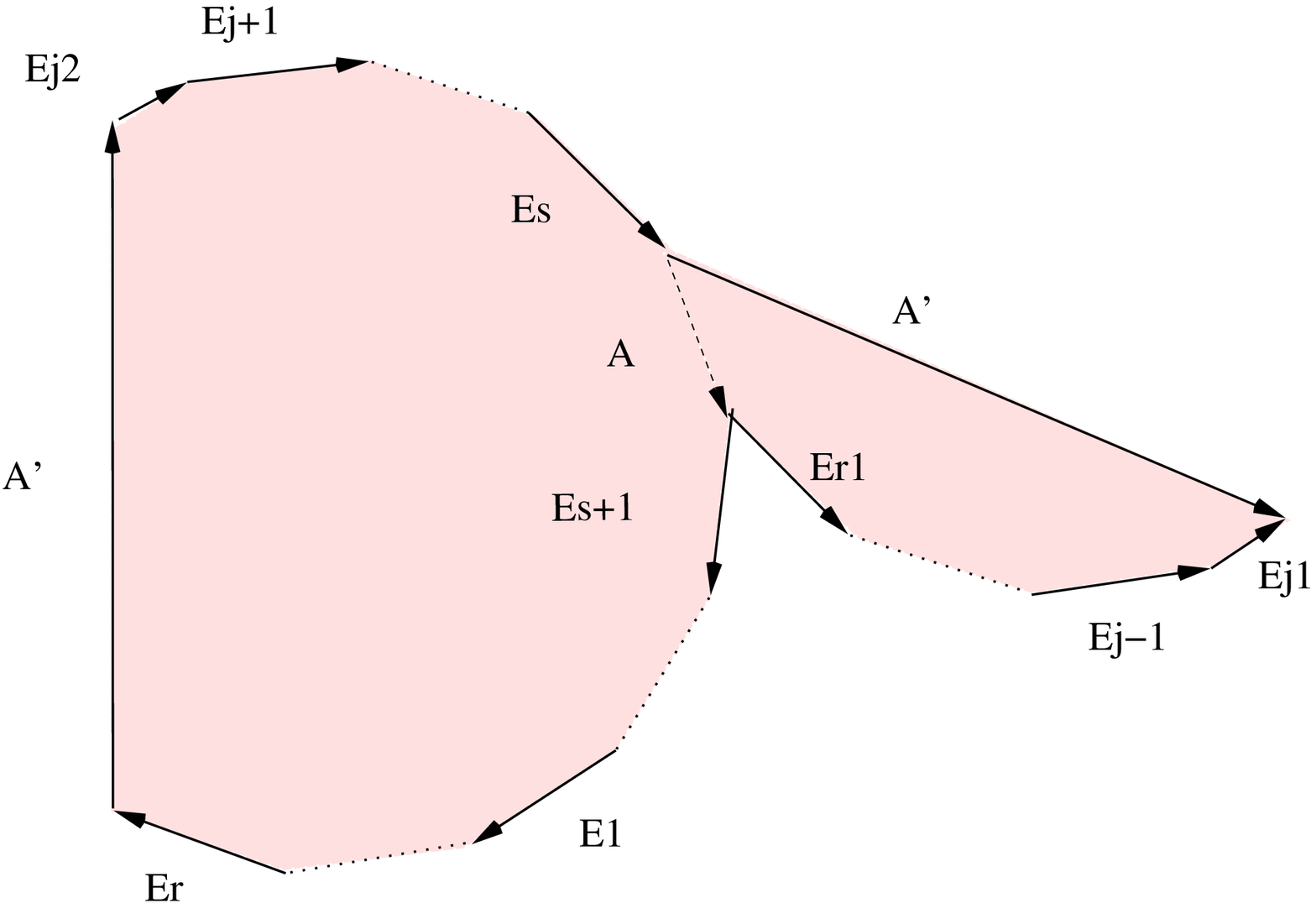}
\caption{$\d=1$}
\label{fig:ocut1c}
\end{center}
\end{subfigure}
\caption{}\label{fig:ocut1}
\end{center}
\end{figure}
Cut the disk $D$ along 
$A^\prime$ to 
give two disks $D_1$ and $D_2$. Identify the segments of the
boundaries of $D_1$ and $D_2$ labelled $A$ to form a new disk $D'$
with boundary labelled with the quadratic word 
\[W'= E_1\cdots E_r A' E_{j2} \cdots E_s (E_{r+1} \cdots
 E_{j1} A'^{-1})^{-\d} E_{s+1} \cdots E_t,\]
as shown in Figures \ref{fig:ocut1b} and \ref{fig:ocut1c}, for 
the cases $\d=-1$ and $\d=1$, respectively. 
By considering links of vertices
we see that $W'$ is irredundant unless it contains, as a subword, 
two occurrences
of 
either $(E_sE_{r+1})^{\pm 1}$, in the case $\d=-1$; or 
$(E_{r+1}^{-1}E_{s+1})^{\pm 1}$, in the case $\d=1$. 
If this is the case, say $E_aE_b=E_{r+1}^{-1}E_s^{-1}$, 
then we replace the subword $E_sE_{r+1}$ by a new letter $E_u$ and
 the subword $E_aE_b$ by $E_u^{-1}$. All other possibilities
are dealt with similarly.  Thus we may assume that 
$W'$ is an irredundant quadratic  word. 
As the surface $\S_W$, obtained by identifying edges of $D$, 
according to their  labels, 
is the same as the surface $\S_{W'}$, the quadratic word $W'$ is 
a genus $n$ Wicks form. 
Let $\cL=\supp(W')$
and  let $t_1$ be an $H$-minimal word in $F(X)$, 
such that $t_1=_H a_2\theta(E_{r+1}\cdots E_{j-1})e_{j1}$. 
 Define a  homomorphism $\psi:F(\cL)\to F(X)$
in the following way.
\begin{equation*}
\psi(E)= \left
\{\begin{array}{ll}
a_1t_1 & \textrm{if $E=A'$}\\
e_{j1} & \textrm{if $E=E_{j1}$}\\
e_{j2} & \textrm{if $E=E_{j2}$}\\[.5em]
e_{j\mu(\xi)}^\xi & \textrm{if $E=E_{i1}$}\\[.5em]
e_{j\nu(\xi)}^\xi & \textrm{if $E=E_{i2}$}\\[.5em]
\theta(E) & \textrm{otherwise.}
\end{array}
\right. ,
\end{equation*}
with the obvious adjustments if the word $W'$ had to be modified to 
make it irredundant.  
Let 
\begin{align*}
w_0&=\theta(E_1\cdots E_r)=\psi(E_1\cdots E_r),\\
w_1&=\theta(E_{r+1}\cdots E_{j-1})e_{j1}=\psi(E_{r+1}\cdots E_{j-1}E_{j1}),\\
w_2&=e_{j2}\theta(E_{j+1}\cdots E_s)=\psi(E_{j2}E_{j+1}\cdots E_s)\textrm{ and}\\
w_3&=\theta(E_{s+1}\cdots E_t)=\psi(E_{s+1}\cdots E_t).
\end{align*}
Then, writing $E_i$ for $E_{i1}E_{i2}$, as before, so $\psi(E_i)
=(e_{j1}e_{j2})^{\g\xi}=e_j^{\g\xi}$, we have 
\begin{align*}
\psi(W') & =_H \psi(E_1\cdots E_rA'E_{j2}\cdots E_s(E_{r+1}\cdots
E_{j1}A'^{-1})^{-\d}E_{s+1}\cdots E_t)\\
         & =_H w_0a_1t_1w_2(w_1t_1^{-1}a_1^{-1})^{-\d}w_3
\end{align*}
and substituting $t_1=_H a_2w_1$ in this expression gives 
\[\psi(W')=_H w_0a_1a_2w_1w_2(a_2^{-1}a_1^{-1})^{-\d}w_3=_H \theta(W).\]
This implies that $\psi(W')$ is conjugate to $h$ in $H$. Thus
$(W',\psi)$ is an element of $\mathcal{F}$. Now since $|\theta(W)|$ was chosen to
be minimal over all pairs in $\mathcal{F}$, it follows that
\begin{eqnarray}\label{psitheta1}
|\psi(W')| & \geq & |\theta(W)|.
           \end{eqnarray}
Also, as $\theta$ is cancellation free on $W$, we  have
\begin{align}
|\theta(W)| &= |w_0|+  |w_1|+  |w_2|+  |w_3|+ 2|a_1|+2|a_2|\nonumber\\
&= (|w_0|+  |w_1|+  |w_2|+  |w_3|+ 2|a_1|+2|t_1|)+2|a_2|-2|t_1|\nonumber\\
 & \ge  |\psi(W')|+2|a_2|-2|t_1|.\nonumber
\end{align}
That is
\begin{equation}\label{psitheta2}
|\theta(W)|\ge  |\psi(W')|+2|a_2|-2|t_1|.
\end{equation}
Inequalities \eqref{psitheta1} and \eqref{psitheta2} imply  
\begin{equation}
|a_2\theta(E_{r+1}\cdots E_{j-1})e_{j1}|_H  = |t_1|\geq  |a_2|,                
\end{equation}
as required.
 Hence  \ref{it:l1-2} holds.

The same argument, using $(W^{-1},\theta)\in \mathcal{F}$ and the cyclic permutation
 of $W^{-1}$ beginning with $E_S^{-1}$, can be used
to show  that \ref{it:l1-1} holds. 

Next we  consider case
\ref{it:l1-3}. 
Again we label the disk $D$ with the word $W$. 
This time, bisect the edge labelled $E_k$. The first half shall be labelled by
$E_{k1}$ and the second half labelled by $E_{k2}$, where $E_{k1},
E_{k2}$ are elements of $ \mathcal{A}^{\pm 1}$ not occurring in $W$. 
As before there is an edge 
$E_i=E_k^{\eta}$, $\eta=\pm 1$, 
and we replace $E_i$ with $E_{i1}E_{i2}$,  where $E_{i1}=E_{k\mu(\eta)}^\eta$ 
and $E_{i2}=E_{k\nu(\eta)}^{\eta}$.  
Cut along an arc  $A''$ joining $\tau(A)$  to  $\tau(E_{k1})$, 
and identify the two edges labelled $A$, 
as in  Figure \ref{fig:ocut2}. 
\begin{figure}[htp]
\begin{center}
\psfrag{E1}{{\scriptsize $E_1$}}
\psfrag{A}{\scriptsize $A$}
\psfrag{Ad}{\scriptsize $A^\d$}
\psfrag{A''}{\scriptsize $A''$}
\psfrag{Er}{{\scriptsize $E_{r}$}}
\psfrag{Er1}{{\scriptsize $E_{r+1}$}}
\psfrag{Ej-1}{{\scriptsize $E_{k-1}$}}
\psfrag{Ej1}{{\scriptsize $E_{k1}$}}
\psfrag{Ej2}{{\scriptsize $E_{k2}$}}
\psfrag{Ej+1}{{\scriptsize $E_{k+1}$}}
\psfrag{Es}{{\scriptsize $E_{s}$}}
\psfrag{Es+1}{{\scriptsize $E_{s+1}$}}
\begin{subfigure}[b]{.32\columnwidth}
\begin{center}
\includegraphics[scale=0.3]{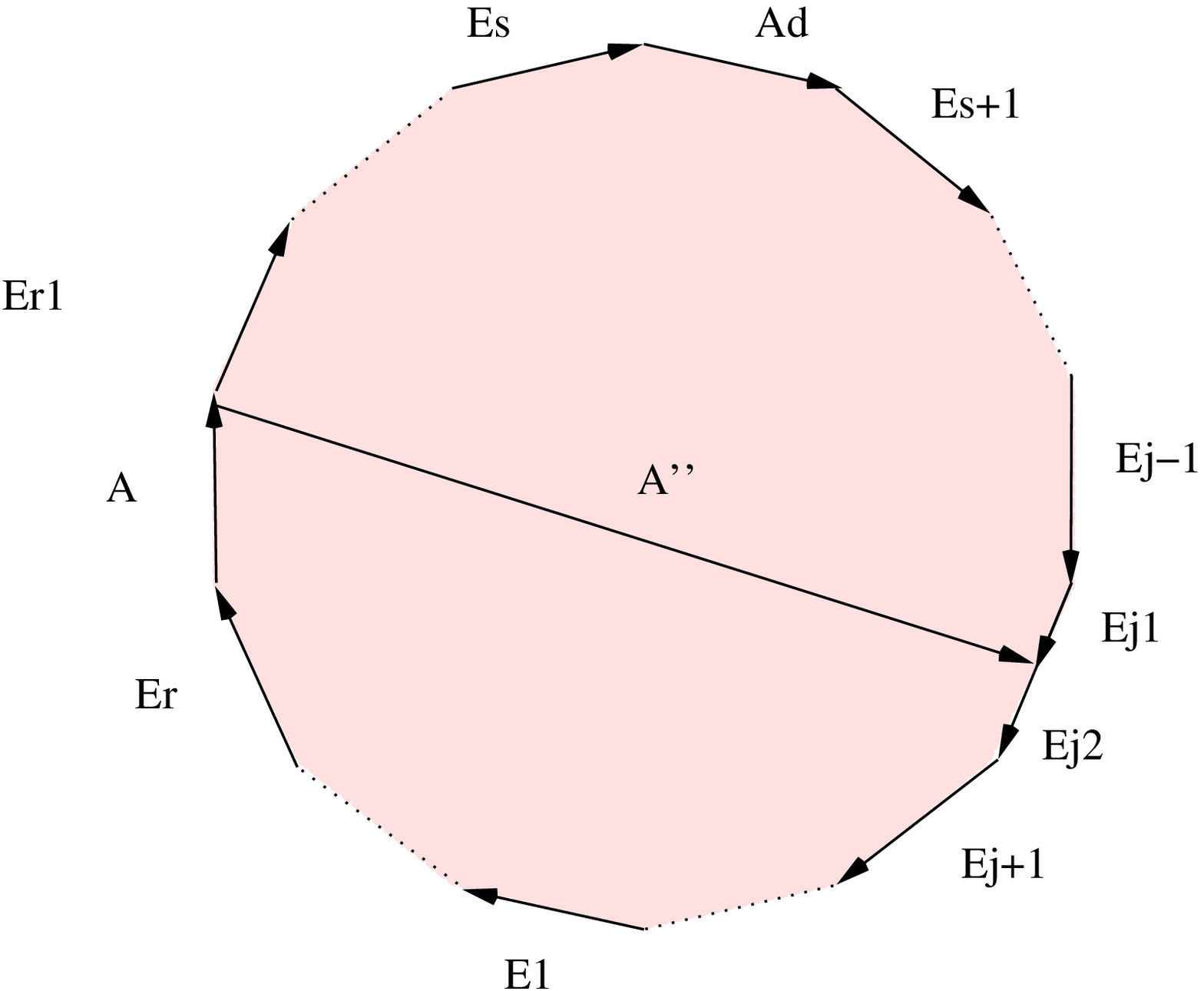}
\caption{}
\label{fig:ocut2a}
\end{center}
\end{subfigure}
\begin{subfigure}[b]{.32\columnwidth}
\begin{center}
\includegraphics[scale=0.3]{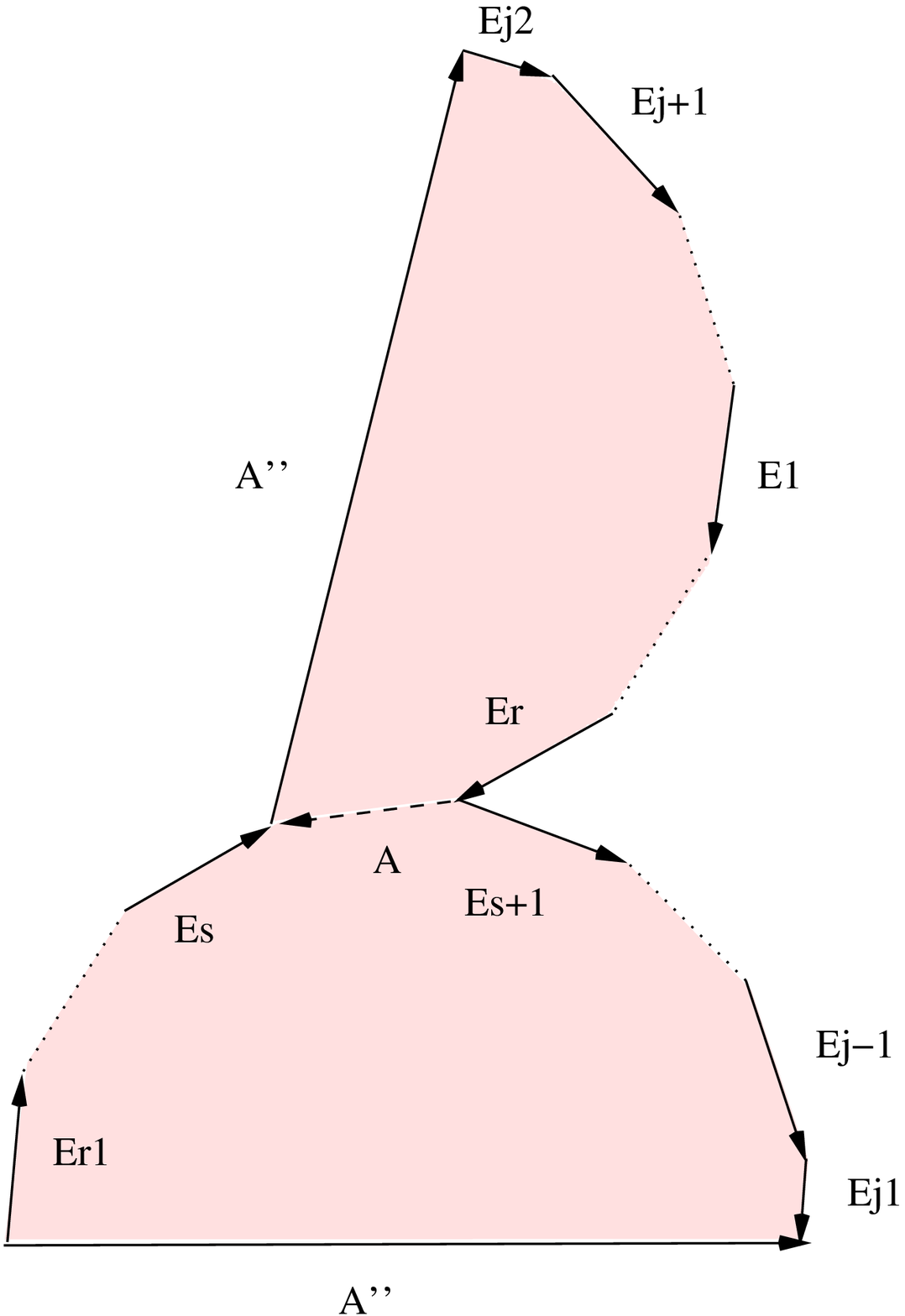}
\caption{$\d=-1$}
\label{fig:ocut2b}
\end{center}
\end{subfigure}
\begin{subfigure}[b]{.32\columnwidth}
\begin{center}
\includegraphics[scale=0.3]{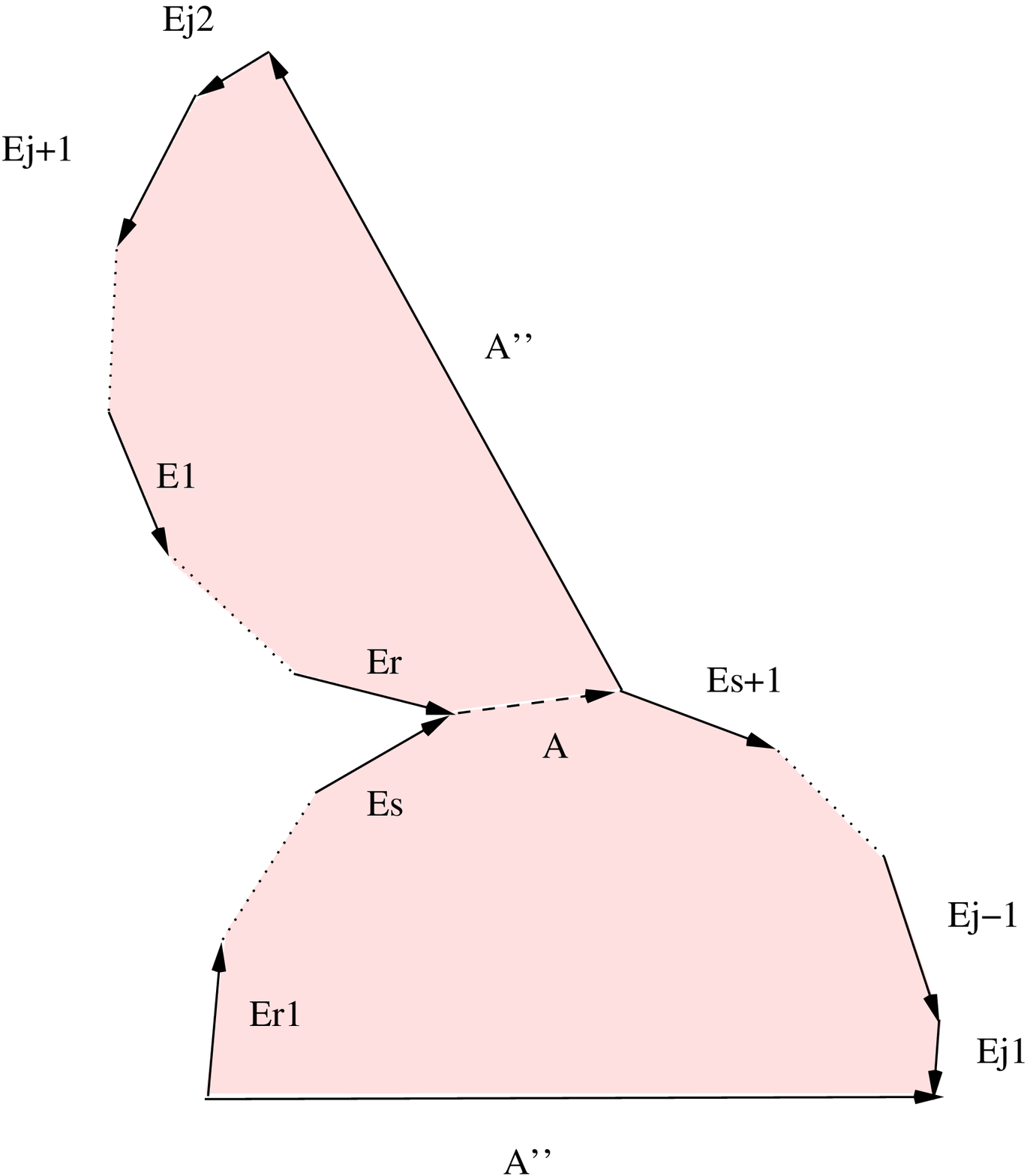}
\caption{$\d=1$}
\label{fig:ocut2c}
\end{center}
\end{subfigure}
\caption{}\label{fig:ocut2}
\end{center}
\end{figure} 
As before, if $\iota(A)$ in $\Gamma _W$ has degree $3$, and $\d=-1$ 
then we also
replace both occurrences the subword $E_{r}E_{s+1}$ 
(or its inverse) with a new edge $E_u\in
\mathcal{A}$ (or its inverse). A similar modification to the edge path
$E_rE_s^{-1}$ may be necessary when $\d=1$.  
The boundary of the new disk is now labelled 
by the irredundant  quadratic word 
\begin{equation*}
W''=E_1\cdots  E_r (E_{s+1} \cdots E_{k1} {A''^{-1}} E_{r+1}\cdots
 E_s)^{-\d} {A''} E_{k2} \cdots  E_t,
\end{equation*}
(replacing $ E_r E_{s+1}$ by $E_u$ if necessary) 
 which is 
a genus $n$ Wicks form. We shall define a labelling function
for $W''$. Let $\cK=\supp(W'')$ and let 
$t_2$ be an $H$-minimal word in $F(X)$ such that 
\begin{equation}
t_2=_H
a_2\theta(E_{r+1}\cdots E_s)(a_1 a_2)^\d\theta(E_{s+1}\cdots E_{k-1})e_{k1}.
\end{equation}
Define a  homomorphism $\phi:F(\cK)\to F(X)$
by 
\begin{equation*}
\phi(E) =\left
\{\begin{array}{ll}
a_2^{-1}t_2 & \textrm{if $E=A''$}\\
e_{k1} & \textrm{if $E=E_{k1}$}\\
e_{k2} & \textrm{if $E=E_{k2}$}\\[.5em]
e_{k\mu(\eta)}^\eta & \textrm{if $E=E_{i1}$}\\[.5em]
e_{k\nu(\eta)}^\eta & \textrm{if $E=E_{i2}$}\\[.5em]
\theta(E) & \textrm{otherwise}.
\end{array}
\right. .
\end{equation*} 

As before we find that $\phi(W'')=_H \theta(W)$, 
\begin{eqnarray}\label{psitheta5}
|\phi(W'')| & \geq & |\theta(W)|
           \end{eqnarray}
and 
\begin{eqnarray}\label{psitheta6}
|\theta(W)|
            & \geq & |\phi(W'')|-2|t_2|+2|a_1|.
\end{eqnarray}
Inequalities (\ref{psitheta5}) and (\ref{psitheta6}), imply that
\begin{equation*}
|a_2\theta(E_{r+1}\cdots E_s)a_2^{-1}a_1^{-1}\theta(E_{s+1}\cdots E_{k-1})e_{k_1}|_H  = |t_2| \geq  |a_1|,               
\end{equation*}
as required.
 Hence \ref{it:l1-3} holds. 

For \ref{it:l1-4} we modify the 
arguments above as follows. As we are assuming $\e=1$  we 
have $\d=1$. First consider the case where 
$i=1$, and follow the proof of case \ref{it:l1-2}, but  take 
$j=s$ and $E_{j1}=E_s$, $E_{j2}=1$. 
Then 
\[W'= E_1\cdots E_r A^{\prime\, 2} E_{s}^{-1} \cdots
 E_{r+1}^{-1} E_{s+1} \cdots E_t.\]
Let $E_0$ be an element of $\cA$ which does not belong to $\supp(W')$ and
 replace $W'$ with 
\[W'_0= E_1\cdots E_r E_0 A^{\prime\, 2} E_0^{-1} E_{s}^{-1} \cdots
 E_{r+1}^{-1} E_{s+1} \cdots E_t.\]
The surface obtained from identifying boundary intervals of a disk 
labelled by $W'_0$ is again $\S_W$, (as can be seen by considering
the link of the vertex incident to $A'$.  
As $|a_1|\ge |a'_1|$, we have $a_1=a'_1a_3$, for 
some terminal subword $a_3$ of $a_1$. 
Define $t_1=_H a_2\theta(E_{r+1} \cdots E_{s})a'_1$, 
let $\cK'=\cK\cup\{E_0\}$ and define $\psi: F(\cK')\maps F(X)$ by   
\begin{equation*}
\psi(E)= \left
\{\begin{array}{ll}
a_3t_1 & \textrm{if $E=A'$}\\
a'_1 & \textrm{if $E=E_0$}\\
\theta(E) & \textrm{otherwise}
\end{array}
\right. ,
\end{equation*}
$w_0$ and $w_3$ as before and 
\[w_1=\theta(E_{r+1}\cdots E_s)=\psi(E_{r+1}\cdots E_s).\]
Then $t_1=_H a_2w_1a'_1$ so 
\begin{align*}
\psi(W'_0) & =_H w_0a'_1a_3t_1a_3t_1 a^{\prime \,-1}_1 w_1^{-1}w_3\\
&=_H w_0a_1a_2w_1a_1a_2w_1a'_1 a^{\prime \,-1}_1 w_1^{-1} w_3\\
& =_H w_0a_1a_2w_1a_1a_2 w_3\\
&=_H \theta(W).
\end{align*}
As before, $W'$ is a genus $n$ Wicks form so $|\theta(W)|\le |\phi(W'_0)|$. 
 Thus 
\[|w_0|+|w_1|+|w_3|+2|a_1|+2|a_2|=\theta(W)\le \phi(W'_0)
\le |w_0|+|w_1|+|w_3|+ 2|a'_1|+2|a_3|+2|t_1|,\]
which 
 implies that $|t_1|\ge |a_2|$, so 
\ref{it:l1-4} holds. 

To prove that \ref{it:l1-5} holds, consider again the 
word $W''$, but this time take $k=s+1$, $E_{k1}=1$ and $E_{k2}=E_{s+1}$. 
Then 
\[W''= E_1\cdots  E_r  E_s^{-1} E_{r+1}^{-1} \cdots
 A^{\prime\prime\, 2} E_{s+1} \cdots  E_t.\]
Let $E_0$ be an element of $\cA$ which does not belong to $\supp(W'')$ and
 replace $W''$ with
\[ W''_0 =E_1\cdots  E_r  E_s^{-1}\cdots E_{r+1}^{-1} 
 E_0^{-1} A^{\prime\prime\, 2}E_0 E_{s+1} \cdots  E_t.\]
The surface obtained from identifying boundary intervals of a disk 
labelled by $W''_0$ is again $\S_W$. 
 As $|a_2|\ge |a'_2|$, we have $a_2=a_4a'_2$, for 
some initial subword $a_4$ of $a_2$, so $a'_1=a_1a_4$. 
Again, let $t_1=_H a_2\theta(E_{r+1} \cdots E_{s})a'_1$, 
 and  
define 
\begin{equation*}
\phi(E)= \left
\{\begin{array}{ll}
a_4^{-1}t_1& \textrm{if $E=A''$}\\
a^{\prime}_2 &\textrm{if $E=E_0$}\\ 
\theta(E) & \textrm{otherwise}
\end{array}
\right. ,
\end{equation*}
and $w_0,w_1,w_3$ as in case \ref{it:l1-4}. Then 
\begin{align*}
\phi(W''_0) &=_H w_0w_1^{-1} a_2^{\prime\,-1} (a_4^{-1} t_1)^2 a'_2 w_3 \\
&= w_0 w_1^{-1} a_2^{\prime\,-1} (a_4^{-1} a_2 w_1a'_1)^2 a'_2 w_3 \\
&= w_0 a'_1 a_4^{-1} a_2 w_1 a'_1 a'_2 w_3\\
&= \theta(W).
\end{align*} 

Again, $W''_0$ is a genus $n$ Wicks form so  $|\theta(W)|\le |\phi(W''_0)|$. 
Thus 
\[|w_0|+|w_1|+|w_3|+2|a_1|+2|a_2|=\theta(W)\le \phi(W''_0)
\le |w_0|+|w_1|+|w_3|+ 2|a'_2|+2|a_4|+2|t_1|,\]
 and consequently $|t_1|\ge |a_1|$, so 
\ref{it:l1-5} holds. 
\end{proof}

Next we record as corollaries some straightforward consequences of
Lemma \ref{lem:l1}.
\begin{corollary}\label{cor:N0}
Let $(W,\theta)$ be a cancellation free pair in $\cF$ and let 
$F$ be a geodesic for $(W,\theta)$. Let $A$ be a long edge of $W$,
with $\s(A)=(\e,\d)$. Let   $\a=\theta(A^{\e})$ and  
$\a'=\theta(A^{\d})$, let 
$x$ be a vertex of $\a$ and $k$ be a positive integer such that 
$k<d(\i(\a),x)<|\a|-k$. Then the following hold.
\be[(i)]
\item\label{it:N01} Let  $y$ be a vertex on $(\theta(W)\cup F)-\a$ such
that $d(x,y)\le k$. Then $y$ lies on $\a'\cup F$ and if 
$\d=\e$ then $y$ lies on $F$.  
\item\label{it:N02} If $k>l$ then there exists a  vertex $y$ such that
$y$ lies on $\a'\cup F$ and $d(x,y)\le l$. If $\d=\e$ then $y$ lies on $F$. 
\ee
\end{corollary}
\begin{proof}
Statement \ref{it:N01} follows from Lemma \ref{lem:l1}. For statement 
\ref{it:N02} note that by Lemma \ref{lem:la}.\ref{it:nla1}, there exists a vertex
$y$ on $\theta(W)-\a$ such that $d(x,y)\le l$. As $k>l$ the statement
now follows from \ref{it:N01}.
\end{proof} 

\begin{corollary}\label{cor:x1x2}
 Let $(W,\theta)$,  
$F$, $A$, $\a$ and $\a'$ be as in Corollary \ref{cor:N0}. 
Let $x_1$ and $x_2$  be any vertices on $\alpha$ and
$\a'$, respectively, such that \mbox{$d(x_1,x_2)\leq k$} for
some constant $k$. If $x_3$ is a vertex on $\a'$ such
that $d(\iota(\a'),x_3)=d(\tau(\alpha ),x_1)$
then $d(x_2,x_3)\leq k$ and $d(x_1,x_3)\le 2k$.
\end{corollary}
\begin{proof}
The proof falls into the following two cases.
\begin{enumerate}
\item[(a)] $d(\iota(\a'),x_2)\leq d(\iota(\a'),x_3)=d(\tau(\alpha
  ),x_1)$.
\item[(b)]  $d(\iota(\a'),x_2)> d(\iota(\a'),x_3)=d(\tau(\alpha
  ),x_1)$.
\end{enumerate}
(a) By Lemma \ref{lem:l1} $d(\tau(\alpha ),x_1)\leq
d(x_1,\iota(\a'))$ and, from the hypothesis and the
triangle inequality,  it follows that
\begin{equation}\label{alp1}
d(\tau(\alpha ),x_1)\leq 
d(x_1,\iota(\a')) \leq  k+d(\iota(\a'),x_2).
\end{equation} 
Since $\alpha$ is a geodesic path we have  
\begin{equation}\label{alp2}
d(\tau(\alpha),x_1)=d(\iota(\a'),x_3) = d(\iota(
\a'),x_2) + d(x_2,x_3).
\end{equation}
It follows from equations (\ref{alp1}) and (\ref{alp2}) that $d(x_2,x_3)\leq k$.

{\flushleft{ (b)}}  By Lemma \ref{lem:l1}, $d(\iota(\a' ),x_2)\leq
d(x_2,\tau(\alpha ))$ and, from the hypothesis and the
triangle inequality,  it follows that
\begin{equation}\label{alp3}
d(\iota(\a' ),x_2)\leq d(x_2,\tau(\alpha)) \leq  k+d(\tau(\alpha),x_1).
\end{equation} 
Since $\alpha$ is a geodesic path we have  
\begin{equation}\label{alp4}
d(\tau(\alpha),x_1)=d(\iota(\a'),x_3) = d(\iota(
\a'),x_2) - d(x_2,x_3).
\end{equation}
It follows from equations (\ref{alp3}) and (\ref{alp4}) that $d(x_2,x_3)\leq k$.
Hence the lemma holds.
\end{proof}

\begin{corollary}\label{cor:N3}
Let $(W,\theta)$,  
$F$, $A$, $\a$ and $\a'$   be as in Corollary \ref{cor:N0}. 
Assume in addition that $o(A)=1$, so $\s(A)=(\e,-\e)$ and $\a'=\a^{-1}$.  
For $i=1,2$, let $u_i$ be vertices of $\a$ and $v_i$ be vertices
of $\a^{-1}$ such that 
\begin{itemize}
\item $d(\t(\a),u_1)\ge d(\t(\a),u_2)$ and 
\item for some constant $k$, $d(u_i,v_i)\le k$, $i=1,2$.  
\end{itemize}
Let $\a_1$ be the subpath $[u_1,u_2]$ of $\a$. (See Figure \ref{fig:N3}.)  
Then $|\a_1|\le 2k+M+1$.
\end{corollary}
\begin{proof}
  Let $w_i$ be the vertex of $\a^{-1}$ such that $d(\t(\a^{-1}),w_i)
=d(\i(\a),u_i)$, $i=1,2$. Let $s_i$ be a geodesic from $u_i$ to $w_i$, 
for $i=1,2$. From Corollary \ref{cor:x1x2}, we have $|s_i|\le 2k$. The subpath
$[w_1,w_2]$ of $\a^{-1}$ is $\a_1^{-1}$, so from Lemma \ref{lb}, we have
\[|\a_1|\le \frac{1}{2}(|s_1|+|s_2|)+M+1\le 2k+M+1.\] 
\end{proof}

\begin{figure}[htp]
\begin{center}
\psfrag{A1}{{\scriptsize $\alpha _1$}}
\psfrag{s}{{\scriptsize $s_1$}}
\psfrag{t}{{\scriptsize $s_2$}}
\psfrag{F}{{\scriptsize $F$}}
\psfrag{u1}{{\scriptsize $u_1$}}
\psfrag{u2}{{\scriptsize $u_2$}}
\psfrag{w1}{{\scriptsize $w_1$}}
\psfrag{w2}{{\scriptsize $w_2$}}
\includegraphics[scale=0.5]{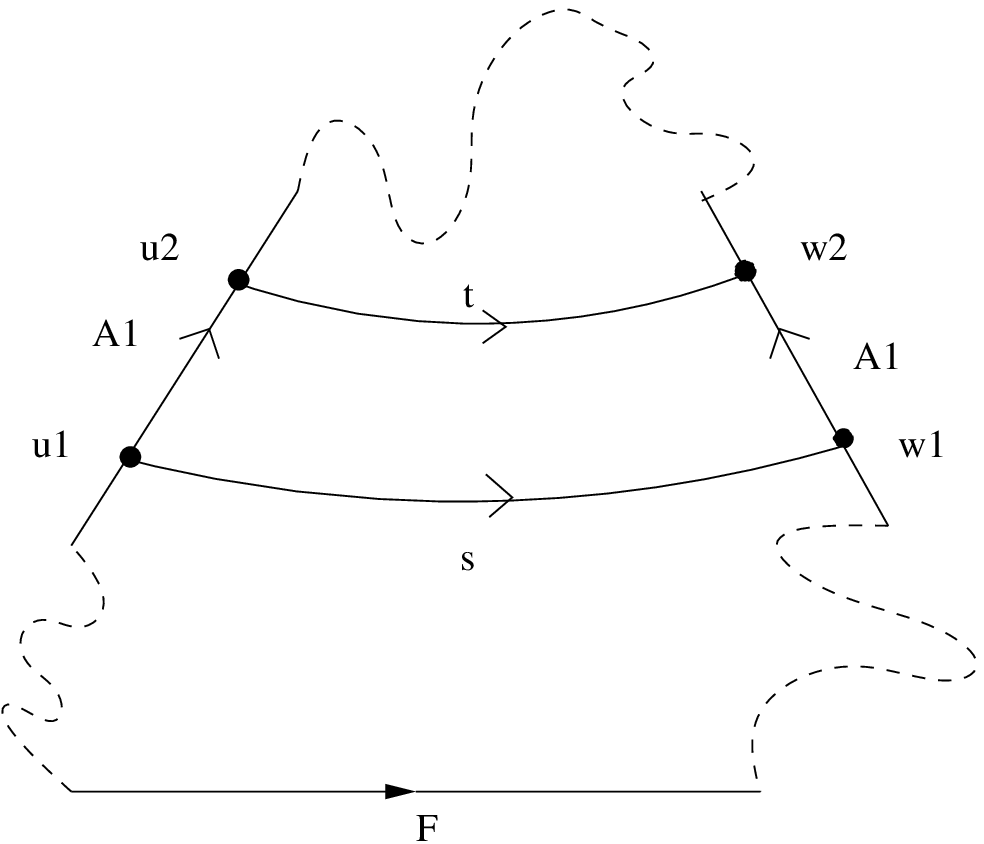}
\caption{}\label{fig:N3}
\end{center}
\end{figure} 

We shall use Corollaries \ref{cor:N0} and \ref{cor:N3} with Lemma \ref{lem:la}  
to factor $F$ into a product of subwords,
each of which corresponds to a unique long edge of $W$. To this end let
$W=E_1\cdots E_m$, where $E_i=A_i^{\e_i}$, with $A_i\in \cA$ and 
$\e_i\in \{\pm 1\}$. Let $U$ be the word obtained from $W$ by 
deleting $E_i$, whenever $E_i$ is a short edge. Then $U$ is a non-trivial 
quadratic word over $\cA$, of genus at most $n$, and length $f$ at most $K(n)$, which we call the \emph{long factorisation} of $W$ (with respect to $\theta$). 
Write $U=U_1\cdots U_f$,
where $U_i\in \{E_1,\ldots, E_m\}$, and $U_i$ is a long edge of $W$.
Then, as $W$ ends with a long edge,  we can write $W=S_1U_1S_2\cdots S_{f}U_f$, where each 
$S_i$ is a subword of $W$ consisting only of short edges (and may
be the empty word). 
By the usual abuse of notation, let 
$\theta(U_i)$ and $\theta(S_i)$ denote subpaths in $\G_X(H)$ of the path
labelled $\theta(W)$. 

\begin{lemma}\label{lem:chouv} 
 In the above notation, let $W=S_1U_1S_2\cdots S_{f}U_f$, where
 $U=U_1\cdots U_f$ is the long factorisation of $W$ and each $S_i$ is a product
of short edges. Then  
$F$ has a partition $F_1,\ldots, F_f$ such that 
\be
\item\label{it:chouv1} if $o(U_i)=-1$ then   $|F_i|\ge 8l+M+3$ and $d(\t(F_i),\t(\theta(U_i))\le 3l+1$; 
\item\label{it:chouv2} if $o(U_i)=1$ then  $|F_i|\ge 6l+2$ and, 
for the unique $j$ such that  $U_i=U_j^{-1}$, we have, assuming $i<j$, that 
$d(\t(F_i),\t(\theta(U_i))\le 5l+M+3$  and 
$d(\t(F_i),\t(\theta(U_j))\le 3l+1$,
\ee
for all $i$ such that $1\le i\le f$. 
\end{lemma} 
\begin{proof}
Let $A\in \supp(W)$ such that $A$ is a long edge and assume that $W=W_0U_iW_0U_jW_1$, where $U_i=A^{\e}$ and $U_j=A^{\d}$. 
Let $\a=\theta(A^{\e})$ and $\b=\theta(A^{\d})$. Then $|\a|=|\b|>12l+M+4$ and factorising as $\a=\a_0\a_1\a_2$, where
$|\a_0|=|\a_2|=2l+1$ we have $|\a_1|>8l+M+2$. 

Now, as above, $W=E_1\cdots E_m$, where $E_i=A_i^{\e_i}$, with $A_i\in \cA$ and 
$\e_i\in \{\pm 1\}$. Apply Lemma \ref{lem:la} to the word $\theta(E_1)\cdots \theta(E_m)F^{-1}$, with $\g_0=\theta(U_i)$. This implies 
 there is a partition $\g^{(1)},\ldots, \g^{(r)}$ of 
$\a$  such that, for each $j\in \{1,\ldots ,r\}$, 
the interval $\g^{(j)}$ is a $\d\lceil \log_2(m)\rceil$-fellow traveller with a geodesic $\g'$, which is  either an interval of  
a partition of $\theta(E_s)$,  
for some $E_s\neq U_i$, or of  $F^{-1}$. As $\d\lceil \log_2(m)\rceil\le l$, it follows that   $\g^{(j)}$ and $\g'$ are $l$-fellow travellers. 
 
   Assume that the partition of $F^{-1}$ given by Lemma \ref{lem:la}
is $F'_1,\ldots ,F'_t$. Suppose first that $o(A)=-1$ and  $\s(A)=(\e,\e)$. Then, as $|a_0|=|a_2|=2l+1$, 
it follows from Corollary \ref{cor:N0}, that 
$\a_1$ is  a subinterval of $\g^{(b)}$, where $1\le d\le r$ and $\g^{(b)}$ 
is an $l$-fellow traveller with $F'_c$, for some $c$, and that $|F'_c|\ge |a_1|\ge 8l+M+3$. 
In this case, a similar argument shows that a subinterval  $\g^{(b')}$ of $\b$ 
is an $l$-fellow traveller with $F'_{c'}$, for some $c'<c$, with  $|F'_c|\ge 8l+M+3$.

This leaves the case  $o(A)=1$ and  $\s(A)=(\e,-\e)$ to consider. In this case, from Corollary \ref{cor:N0}, $\a_1$ is a subinterval
of $\g^{(b)}\g^{(b+1)}$, where $\g^{b}$ is an $l$-fellow traveller with an interval $F'_c$ of ${F'}^{-1}$ and  $\g^{(b+1)}$ is 
an $l$-fellow traveller with an interval ${\g'}^{(d)}$ of the partition ${\g'}^{(1)},\ldots, {\g'}^{(r')}$  of $\b$ given by
Lemma \ref{lem:la}. Thus $|\g^{(b)}|+|\g^{(b+1)}|=|\g^{(b)}\g^{(b+1)}|\ge |\a_1|\ge 8l+M+3$. From Corollary \ref{cor:N3}, 
$|\g^{(b+1)}|\le 2l+M+1$, so $|F'_c|=|\g^{(b)}|\ge 6l+2$. Moreover, in this case the interval $\g'^{(d+1)}$ is an $l$-fellow
traveller with an interval $F'_{c'}$ of the partition of ${F'}^{-1}$, with $c'< c$, and again $|F'_{c'}|\ge 6l+2$.

Recalling that $\a=\theta(U_i)$ and $\b=\theta(U_j)$,  in both cases we denote the 
initial vertex $\i(F'_c)$ of $F'_c$  by $v_i$ and   the 
initial vertex $\i(F'_{c'})$ of $F'_{c'}$ of by $v_j$. 
By construction, if $o(A)=-1$,  then $d(v_i,\t(\theta(U_i)))\le 3l+1$ and if $o(A)=1$, then   
$d(v_i,\t(\theta(U_i)))\le 3l+1+2l+M+1=5l+M+2$, while in both cases  
$d(v_j,\t(\theta(U_j)))\le 3l+1$. 

Define $v_i$ and $v_j$ in this way for each long edge $A\in \supp(W)$. Setting $v_0=\t(F)$  results in $f+1$ points $v_0, v_1,\ldots ,v_f$ on 
$F$ such that $d(v_0,v_{i})>d(v_0,v_{i+1})$ and, with $F_{f-i}=[v_{i+1},v_{i}]$, we have a partition $F_1,\ldots ,F_f$ of $F$, where   $|F_i|\ge 6l+2$. 
\end{proof}

\begin{definition}\label{def:Ffact} 
With the notation of the paragraph above Lemma \ref{lem:chouv}, let $W=S_1U_1\cdots S_{f}U_f$, let $U=U_1\cdots U_f$ be the long factorisation of $W$,  
 let $F=F_1\cdots F_f$ be the factorisation of $F$ found in Lemma \ref{lem:chouv} and let $v(U_i)=\t(F_i)$.  
For $i=1,\ldots, l$, let $a(U_i)$ be the label of  a geodesic path from $\t(\theta(U_i))$ 
to $v(U_i)$, let $s(U_i)=\theta(S_{i})$  (and let $a(U_0)=1$ and $S_0=1$) and let $\a(U_i)=F_i$. 
For later reference let $S(U_i)=S_i$, for all $i$. 
Let $A\in \supp(U)$.
\be[(i)]
\item If $o(A)=1$,  let $i,j$ be such that $U_i=A$ and $U_j=A^{-1}$. 
\item If $o(A)=-1$ and $\s(A)=(\e,\e)$,   
let $i,j$ be such that $U_i=U_j=A^{\e}$, $i<j$. 
\ee 
Define 
\begin{gather*}
v_1(A)=v(U_i), \quad a_1(A)=a(U_i), \quad b_1(A)=a(U_{i-1}) ,\quad S_1(A)=S_{i},\quad s_1(A)=s(U_{i})\\
\textrm{ and } \a_1(A)=
\begin{cases}
\a(U_i),& \textrm{if } o(A)=1,\\
\a(U_i)^\e,& \textrm{if } o(A)=-1;
\end{cases}\\
v_2(A)=v(U_j), \quad a_2(A)=a(U_j), 
\quad  b_2(A)=a(U_{j-1}),\quad S_2(A)=S_{j},\quad  s_2(A)=s(U_{j})\\
\textrm{ and }\a_2(A)=
\begin{cases}
\a(U_j)^{-1},& \textrm{if } o(A)=1,\\
\a(U_j)^\e,& \textrm{if } o(A)=-1.
\end{cases}\\
\end{gather*}
(See Figure \ref{fig:Ffact}.) 

Finally define $H$-minimal words 
\begin{align*}
z_1(A)&=_H
\begin{cases} 
b_2(A)^{-1}s_2(A)a_1(A), & \textrm{if } o(A)=1,\\
a_2(A)^{-1}a_1(A), & \textrm{if } o(A)=-1,
\end{cases}
\textrm{ and }\\[.5em]
 z_2(A)&=_H 
\begin{cases} 
b_1(A)^{-1}s_1(A)a_2(A), & \textrm{if } o(A)=1,\\
b_1(A)^{-1}s_1(A)s_2(A)^{-1}b_2(A), & \textrm{if } o(A)=-1.
\end{cases}
\end{align*}
\end{definition}
\begin{figure}[htp]
\begin{center}
\begin{subfigure}[b]{.45\columnwidth}
\begin{center}
\psfrag{A}{{\scriptsize $\theta(A)$}}
\psfrag{A1}{{\scriptsize $\theta(A)$}}
\psfrag{A2}{{\scriptsize $\theta(A)$}}
\psfrag{a1}{{\scriptsize $a_1(A)$}}
\psfrag{b1}{{\scriptsize $b_1(A)$}}
\psfrag{b2}{{\scriptsize $b_2(A)$}}
\psfrag{a2}{{\scriptsize $a_2(A)$}}
\psfrag{s1}{{\scriptsize $s_1(A)$}}
\psfrag{s2}{{\scriptsize $s_2(A)$}}
\psfrag{v1}{{\scriptsize $v_1(A)$}}
\psfrag{v2}{{\scriptsize $v_2(A)$}}
\psfrag{f1}{{\scriptsize $\a_1(A)$}}
\psfrag{f2}{{\scriptsize $\a_2(A)$}}
\includegraphics[scale=0.35]{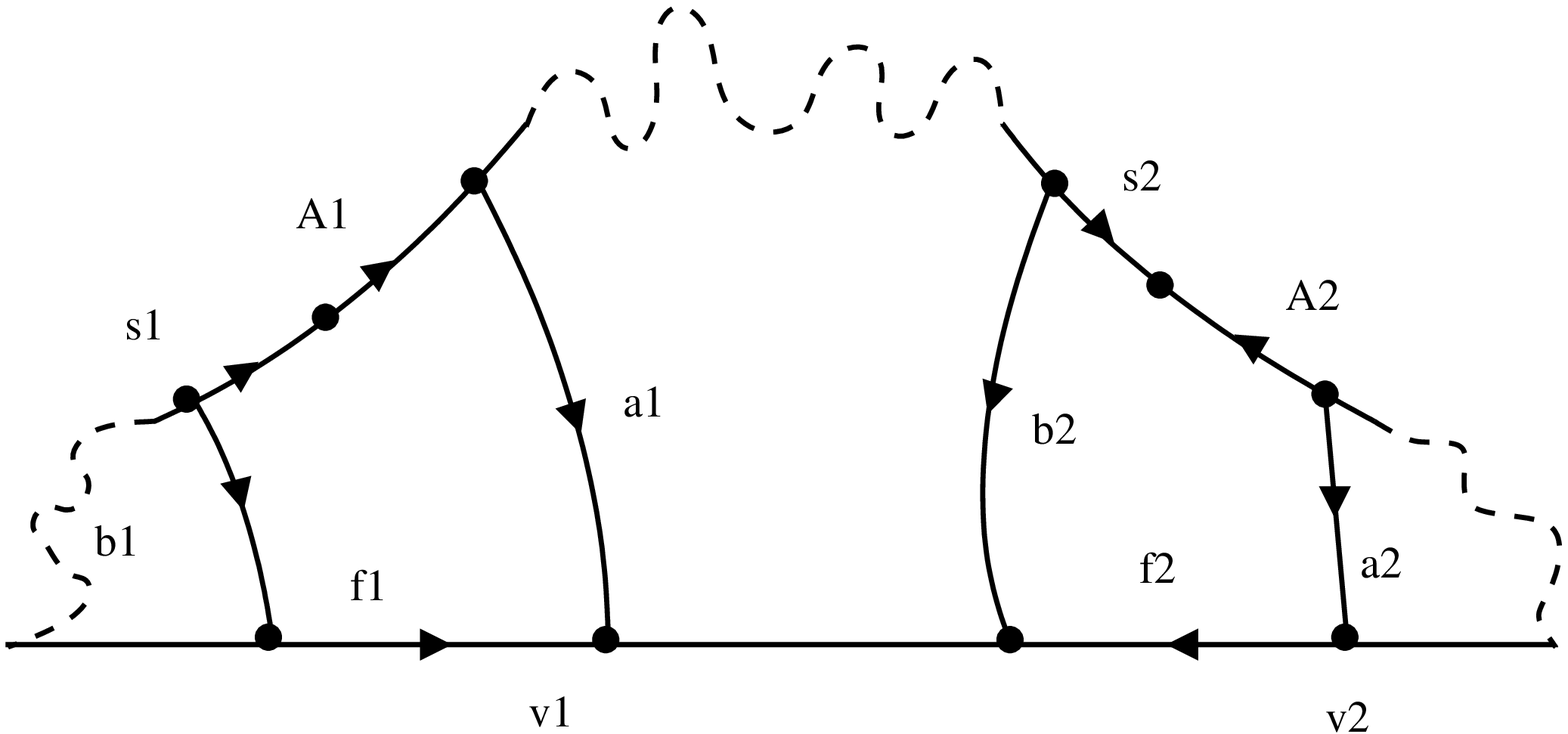}
\caption{$o(A)=1$, $A=U_i$, $A^{-1}=U_j$, $i<j$}\label{fig:Ffact1}
\end{center}
\end{subfigure}
\quad
\begin{subfigure}[b]{.45\columnwidth}
\begin{center}
\psfrag{A}{{\scriptsize $\theta(A)$}}
\psfrag{A1}{{\scriptsize $\theta(A)$}}
\psfrag{A2}{{\scriptsize $\theta(A)$}}
\psfrag{a1}{{\scriptsize $a_2(A)$}}
\psfrag{b1}{{\scriptsize $b_2(A)$}}
\psfrag{b2}{{\scriptsize $b_1(A)$}}
\psfrag{a2}{{\scriptsize $a_1(A)$}}
\psfrag{s1}{{\scriptsize $s_2(A)$}}
\psfrag{s2}{{\scriptsize $s_1(A)$}}
\psfrag{v1}{{\scriptsize $v_2(A)$}}
\psfrag{v2}{{\scriptsize $v_1(A)$}}
\psfrag{f1}{{\scriptsize $\a_2(A)$}}
\psfrag{f2}{{\scriptsize $\a_1(A)$}}
\includegraphics[scale=0.35]{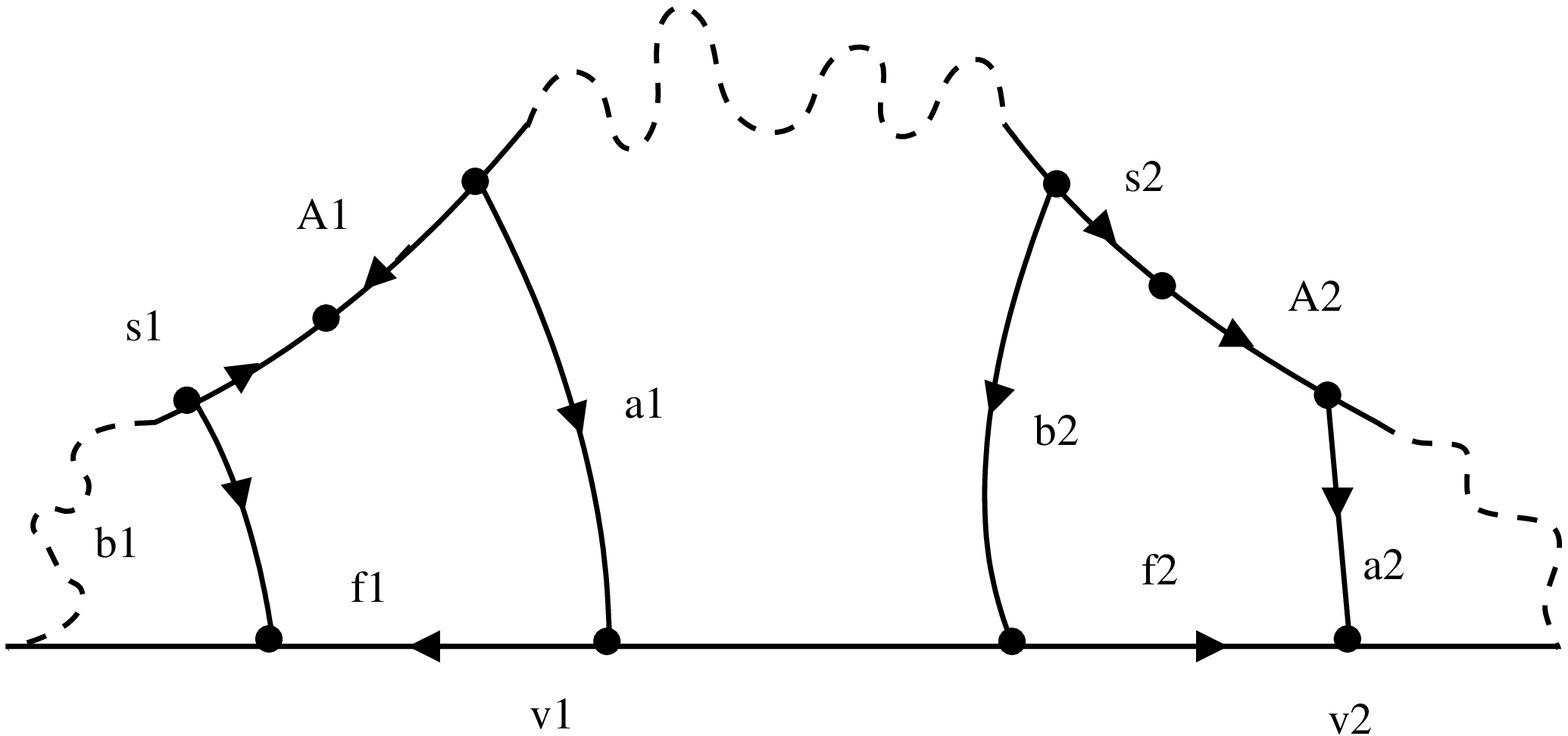}
\caption{$o(A)=1$, $A=U_i$, $A^{-1}=U_j$, $i>j$}\label{fig:Ffact2}
\end{center}
\end{subfigure}
~\\[1em]
\begin{subfigure}[b]{.45\columnwidth}
\begin{center}
\psfrag{A}{{\scriptsize $\theta(A)$}}
\psfrag{A1}{{\scriptsize $\theta(A)$}}
\psfrag{A2}{{\scriptsize $\theta(A)$}}
\psfrag{a1}{{\scriptsize $a_1(A)$}}
\psfrag{b1}{{\scriptsize $b_1(A)$}}
\psfrag{b2}{{\scriptsize $b_2(A)$}}
\psfrag{a2}{{\scriptsize $a_2(A)$}}
\psfrag{s1}{{\scriptsize $s_1(A)$}}
\psfrag{s2}{{\scriptsize $s_2(A)$}}
\psfrag{v1}{{\scriptsize $v_1(A)$}}
\psfrag{v2}{{\scriptsize $v_2(A)$}}
\psfrag{f1}{{\scriptsize $\a_1(A)$}}
\psfrag{f2}{{\scriptsize $\a_2(A)$}}
\includegraphics[scale=0.35]{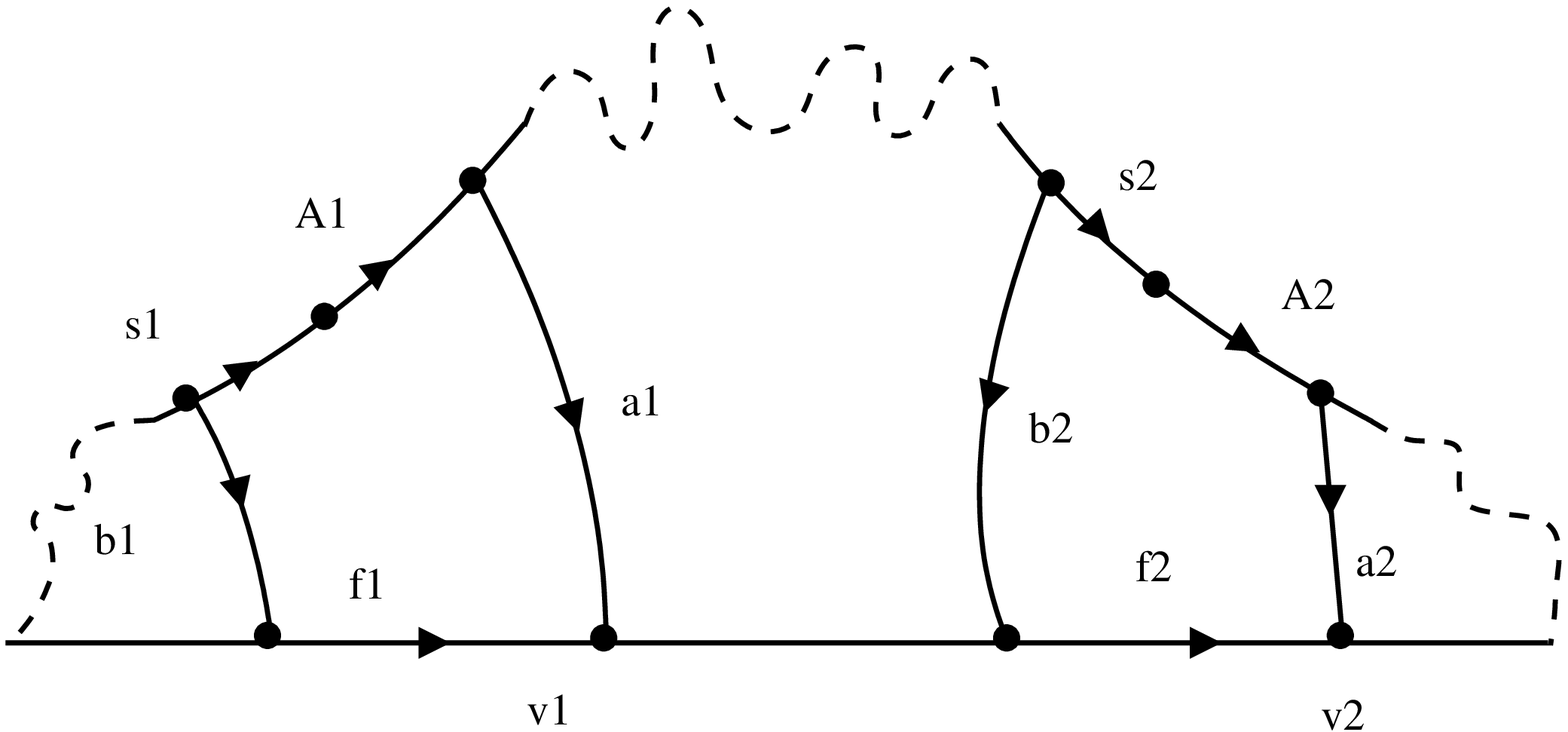}
\caption{$o(A)=-1$, $A=U_i=U_j$}\label{fig:Ffact3}
\end{center}
\end{subfigure}
\quad
\begin{subfigure}[b]{.45\columnwidth}
\begin{center}
\psfrag{A}{{\scriptsize $\theta(A)$}}
\psfrag{A1}{{\scriptsize $\theta(A)$}}
\psfrag{A2}{{\scriptsize $\theta(A)$}}
\psfrag{a1}{{\scriptsize $a_1(A)$}}
\psfrag{b1}{{\scriptsize $b_1(A)$}}
\psfrag{b2}{{\scriptsize $b_2(A)$}}
\psfrag{a2}{{\scriptsize $a_2(A)$}}
\psfrag{s1}{{\scriptsize $s_1(A)$}}
\psfrag{s2}{{\scriptsize $s_2(A)$}}
\psfrag{v1}{{\scriptsize $v_1(A)$}}
\psfrag{v2}{{\scriptsize $v_2(A)$}}
\psfrag{f1}{{\scriptsize $\a_1(A)$}}
\psfrag{f2}{{\scriptsize $\a_2(A)$}}
\includegraphics[scale=0.35]{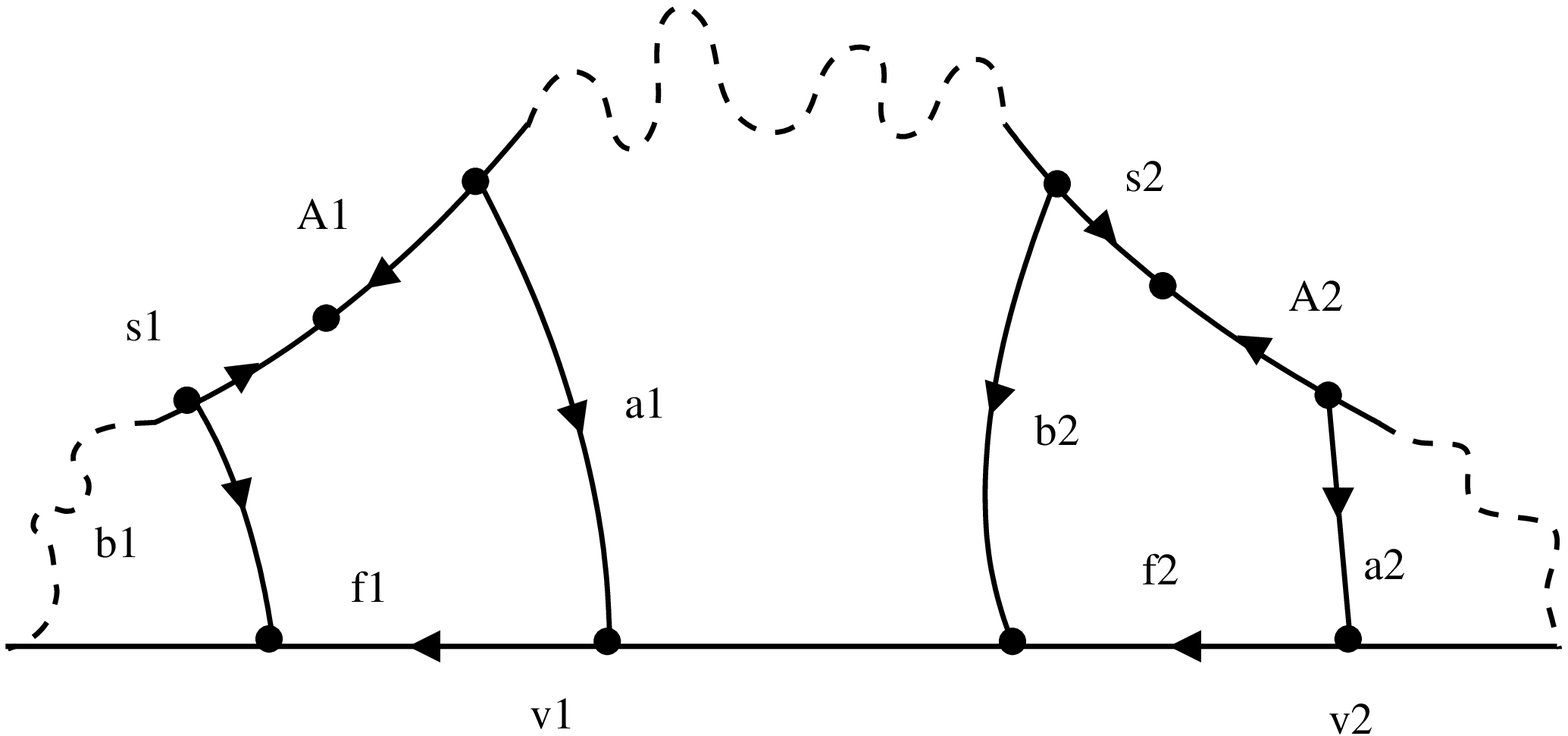}
\caption{$o(A)=-1$, $A^{-1}=U_i=U_j$}\label{fig:Ffact4}
\end{center}
\end{subfigure}
\end{center}
\caption{}\label{fig:Ffact}
\end{figure}  
\begin{example}\label{ex:U}
Suppose that ${W}=ABC^{-1}A^{-1}B^{-1}C^{-1}$ where $A$ and $C$ are
long edges and $B$ is a short edge. Then $U=AC^{-1}A^{-1}C^{-1}$
and 
we have the 
 paths in 
$\Gamma_X(H)$ shown in Figure \ref{exABC}. Here  
 $|a_1(A)|$, $|a_2(A)|$ and $|a_1(C)|$ are at most $5l+M+3$ and $|a_2(C)|=0$.
(As in our standing assumption, the last letter of $W$ is long).
\begin{figure}[htp]
\begin{center}
\psfrag{A}{{\scriptsize $\theta(A)$}}
\psfrag{C}{{\scriptsize $\theta(C)$}}
\psfrag{B}{{\scriptsize $\theta(B)$}}
\psfrag{a1}{{\scriptsize $a_1(A)$}}
\psfrag{c1}{{\scriptsize $a_1(C)$}}
\psfrag{a2}{{\scriptsize $a_2(A)$}}
\psfrag{A1}{{\scriptsize $\a_1(A)$}}
\psfrag{C1}{{\scriptsize $\a_1(C)$}}
\psfrag{A2}{{\scriptsize $\a_2(A)$}}
\psfrag{C2}{{\scriptsize $\a_2(C)$}}
\psfrag{u1}{{\scriptsize $v_1(A)$}}
\psfrag{v1}{{\scriptsize $v_1(C)$}}
\psfrag{u2}{{\scriptsize $v_2(A)$}}
\psfrag{v2}{{\scriptsize $v_2(C)$}}
\includegraphics[scale=0.6]{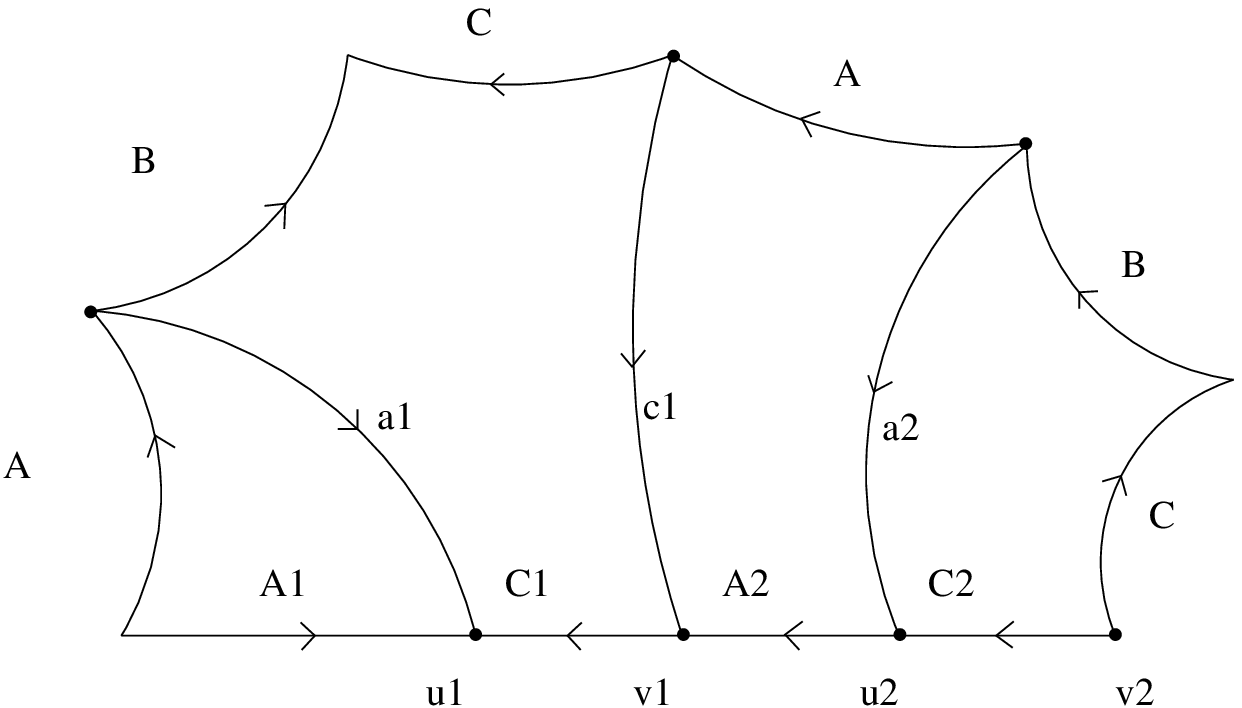}
\caption{}\label{exABC}
\end{center}
\end{figure}  
In this case $S_0=S_2= S_4=1$, $S_1=B$ and $S_3=B^{-1}$, so 
$b_1(A)=1$, $b_2(A)=a_1(C)$, $b_1(C)=a_1(A)$ and
$b_2(C)=a_2(A)$, $s_i(A)=1$, $s_1(C)=\theta(B)$ and
$s_2(C)=\theta(B)^{-1}$. As $o(A)=1$ and $o(C)=-1$, 
$z_i(A)$ and $z_i(C)$ are $H$-minimal
elements of $F(X)$ such that 
\[z_1(A)=(a_1(C))^{-1}a_1(A), \quad z_2(A)=a_2(A),\quad z_1(C)= 
a_1(C) \textrm{ and } z_2(C)=(a_1(A))^{-1}\theta(B)^2 a_2(A).\]
\end{example} 

We shall need the following lemma in the next subsection.
\begin{lemma}\label{lem:extedge}
Let $A$ be a long edge of $W$. Then 
\be[(i)]
\item \label{it:extedge1} 
$z_2(A)\a_2(A)z_1(A)\a_1(A)^{-1}=_H 1$, if $o(A)=1$, and 
\item\label{it:extedge2} 
$z_2(A)\a_2(A)^{\e}z_1(A)\a_1(A)^{-\e}=_H 1$,
 if $o(A)=-1$ and $\s(A)=(\e,\e)$. 
\ee
\end{lemma}
\begin{proof}
By definition, for $1\le i\le l$ the Cayley graph $\G_X(H)$ contains
 a closed path, based at $v(U_{i-1})$ on $F$, consisting of the 
concatenation of paths  labelled 
$\a(U_i)$, $a(U_i)^{-1}$, $\theta(U_i)^{-1}$, $\theta(S_{i-1})^{-1}$ and
$a(U_{i-1})$. (See Figure \ref{fig:glue}). Thus, for $A\in \supp(U)$ 
with $o(A)=1$, we have both  
\[\a_1(A)a_1(A)^{-1}\theta(A)^{-1}s_1(A)^{-1}b_1(A)=_H 1 \textrm{ and }
\a_2(A)^{-1}a_2(A)^{-1}\theta(A)s_2(A)^{-1}b_2(A)=_H 1.
\]
Similarly, for $A\in \supp(U)$ with $o(A)=-1$, we have 
\[\a_k(A)^\e a_k(A)^{-1}\theta(A)^{-\e}s_k(A)^{-1}b_k(A)=_H 1, \textrm{ for }
k=1,2.
\]
The result follows from these two facts and the definition of $z_i(A)$.
\begin{figure}[htp]
\begin{center}
\begin{subfigure}[b]{.45\columnwidth}
\begin{center}
\psfrag{A}{{\scriptsize $\theta(A)$}}
\psfrag{A1}{{\scriptsize $\a_1(A)$}}
\psfrag{A2}{{\scriptsize $\a_2(A)$}}
\psfrag{a1}{{\scriptsize $a_1(A)$}}
\psfrag{ba1}{{\scriptsize $b_1(A)$}}
\psfrag{ba2}{{\scriptsize $b_2(A)$}}
\psfrag{a2}{{\scriptsize $a_2(A)$}}
\psfrag{s1}{{\scriptsize $s_1(A)$}}
\psfrag{s2}{{\scriptsize $s_2(A)$}}
\psfrag{z1}{{\scriptsize $z_1(A)$}}
\psfrag{z2}{{\scriptsize $z_2(A)$}}
\includegraphics[scale=0.45]{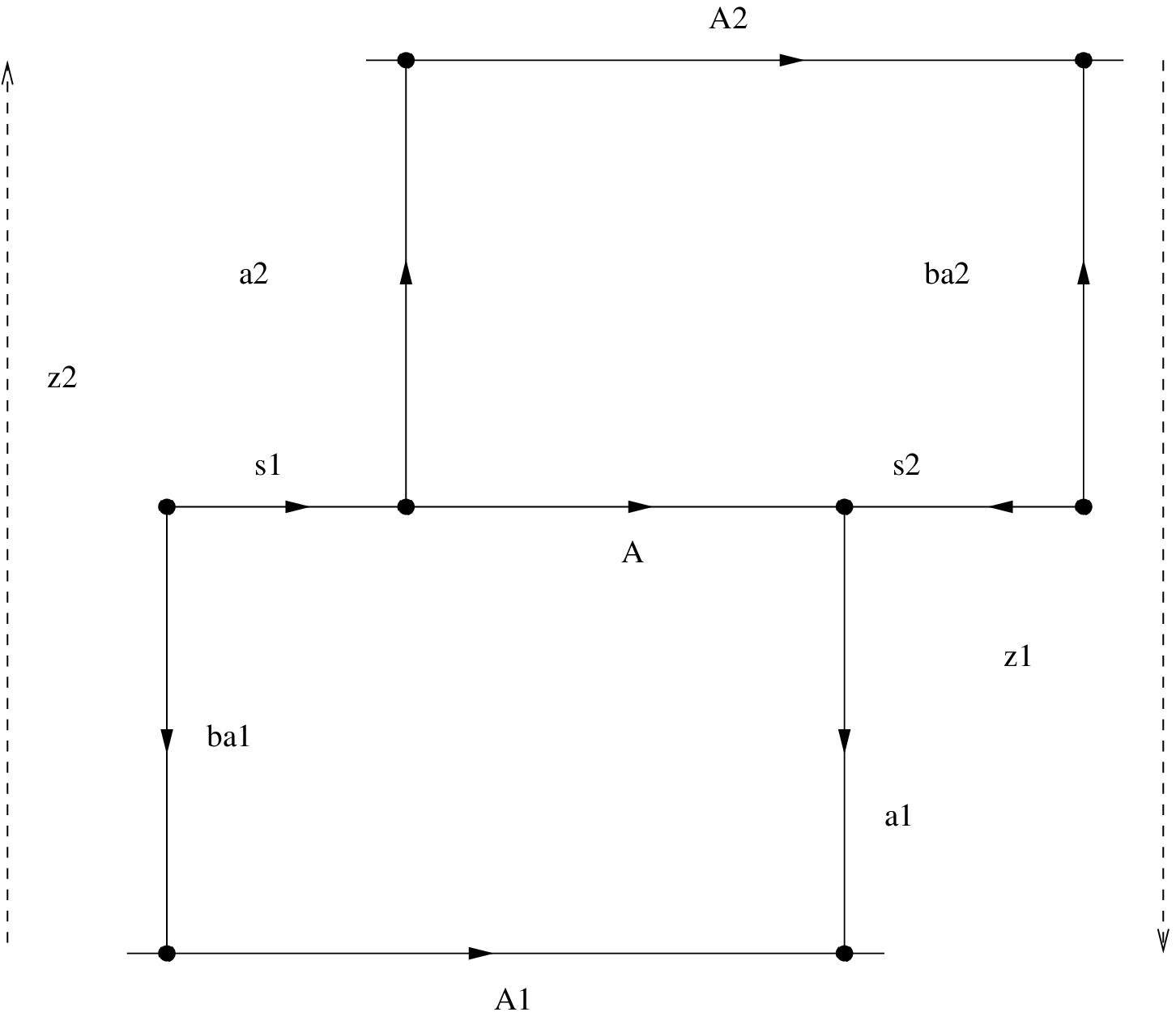}
\caption{$o(A)=1$, $A=U_i$, $A^{-1}=U_j$}\label{fig:glueo}
\end{center}
\end{subfigure}
\quad\quad
\begin{subfigure}[b]{.45\columnwidth}
\begin{center}
\psfrag{A}{{\scriptsize $\theta(A)^\e$}}
\psfrag{A1}{{\scriptsize $\a_1(A)^\e$}}
\psfrag{A2}{{\scriptsize $\a_2(A)^\e$}}
\psfrag{a1}{{\scriptsize $a_1(A)$}}
\psfrag{ba1}{{\scriptsize $b_1(A)$}}
\psfrag{ba2}{{\scriptsize $b_2(A)$}}
\psfrag{a2}{{\scriptsize $a_2(A)$}}
\psfrag{s1}{{\scriptsize $s_1(A)$}}
\psfrag{s2}{{\scriptsize $s_2(A)$}}
\psfrag{z1}{{\scriptsize $z_1(A)$}}
\psfrag{z2}{{\scriptsize $z_2(A)$}}
\includegraphics[scale=0.45]{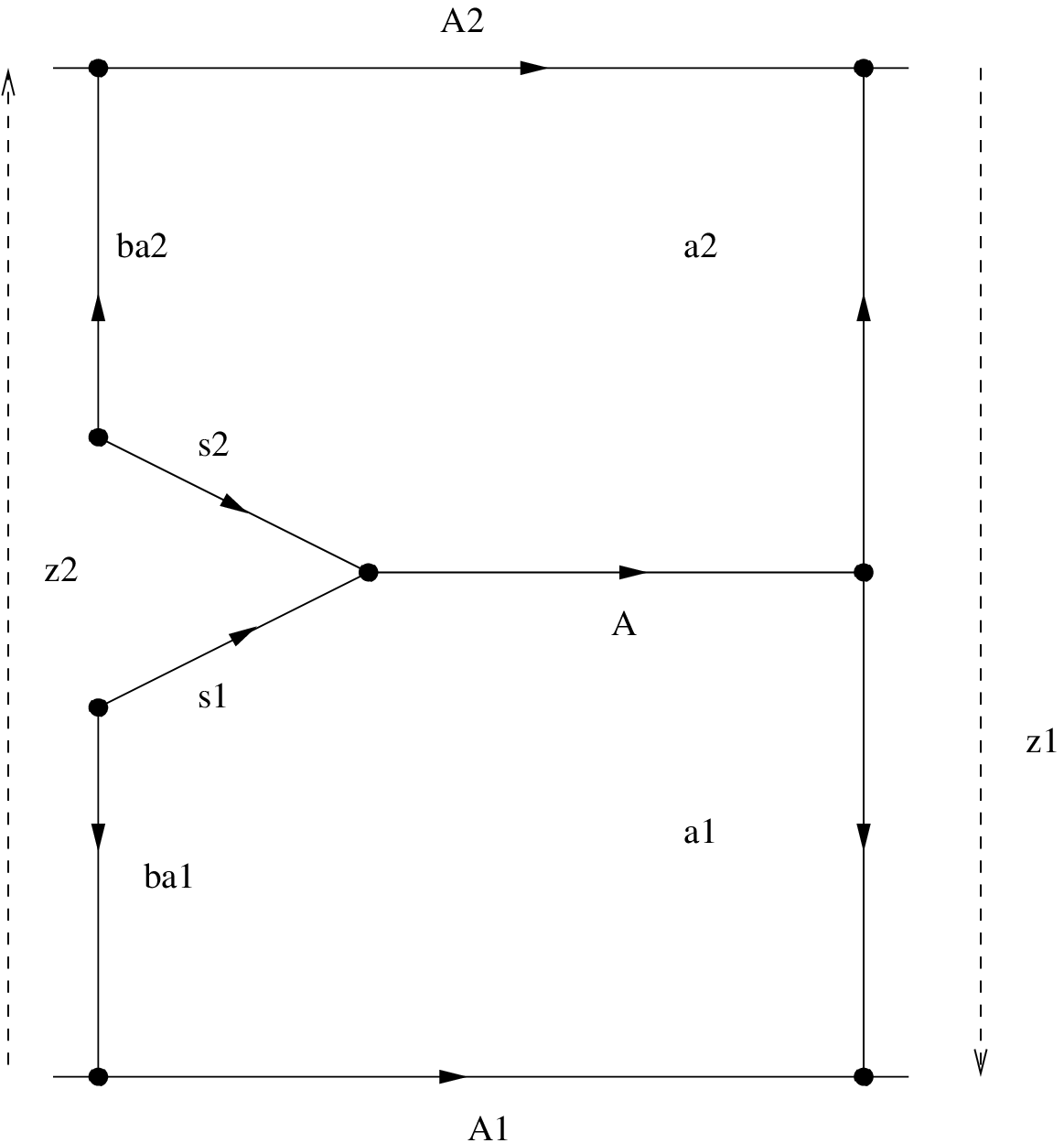}
\caption{$o(A)=-1$, $\s(A)=(\e,\e)$}\label{fig:gluen}
\end{center}
\end{subfigure}
\end{center}
\caption{}\label{fig:glue}
\end{figure}  
\end{proof}
\subsection{Construction of an extension of $U$}\label{stext}

Write ${W}$ around the boundary of a disk (i.e.~divide the
boundary up into $|{W}|$ segments, assigning a letter to each) and identify the
long edges, respecting orientation. This results in a surface $S$ of genus
$k\leq n$ with $Q$ holes. The boundary of the disk becomes a graph on
this surface, which we denote by $\Gamma_S$. Thus $\G_S$ 
consists of  short edges, all of which are written
around the boundary components of $S$,  and long edges, all of which are 
properly embedded on $S$. 

Let 
$C$ be the  set of cyclic words, over $\cA$,
labelling the $Q$ boundary components of $S$. Define $\sim$ to be
the equivalence relation on $C$ generated by the relation which
consists of all pairs $(U,V)$ of words in $C$ such that $\supp_{F(\cA)}(U)
\cap \supp_{F(\cA)}(V)\neq \nul$.  Let $B_1,\ldots, B_p$ 
be the equivalence classes of $\sim$. Then each 
 set $B_i$
consists of a $t_i$-tuple of cyclic words $(W_1^i,\ldots ,W_{t_i}^i)$, where
$W_j^i$ is a cycle forming a boundary component of $S$, and by definition
of $B_i$,   this $t_i$-tuple is quadratic.
\begin{definition}\label{def:sgenus}
 The
$S$-\emph{genus} of $B_i$ is  
$g_i=\genus_H(\theta(W_1^i),\ldots ,\theta(W_{t_i}^i))+t_i-1$, 
for $i=1,\ldots ,p$.
\end{definition} 
\begin{lemma}\label{gi} 
Let $g=\sum _{i=1}^{p} g_i$, where $g_i$ is the $S$-genus of $B_i$, and let $\genus(S)=k$.
For each equivalence class $B_i$ consisting of  a $t_i$-tuple $(W_1^i,\ldots
,W_{t_i}^i)$ we have
\begin{equation*}
\genus_H(\theta(W_1^i),\ldots ,\theta(W_{t_i}^i))=\genus_{F(\mathcal{A})}(W_1^i,\ldots
  ,W_{t_i}^i),\quad \textrm{for}\,\,i=1,\ldots ,p.
\end{equation*}
It follows that $g+k=n$.
\end{lemma}  
\begin{proof}
Let $h_i= \genus_{F(\cA)}(W_1^i,\ldots
,W_{t_i}^i)+t_i-1$. Then $h_i\ge g_i$, for all $i$.   
Assume that $h_j>g_j$, for some
$1\leq j\leq p$. Now, since $W$ has
genus  $n$ in
$F(\mathcal{A})$, if we identify all
the  short edges on the genus $k$ surface $S$, respecting orientation,
we obtain the closed compact surface $\S_W$ of genus $n$.  
On the other hand we may  construct a surface homeomorphic to $\S_W$ as 
follows. For each $B_i$ there exists a van Kampen diagram $K_i$, over 
$F(\cA)$, on a surface $\S_i$ of genus $h_i-t_i+1$, with $t_i$ boundary
components, $W_{1}^j,\ldots ,W_{t_i}^j$. Identifying boundary components
of $\S_i$ to the $t_i$ boundary components of $S$, with the same 
boundary cycles, for all $i$, we obtain a closed surface of genus
$k+\sum_{i=1}^p [(h_i-t_i+1)+t_i-1]=k+\sum_{i=1}^p h_i$. By construction
this surface is homeomorphic to $\S_W$. 

Therefore, 
\begin{eqnarray*}
n &=&k+\sum_{i=1}^{p}h_i\\
& > &k+\sum_{i=1}^{p}g_i.
\end{eqnarray*}
But this implies that $\genus_H(h)=\genus_H(\theta(W))\leq
k+\sum_{i=1}^{p}g_i<n$, a contradiction. Thus $g_i=h_i$, for all
$i=1,\ldots ,p$,  and
$n=k+\sum g_i =k+g$.  Hence the lemma holds.   
\end{proof}
Consider once more the surface $S$  with $Q$ holes with the embedded graph
$\Gamma_S$ consisting of  long and
short edges. If we paste a disk onto each of the $Q$
boundary components and contract the cyclic word of short edges to a
point, we obtain a graph, consisting  of long edges only, 
on a closed  compact surface
of genus $k=\genus(S)$. The quadratic  word associated with this graph is
the long factorisation $U$ of $W$. As usual we denote this graph $\Gamma_U$. 
 We shall construct an extension of $U$ over $H$. 

As above, let the long factorisation of $W$ be $U=U_1\cdots U_l$. As in Section \ref{sec:ext}, 
let $U^\prime$ the Hamiltonian
cycle associated to $U$: that is  the word obtained
from $U$ by replacing each  occurrence $U_i$ by $U_{i1}$ or 
$U_{i2}$, as described in Section \ref{sec:ext}. 
 Thus the word $U=AC^{-1}A^{-1}C^{-1}$ 
in Example \ref{ex:U} is replaced by $U^\prime = A_1C_1^{-1}A_2^{-1}C_2^{-1}$. 
The graph obtained from $\G_U$ by replacing the directed edge 
$U_i$ with edges $U_{i1}$ and $U_{i2}$ is called $\G^\prime_U$, as
before. Now define a labelling function $\psi$ on the set $\{A_1,A_2:
A\in \supp(U)\}$, and by extension on the free group these letters
generate,  by 
\[\psi(A_1)=\a_1(A)\textrm{ and } \psi(A_2)=\a_2(A).\]
By definition of $\a_1$ and $\a_2$ we have then
\begin{equation}\label{eq:Ffac}
F=\psi(U^\prime).
\end{equation}
For instance, in Example \ref{ex:U} we have 
$\psi(U^\prime)=\a_1\g_1^{-1}\a_2^{-1}
\g_2^{-1}=F.$ 

Let $v_{j}^i$ be the vertex of $\G_U$ obtained by collapsing the 
boundary component $W_j^i$ of $S$ to a point. 
Let 
$B_i'=\{v_1^i,\ldots ,v_{t_i}^j\}$. \label{def:Bi'}
We shall construct a genus $g_i$ extension of $\G_U$ on 
$B_i'$, for appropriate $g_i$, for all $i$ (see Definition \ref{defn:ext}). Each internal vertex  
of $\G_S$ becomes a vertex of $\G_U$ and we shall construct a genus
$0$ extension on each such vertex.

Let $v$ be a vertex of $\G_U$ of degree $d$  and assume that, 
after suitable relabelling of edges, the vertex sequence of $\lk(v)$ is
 $E^{\e_1}_1,\ldots
,E^{\e_d}_d$, and the incidence sequence of $v$ is $O_1,\ldots , O_d$. 
Then the 
cyclic word $U$ contains the subwords $(E_q^{\e_q}E_{q+1}^{\e_{q+1}})^{O_q}$,
for $q=1,\ldots, d$, as in Figure \ref{fig:extstep1a} (with ``$e$'' replaced
throughout by ``$E$''.)

Now extend each vertex $v$ of $\G_U$ by a cycle graph $C_v$, as 
in Step 2 of the 
definition of an extension in Section \ref{sec:ext}. 
 As in the definition, 
if  the cycle $C_v=c_1\cdots c_d$, 
where $d$ is the degree of $v$ and $\i(c_q)=v_q$, $\t(c_q)=v_{q+1}$, then 
using the notation above, we shall have
 $v_q=
\t((E^q_{\nu_q})^{\e_q})
=\t(
(E^{q-1}_{\mu_{q-1}})^{\e_{q-1}}
)$, 
for $q=1,\ldots d$ (superscripts modulo $d$), 
as in Figure \ref{fig:extstep2c}.  
  The resulting graph is $\G_U''$. 

To complete Step 3 of the construction of the extension we must
 extend the labelling function $\psi$, defined above on $A_1$ and $A_2$, 
for $A\in \supp(U)$, to the edges of the cycles added to form $\G''_U$. 
For the cycle $C_v$ above we define, for $q=1,\ldots, d$, 
\begin{equation}\label{eq:psidefo}
\psi(c_q)=z_{\mu_q}(E_q)^{O_q}, \textrm{ if } o(E_q)=1, 
\end{equation}
and 
\begin{equation}\label{eq:psidefn}
\psi(c_q)=z_{\mu(O_q)}(E_q)^{\g_q}, \textrm{ if } o(E_q)=-1, 
\end{equation}
where 
\[\g_q=O_q\nu^{-1}(l_q)=
\begin{cases}
-O_q, & \textrm{if } l_q=1\\
O_q, & \textrm{if } l_q=2.
\end{cases}
.\]
\begin{lemma}\label{lem:uext}
The pair $(\G''_U,\psi)$ is an extension of $U$ over $H$. 
\end{lemma}
\begin{proof}
We must verify that (a), (b) and (c), of Step 3 of the definition
of extension, hold. By definition of $\psi$, (a) holds. 
To see that (b) holds, consider a long edge $A$ of $\G_U$ and suppose that 
$A$ is replaced by the pair $(A_1,A_2)$, in forming  $\G''_U$. We have
defined $\psi(A_i)=\a_i(A)$, for $i=1,2$. Assume that $\i(A)=u$ and 
$\t(A)=v$. Without loss of generality we may assume that the 
links of $u$ and $v$ have vertex sequences $E_1^{\e_1}, E_2^{\e_2},
\ldots ,E_d^{\e_d}$ and $G_1^{\d_1},G_2^{\d_2},\ldots, G_f^{\d_f}$, 
with $E_1^{\e_1}=A^{-1}$ and $G_1^{\d_1}=A$. Write
$O_q=O_q(u)$ and $O'_q=O_q(v)$, and similarly write $r_q$, $l_q$, $\mu_q$
and $\nu_q$, for $r_q(u)$, $l_q(u)$, $\mu_q(u)$
and $\nu_q(u)$; and $r'_q$, $l'_q$, $\mu'_q$
and $\nu'_q$, for $r_q(v)$, $l_q(v)$, $\mu_q(v)$
and $\nu_q(v)$. 
We may assume (by reversing orientation of all vertices if necessary) that
$O_1=1$. 
 From Lemma \ref{lem:vind}, the cyclic word $U'$ contains
\begin{equation}\label{eq:ulk}
A_{r_1}^{-1}E_{2,l_2}^{-\e_2}\textrm{ and }
(E_{d,r_d}^{\e_d}A_{l_1})^{O_d}
 \end{equation}
and 
\begin{equation}\label{eq:vlk}
(A_{r'_1}G_{2,l'_2}^{-\d_2})^{O'_1}\textrm{ and }
(G_{f,r'_f}^{\d_f}A_{l'_1}^{-1})^{O'_f}.
 \end{equation}
Assume further that $u$ and $v$ are extended by cycles $C_u$ and $C_v$, 
where $C_u=c_1\cdots c_d$ and $C_v=c'_1\cdots c'_f$, with 
\[\i(c_q)=v_{q}, \t(c_q)=v_{q+1}, \i(c'_q)=v'_{q}\textrm{ and }
\t(c'_{q+1})=v'_{q+1}.\]
Then 
\[\i(A_{l_1})=v_1=\i(c_1), \i(A_{r_1})=v_2=\t(c_1), 
\t(A_{l'_1})=v'_1=\i(c'_1)\textrm{ and }\t(A_{r'_1})=v'_2=\t(c'_1).
\]
 Consider first the case $o(A)=1$. 
As $\e_1=-1$ and $\d_1=1$, we have $\mu_1=2$, $\nu_1=1$, 
$\mu'_1=1$ and $\nu'_1=2$.  
 Then 
\begin{equation}\label{eq:altphi1}
\psi(c_1)=z_{\mu_1}(A)=z_2(A)
\end{equation}
 and 
\[l_1=\nu_1=1, r_1=\mu_1=2, \i(A_1)=\i(c_1)
\textrm{ and }
\i(A_2)=\t(c_1).\]
 If $O'_f=1$ then  
\begin{equation}\label{eq:altphi2}
\psi(c'_1)=z_{\mu'_1}(A)=z_1(A)
\end{equation}
 and 
\[l'_1=\nu'_1=2, r'_1=\mu'_1=1, \t(A_2)=\i(c'_1)
\textrm{ and }
\t(A_1)=\t(c'_1).\]
We therefore have a closed path $c_1A_2c'_1A_1^{-1}$ and, 
from Lemma 
\ref{lem:extedge} \ref{it:extedge1}, 
\[\psi(c_1)\psi(A_2)\psi(c'_1)\psi(A_1^{-1})=
z_2(A)\a_2(A)z_1(A)\a_1(A)^{-1}=_H 1,\] 
so 
 (b) holds in this case.  
If  $O'_f=-1$ then 
\begin{equation}\label{eq:altphi3}
\psi(c'_1)=z_{\mu'_1}(A)^{-1}=z_1(A)^{-1}
\end{equation}
and a similar argument shows that we have a closed path 
$c_1A_2c_1^{\prime\, -1}A_1^{-1}$; so (b) also holds in this case.

Now suppose that $o(A)=-1$. Then, from \eqref{eq:ulk}, $U'$ contains
$A^{-1}$, so $\s(A)=(-1,-1)$. From \eqref{eq:vlk}, then $O'_1=-1$. Hence
\begin{equation}\label{eq:nonaltphi}
\psi(c_1)=z_1(A)^{\nu^{-1}(l_1)}\textrm{ and }
\psi(c'_1)=z_2(A)^{-\nu^{-1}(l'_1)}.
\end{equation}
If $r_1=l'_1$ and $l_1=r'_1$ then 
\[\psi(c'_1)=z_2(A)^{-\nu^{-1}(r_1)}=z_2(A)^{\nu^{-1}(l_1)},\]
and 
\[\t(A_{r_1})=\t(A_{l'_1})=\i(c'_1) \textrm{ and } \t(A_{l_1})=
\t(A_{r'_1})=\t(c'_1),\]
so we have a closed path $c_1A_{r_1}c'_1A_{l_1}^{-1}$ and 
\[\psi(c_1)\psi(A_{r_1})\psi(c'_1)\psi(A_{l_1}^{-1})
=z_1(A)^{\nu^{-1}(l_1)}\a_{r_1}(A)z_2(A)^{\nu^{-1}(l_1)}\a_{l_1}(A)^{-1}.\]
Thus, as $\e=-1$,  if $l_1=1$ we have 
\[\psi(c_1)\psi(A_{r_1})\psi(c'_1)\psi(A_{l_1}^{-1})
=z_1(A)^{-1}\a_{2}(A)^{-\e}z_2(A)^{-1}\a_1(A)^{\e}\]
and if $l_1=2$ then 
\[\psi(c_1)\psi(A_{r_1})\psi(c'_1)\psi(A_{l_1}^{-1})
=z_1(A)^{-1}\a_{1}(A)^{-\e}z_2(A)\a_1(A)^{\e}\]
and in both cases (b) follows from Lemma 
\ref{lem:extedge} \ref{it:extedge2}.
The case $r_1=r'_1$ and $l_1=l'_1$ follows in the same way, 
so (b) holds in all cases. 

As $\psi(U')=F$, (c) holds. Therefore
we have an extension of $U$ over $H$, as required.
\end{proof}

\subsection{Length and genus of $(\G''_U,\psi)$}
The next two 
Lemmas prepare the ground for  the proof of Lemma \ref{lem:genext}.  
\begin{lemma}\label{lem:bdryword}
Let $v_{j}^i$ be the vertex of $\G_U$ obtained by collapsing the 
boundary component $W_j^i$ of $S$ to a point. Let $E_1^{\e_1}, \ldots, E_d^{\e_d}$  
be
the vertex sequence of $\lk(v_{j}^i)$, and $O_1,\ldots, O_d$  the 
incidence sequence of $v_{j}^i$. Then $\theta(W_j^i)$ is conjugate
to $u_1^{O_1} \cdots u_d^{O_d}$, where 
\[u_q=
\begin{cases}
s_{l_{q+1}}(E_{q+1}),&  \textrm{if }O_q=1\\
s_{r_q}(E_q),& \textrm{if } O_q=-1.
\end{cases}
\]
\end{lemma}
\begin{proof}
By definition of $O_q$, the cyclic word $U$ contains 
$(E_{q}^{\e_q} E_{q+1}^{-\e_{q+1}})^{O_{q}}$, and hence 
$W$ contains  $(E_{q}^{\e_q}S_q^{\prime O_q}E_{q+1}^{-\e_{q+1}})^{O_{q}}$,
where 
\[S'_q=\begin{cases}
S(E_{q+1}^{-\e_{q+1}}),& \textrm{if } O_q=1,\\
S(E_{q}^{-\e_{q}}),& \textrm{if } O_q=-1,
\end{cases}
\]
 for $q=1,\ldots, d$. Moreover, from Lemma \ref{lem:vind}, 
the cyclic word $U'$ contains $(E_{q,r_q}^{\e_q} E_{q+1,l_{q+1}}^{-\e_{q+1}})^{O_{q}}$,
 in place of $(E_{q}^{\e_q} E_{q+1}^{-\e_{q+1}})^{O_{q}}$. Hence, by
 considering all the possible values when $O_q=\pm 1$, $o(E_{q+1})=\pm 1$ and 
$\e_{q+1}=\pm 1$, we see that 
if $O_q=1$ then $S'_q=S_{l_{q+1}}( E_{q+1})$, while if $O_q=-1$ then 
$S'_q=S_{r_{q}}( E_{q})$. 
\end{proof}
In the following lemma we abbreviate 
the notation of Definition \ref{def:Ffact},
 writing $a_{l_q}$, $b_{l_q}$, $s_{l_q}$, $a_{r_q}$,   $b_{r_q}$ and  $s_{r_q}$ for
$a_{l_q}(E_q)$, $b_{l_q}(E_q)$, $s_{l_q}(E_q)$, $a_{r_q}(E_q)$,   $b_{r_q}(E_q)$ and  $s_{r_q}(E_q)$, respectively. 
\begin{lemma}\label{lem:cyccanc}
Let $v$ be a vertex of $\G_U$ with incidence sequence $O_1,\ldots, O_d$ and 
link  $E_1^{\e_1}, \ldots, E_d^{\e_d}$. Let $C_v=c_1\cdots c_d$ be 
the cycle extending $v$, 
let $V_q=(o(E_q),O_q;o(E_{q+1}),O_{q+1})$, and $t_q=\psi(c_q)\psi(c_{q+1})$, for $q=1,\ldots, d$. 
Then, for $q=1,\ldots, d$,
\[\a(c_q)=
\begin{cases}
b_{l_q}^{-1}s_{l_q}a_{r_q},&\textrm{if } o(E_q)=1\textrm{ and } O_q=1\\
a_{l_q}^{-1}s_{r_q}^{-1}b_{r_q},&\textrm{if } o(E_q)=1\textrm{ and } O_q=-1\\
a_{l_q}^{-1}a_{r_q},&\textrm{if } o(E_q)=-1\textrm{ and } O_q=1\\
b_{l_q}^{-1}s_{l_q}s_{r_q}^{-1}b_{r_q},&\textrm{if } o(E_q)=-1\textrm{ and } O_q=-1\\
\end{cases}
,
\]
\[V_q=(\e,\d:\g,\d\g),\textrm{ for some } \e,\d, \g=\pm 1,\]
and
\be[(i)]
\item\label{it:cyccanc1} $t_q=b^{-1}_{l_q}s_{l_q}s_{l_{q+1}}a_{r_{q+1}}$, if $V_q=(1,1;1,1)$, 
\item $t_q=a^{-1}_{l_q}s_{r_q}^{-1}s_{r_{q+1}}^{-1}b_{r_{q+1}}$, if $V_q=(1,-1;1,-1)$,
\item $t_q=b^{-1}_{l_q}s_{l_q}s_{l_{q+1}}s_{r_{q+1}}^{-1}b_{r_{q+1}}$, if 
$V_q=(1,1;-1,-1)$,
\item $t_q=a^{-1}_{l_q}s_{r_q}^{-1}a_{r_{q+1}}$, if $V_q=(1,-1;-1,1)$,
\item $t_q=a^{-1}_{l_q}s_{l_{q+1}}a_{r_{q+1}}$, if $V_q=(-1,1;1,1)$,
\item $t_q=b^{-1}_{l_q}s_{l_q}s_{r_q}^{-1}s_{r_{q+1}}^{-1}b_{r_{q+1}}$, if 
$V_q=(-1,-1;1,-1)$,
\item $t_q=a^{-1}_{l_q}s_{l_{q+1}}s_{r_{q+1}}^{-1}b_{r_{q+1}}$, if $V_q=(-1,1;-1,-1)$ and 
\item $t_q=b^{-1}_{l_q}s_{l_q}s_{r_q}^{-1}a_{r_{q+1}}$, if 
$V_q=(-1,-1;-1,-1)$.
\ee
\end{lemma}
\begin{proof}
If $o(E_q)=1$ then 
\[\psi(c_q)=z_{\mu_q}(E_q)^{O_q}=\left[b_{\nu_q}^{-1}s_{\nu_q}a_{\nu_q} 
\right]^{O_q},\]
and the expressions of $\a(c_q)$ when $o(E_q)=1$, follow. When $o(E_q)=-1$ 
then 
\[\psi(c_q)=z_{\mu(O_q)}(E_q)^{O_q\nu^{-1}(l_q)}
=\begin{cases}
\left[a_2(E_q)^{-1}a_1(E_q) \right]^{\nu^{-1}(l_q)}, &\textrm{if } O_q=1,\\
\left[b_1(E_q)^{-1}s_1(E_q) s_2(E_q)^{-1}b_1(E_q)\right]^{-\nu^{-1}(l_q)}, &\textrm{if } O_q=-1,\\
\end{cases}
\]
and the given expressions for $\psi(c_q)$ when $o(E_q)=-1$ follow, on
considering the two possibilities, $l_q=1$ and $l_q=2$.
That $V_q$ is determined by its first three entries, as claimed, follows
from the definition of $O_q$. 

To compute $t_q$, observe that, from Lemma \ref{lem:vind}, 
if $O_q=1$ then $U'$ contains $E_{q,l_q}^{\e_q}E_{q+1,l_{q+1}}^{-\e_{q+1}}$, 
so that $a_{r_q}=b_{l_{q+1}}$. Similarly, if $O_1=-1$ then 
$U'$ contains $E_{q+1,l_{q+1}}^{\e_{q+1}}E_{q,l_q}^{-\e_q}$, so 
$a_{l_{q+1}}=b_{r_q}$. 
To see that \ref{it:cyccanc1} holds note that when $o(E_q)=o(E_{q+1})=O_q=
O_{q+1}=1$, we have $l_q=\nu_q$, $r_q=\mu_q$, $l_{q+1}=\nu_{q+1}$ and 
$r_{q+1}=\mu_{q+1}$. As $O_q=1$ we have $a_{r_q}=b_{l_{q+1}}$, 
so \[
t_q=[b_{\nu_q}s_{\nu_q}a_{\mu_q}][b_{\nu_{q+1}}s_{\nu_{q+1}}a_{\mu_{q+1}}]
=[b_{l_q}^{-1}s_{l_q}a_{r_q}][b_{l_{q+1}}^{-1}s_{l_{q+1}}a_{r_{q+1}}]
=b_{l_q}^{-1}s_{l_q}s_{l_{q+1}}a_{r_{q+1}}.\]
The other seven cases follow similarly. 
\end{proof}
\begin{lemma}\label{lem:genext}
Let $v$ be an internal vertex of $\G_S$ and (as on page \pageref{def:Bi'})  
let $B'_i=\{v_1^i,\ldots, v_{t}^i\}$ be the set of vertices of $\G_S$ corresponding to the 
 equivalence
class $B_i$, and let $g_i$ be 
the $S$-genus of $B_i$.  Let $C_v$ and $C_j^i$ be the cycles  extending  
$v$ and $v_j^i$, respectively, in $\G''_U$.  Then  the pair $(\G''_U,\psi)$ is 
\be
\item\label{it:genext1} a genus $0$ joint extension on $v$ by the cyclic word $\psi(C_v)$ and 
\item a genus $g_i$ joint extension on $B'_i$ by the words 
$\psi(C_1^i), \ldots , \psi(C_{t_i}^i)$. 
\ee
\end{lemma}
\begin{proof} 
First let $u$ be any vertex of $\G_U$, 
with incidence sequence $O_1,\ldots, O_d$
and let  $E_1^{\e_1}, \ldots, E_d^{\e_d}$ 
be the vertex sequence of $\lk(u)$. Let $C_u=c_1\cdots c_d$ be the 
cycle extending $u$. 
 It follows from Lemma \ref{lem:cyccanc} that 
all the words $a_{l_q}$, $a_{r_q}$, $b_{l_q}$ and $b_{r_q}$ cancel
in the product $\psi(C_u)=\psi(c_1)\cdots \psi(c_d)$. Thus
\[\psi(C_u)=\hat\psi(c_1)\cdots \hat\psi(c_d),\]
where 
\[
\hat\psi(c_q)=
\begin{cases}
s_{l_q},&\textrm{if } o(E_q)=1\textrm{ and } O_q=1\\
s_{r_q}^{-1},&\textrm{if } o(E_q)=1\textrm{ and } O_q=-1\\
1,&\textrm{if } o(E_q)=-1\textrm{ and } O_q=1\\
s_{l_q}s_{r_q}^{-1},&\textrm{if } o(E_q)=-1\textrm{ and } O_q=-1\\
\end{cases}
.
\]
If $u$  an internal vertex of $\G_S$ then, by definition, $s_{l_q}=s_{r_q}=1$, 
for all $q$. 
 Therefore, for the internal vertex $v$ of $\G_S$ we have
 $\psi(C_v)=1$. Moreover, if $v$ is an internal vertex of $\G_S$ then
it is also a vertex of $\G_W$ and hence has degree at least $3$, unless
$U=A^{\pm 2}$, for some $A\in\cA$.  
Thus \ref{it:ext1} and \ref{it:ext2} of Definition \ref{defn:ext} hold,  and \ref{it:genext1} follows. 

Assume now that $u=v_j^i\in B'_i$. 
If $o(E_q)=\e$, for fixed $\e=\pm 1$, for all $q$, then 
it follows from Lemma \ref{lem:bdryword} that $\psi(C_u)$ is conjugate in 
$H$ 
to $\theta(W_j^i)$. Hence we may assume that $o(E_q)=1$ for at least one
index $q$ and that $o(E_q)=-1$, for at least two indices $q$. In this
case we may, by renumbering the $E_q$ if necessary, choose $E_1$ such that 
$o(E_1)=-1$ and $O_1=1$. Let $1=i_1<i_2<\cdots <i_{2m}\le d$ be integers
such that $o(E_{q})=-1$ if and only if  $q\in \{i_1,\ldots, i_{2m}\}$. 
Then $O_q=-1$ if $i_{2j}\le q< i_{2j+1}$, and $O_q=1$ if 
$i_{2j+1}\le q< i_{2(j+1)}$, for $j=1,\ldots, m$ (subscripts modulo $2m$). 
Therefore
\begin{gather*}
\hat\psi(c_{i_1})=\hat\psi(c_{i_3})=\cdots =\hat\psi(c_{i_{2m-1}})=1,\\
\hat\psi(c_{i_{2j}})=s_{l_{i_{2j}}}s_{r_{i_{2j}}}^{-1},\textrm{ for } 
j=1,\ldots, m,\\
\hat\psi(c_{q})=s_{l_q},\textrm{ for } i_{2j+1}< q< i_{2(j+1)},\textrm{ and }\\
\hat\psi(c_{q})=s_{r_q}^{-1},\textrm{ for } i_{2j}< q< i_{2j+1}.
\end{gather*}
Hence, for $j=1,\ldots ,m$, 
\[\hat\psi(c_{i_{2j-1}})\cdots \hat\psi(c_{i_{2j+1}-1})=s_{l_{i_{2j-1}+1}}
\cdots s_{l_{i_{2j}}}
s_{r_{i_{2j}}}^{-1}\cdots s_{r_{i_{2j+1}-1}}^{-1},\] 
where the subscripts of $i$ are integers modulo $2m$. 
From Lemma \ref{lem:bdryword} again, it follows that $\psi(C_{v_j^i})$ is 
conjugate in $H$ to $\theta(W_j^i)$. 

 This is true for all vertices of $B'_i$, so 
if $B_i'=\{v_1^i,\ldots ,v_{t_i}^i\}$ then 
\[\genus_H(\psi(C_1^i),\ldots ,\psi(C_{t_i}^i))=
\genus_H(\theta(W^i_1),\ldots ,\theta(W^i_{t_i}))=g_i-t_i+1,\]
where $g_i$ is the $S$-genus of $B_i$. Thus \ref{it:ext1}  of Definition
\ref{defn:ext} holds for the 
extension of $B_i'$ by $\psi(C_1^i),\ldots ,\psi(C_{t_i}^i)$. 

To verify
\ref{it:ext2}, suppose that $g_i=0$, $t_i=1$ and that the degree of $v_1^i$ is $d$. 
Then $\genus_H(\psi(C_1^i))=\genus_H(\theta(W_1^i))=0$ and Lemma \ref{gi} implies
that $\genus_{F(\mathcal{A})}(W_1^i)=0$. 
If $d=1$ then 
$W_1^i=1$. But $W_1^i=S(E_1^{\e_1})$ is a cyclic subword of the Wicks
form $W$ (or it's inverse) and by definition $W$ is cyclically reduced. Therefore this
can't occur. 

Now suppose that $d=2$. Then $v_j^i$ has link $E_1^{\e_1},E_2^{\e_2}$, 
and since there are an even number of non-alternating edges incident
to every vertex of $\G_W$, $o(E_1)=o(E_2)=\pm 1$. Moreover
$W$ contains $(E_1^{\e_1}S_1^{\prime O_1}E_2^{-\e_2})^{O_1}$ and 
$(E_2^{\e_2}S_2^{\prime O_2}E_1^{-\e_1})^{O_2}$, where $S_q^\prime$ is 
defined in the proof of Lemma \ref{lem:bdryword}. If $o(E_1)=o(E_2)=1$, 
then $O_1=O_2$ and $W^{\pm 1}$ contains 
$E_1^{\e_1}S_1^{\prime}E_2^{-\e_2}$ and 
$E_2^{\e_2}S_2^{\prime}E_1^{-\e_1}$. In this case the label of the
cyclic word $W^i$ is $(S_1^\prime S_2^\prime)^{\pm 1}$, so $S'_1=
S_2^{\prime\, -1}$. Hence $W$ is redundant, a contradiction. 
In the case $o(E_1)=o(E_2)=-1$, we may assume that $O_1=-O_2=1$ and
so $W^{\pm 1}$ contains $E_1^{\e_1}S_1^{\prime}E_2^{-\e_2}$ and 
$E_1^{\e_1}S_2^{\prime}E_2^{-\e_2}$. In this case the boundary label 
of $W^i$ is $(S_1^\prime S_2^{\prime\, -1})^{\pm 1}$, so 
$W$ contains two occurrences of $E_1^{\e_1}S_1^{\prime}E_2^{-\e_2}$, which
again implies redundancy.
 
  Therefore part \ref{it:ext2} of Definition \ref{defn:ext} holds. Thus,  
we have the required 
genus $g_i$ extension of $B_i'$. 
\end{proof}

We are now in a position to complete the main part of the proof of Theorems \ref{Thexp} and \ref{Thexp-}, 
by showing that $(\G''_U,\psi)$ is the required extension of $U$ over $H$. Once this has been done
the remaining step is to bound the length of $R$, and this is addressed in Section \ref{sec:Rbound}.

\begin{proposition}\label{prop:ext}
Let $\genus(S)=k$ and let $g=n-k$. Then 
the pair $(\G''_U,\psi)$ is a genus $g$ extension of $U$ over $H$, of length
at most $2K(n)(12l+M+4)$.
\end{proposition}
\begin{proof}
Let $C_1,\ldots, C_Q$ be the cycles extending the vertices of $\G_U$ corresponding
 to the $Q$ boundary components of $S$; and let $C_{Q+1},\ldots, C_{Q+t}$ be the 
cycles extending the internal vertices of $S$. Then the length of 
 $(\G''_U,\psi)$ is 
\[\sum_{i=1}^{Q+t}|\psi(C_i)|.\]
For $1\le j\le Q+t$, we have $\psi(C_j)$ equal to  
a product of terms of the form  $a_j(A)$, $b_j(A)$ and $s_j(A)$, where 
$A$ is a long edge of $W$, and $j=1$ or $2$.  
From the proof of Lemma \ref{lem:uext}, in particular equations
\eqref{eq:altphi1}, \eqref{eq:altphi2}, \eqref{eq:altphi3} and 
\eqref{eq:nonaltphi}, for each long edge $A$ of $W$ and for $j=1,2$, 
each of $a_j(A)$, $b_j(A)$ and $s_j(A)$ occurs in exactly one 
$\psi(c)$, for some letter $c$ of one of the cycles $C_v$.   
 
Every $s_j(A)$ is the product of the labels of the short edges
between two consecutive long edges of $W$;  
 and short edges  have
labels in $F(X)$ of 
length  at most  $12l+M+4$.  
Each
$a_{j}(A)$ and $b_{j}(A)$ is a word of length at most $5l+M+3$ arising from  a
long edge of $W$ (Lemma \ref{lem:chouv}).  

By Lemma \ref{cull},  the maximum
number of letters in $W$ is $K=\max\{2,12n-6\}$. 
Thus the sum of number of short
letters and long letters is at most $K$. Therefore, if
$\mathcal{S}$ and $\mathcal{L}$ are the numbers of short edges and long
edges of $W$, respectively, the length of the extension 
is bounded  by
\begin{eqnarray*}
(12l+M+4)\mathcal{S}+2(5l+M+3)\mathcal{L} &\leq & 2K(12l+M+4).
\end{eqnarray*}   

From Lemma \ref{lem:genext} $(\G''_U,\psi)$ is an extension 
of genus $\sum_{i=1}^p g_i$, and the proposition now follows from Lemma \ref{gi}.
\end{proof}
Combining equation \eqref{eq:Ffac} and Proposition \ref{prop:ext}, 
we have the required extension 
 and it remains only to verify the 
bound on the length of $R$. 

\subsection{Bounding the length of $R$}\label{sec:Rbound}
Now assume $h=RFR^{-1}$, where $h$ satisfies the hypotheses and $F$ the conclusions of Theorem \ref{Thexp} or Theorem \ref{Thexp-}. 
We shall consider the geodesic quadrilateral with sides labelled $F, R^{-1}, h^{-1}, R$, based at $1$ in $\G_X(H)$, 
and refer to these geodesics by their labels. We apply Lemma \ref{lem:la} to this quadrilateral, with $\g_0=F$. 
This gives a partition  $F_\i,F_\g,F_\t$ of $F$ such that $F_\t^{-1}$ $\d$-fellow travels with an initial segment of $R^{-1}$, 
$F_\i^{-1}$ $\d$-fellow travels with a terminal segment of $R$ and $F_\g^{-1}$, which may be empty, $2\d$-fellow travels with a 
segment of $h$, as in Figure \ref{fig:bdr1}. 
\begin{figure}[htp]
\begin{center}
\begin{subfigure}[b]{.32\columnwidth}
\begin{center}
\psfrag{h}{{\scriptsize $h$}}
\psfrag{Fi}{{\scriptsize $F_\i$}}
\psfrag{Fg}{{\scriptsize $F_\g$}}
\psfrag{Ft}{{\scriptsize $F_\t$}}
\psfrag{R}{{\scriptsize $R$}}
\psfrag{<d}{{\scriptsize $\le \d$}}
\psfrag{<2d}{{\scriptsize $\le 2\d$}}
\includegraphics[scale=0.35]{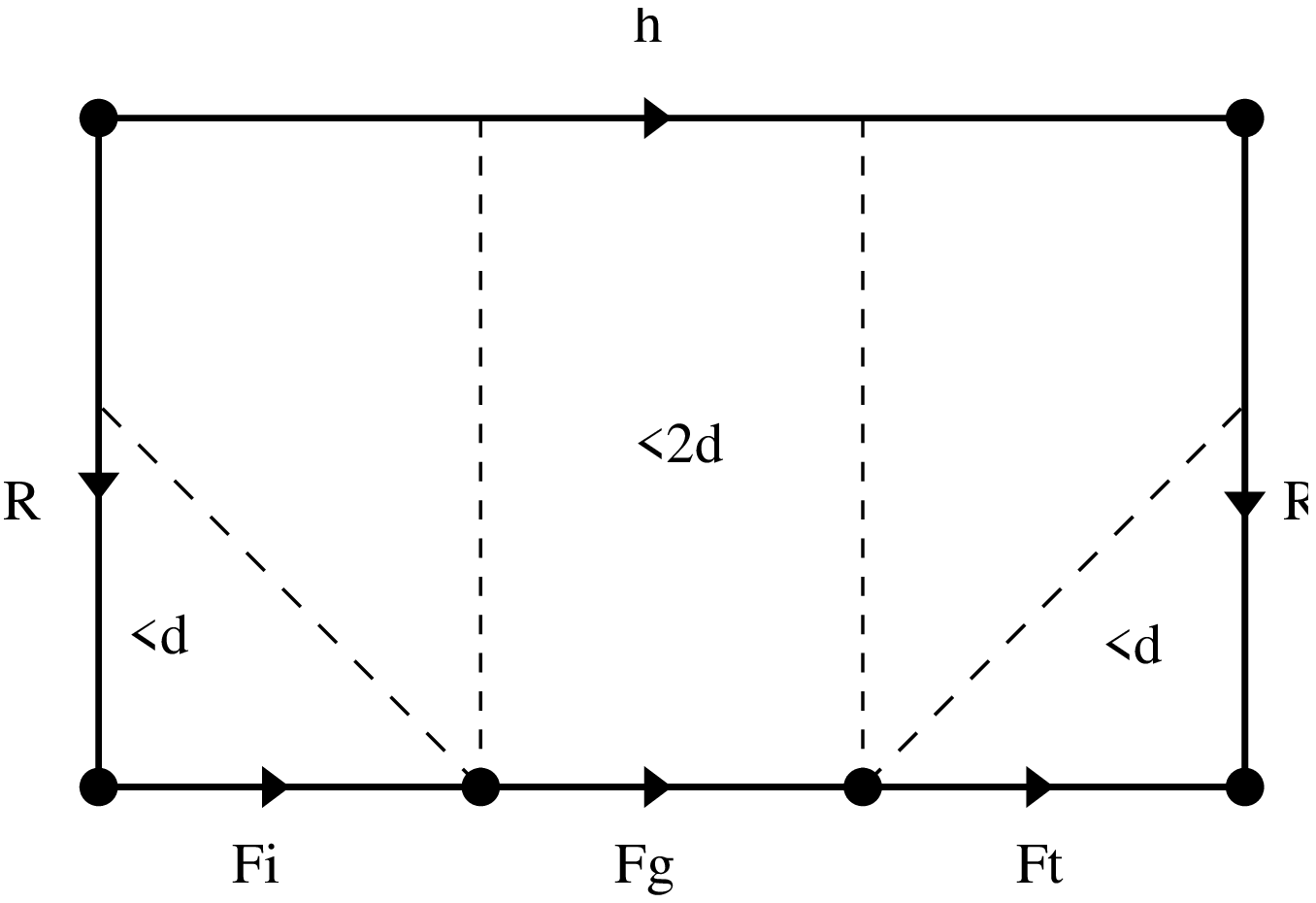}
\caption{A partition $F_\i,F_\g,F_\t$}\label{fig:bdr1}
\end{center}
\end{subfigure}
\begin{subfigure}[b]{.32\columnwidth}
\begin{center}
\psfrag{h}{{\scriptsize $h$}}
\psfrag{F0}{{\scriptsize $F_0$}}
\psfrag{F1}{{\scriptsize $F_1$}}
\psfrag{R0}{{\scriptsize $R_0$}}
\psfrag{R1}{{\scriptsize $R_1$}}
\psfrag{R}{{\scriptsize $R$}}
\psfrag{d}{{\scriptsize $d$}}
\psfrag{p}{{\scriptsize $p$}}
\psfrag{q}{{\scriptsize $q$}}
\psfrag{a}{{\scriptsize $a$}}
\psfrag{TE}{{\scriptsize $\theta(E)$}}
\includegraphics[scale=0.35]{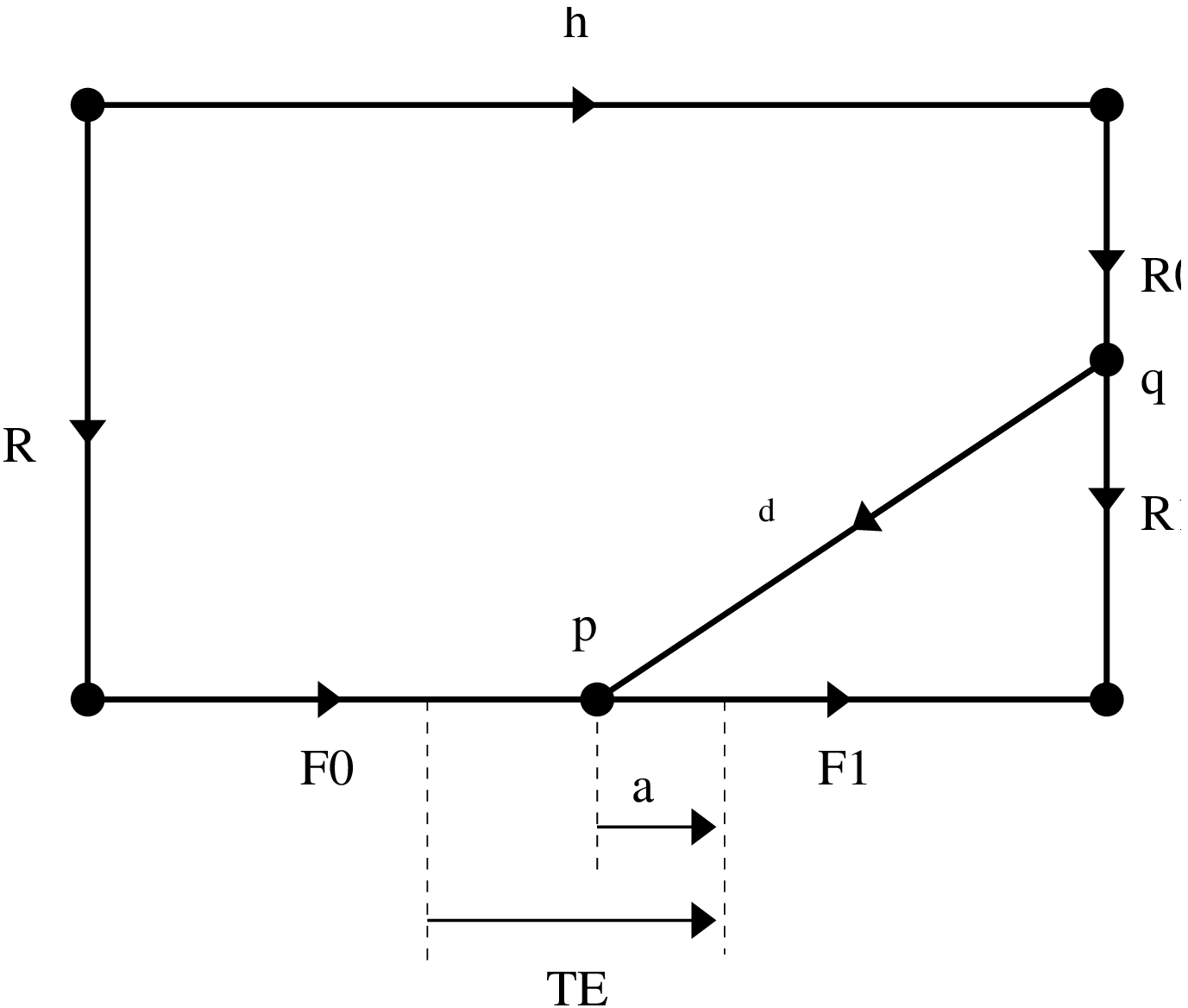}
\caption{$p$ lies on $F_\t$ and $q$ on $R^{-1}$}\label{fig:bdr2}
\end{center}
\end{subfigure}
\begin{subfigure}[b]{.32\columnwidth}
\begin{center}
\psfrag{h}{{\scriptsize $h$}}
\psfrag{Fi}{{\scriptsize $F_\i$}}
\psfrag{Fg}{{\scriptsize $F_\g$}}
\psfrag{Ft}{{\scriptsize $F_\t$}}
\psfrag{R}{{\scriptsize $R$}}
\psfrag{F0}{{\scriptsize $F_0$}}
\psfrag{F1}{{\scriptsize $F_1$}}
\psfrag{R0}{{\scriptsize $R_0$}}
\psfrag{R1}{{\scriptsize $R_1$}}
\psfrag{h0}{{\scriptsize $h_0$}}
\psfrag{h1}{{\scriptsize $h_1$}}
\psfrag{<d}{{\scriptsize $\le \d$}}
\psfrag{<2d}{{\scriptsize $\le 2\d$}}
\psfrag{d}{{\scriptsize $d$}}
\psfrag{p}{{\scriptsize $p$}}
\psfrag{q}{{\scriptsize $q$}}
\psfrag{a}{{\scriptsize $a$}}
\psfrag{TE}{{\scriptsize $\theta(E)$}}
\includegraphics[scale=0.35]{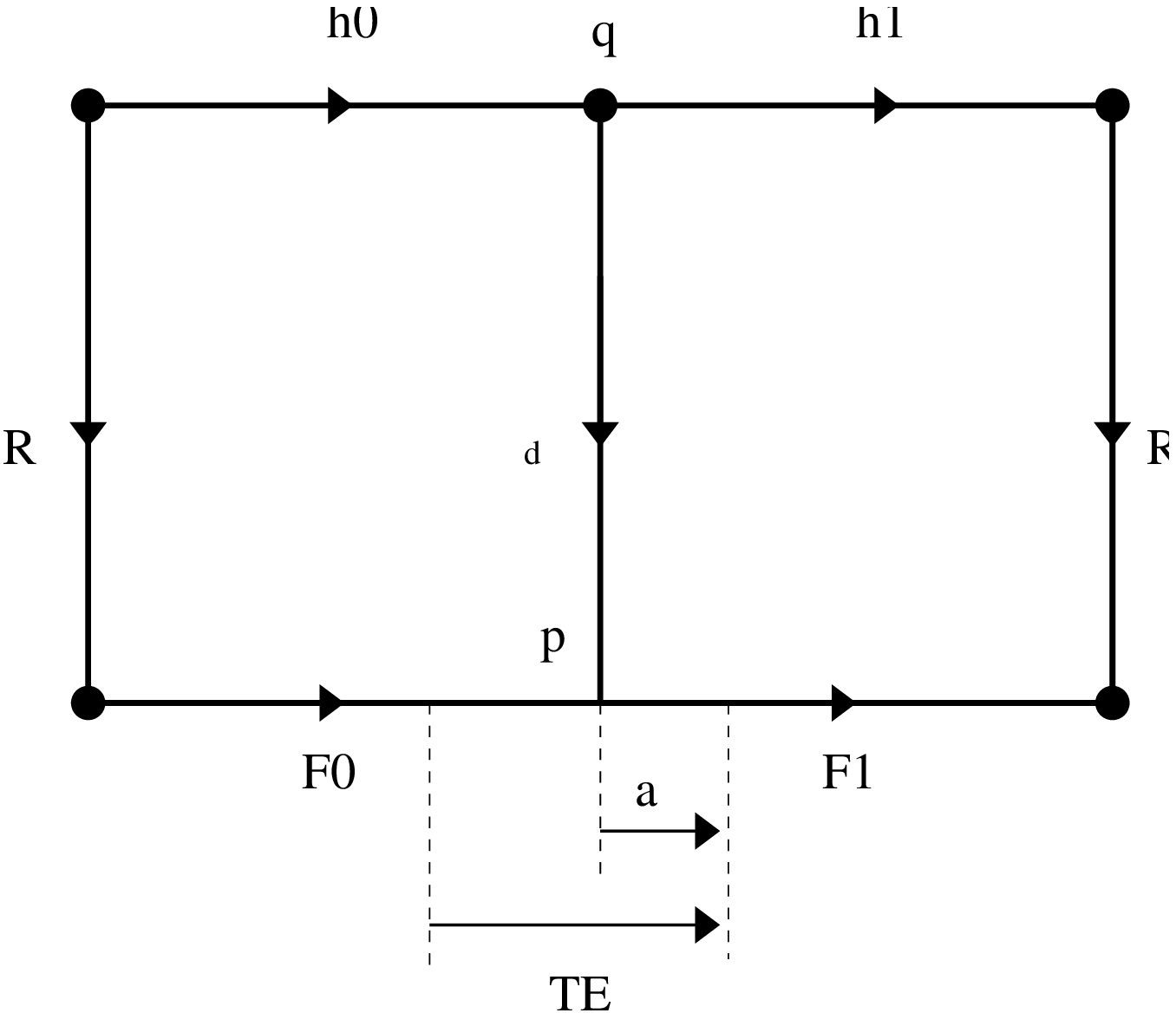}
\caption{$p$ lies on $F_\g$ and $q$ on $h$}\label{fig:bdr3}
\end{center}
\end{subfigure}
\end{center}
\caption{}\label{fig:bdra}
\end{figure}  
Assume first that $F$ satisfies \ref{it:exp1} of Theorem \ref{Thexp} or Theorem \ref{Thexp-}. 
If $|F|\le 12l+M+4\d+5$ then, from  Lemma \ref{lb}, there is $R$ such that $h=_HRFR^{-1}$ and $|R|\le |h|/2+6l+3M/2+2\d+7/2$, as claimed.
Therefore we assume that $|F|>12l+M+4\d+5$ and that $F=\theta(W)$, where $W$ is a Wicks form of the appropriate type, with
$|\theta(E)|\le 12l+M+4$, for all $E \in \supp(W)$. 
Moreover, we assume that $R$ is minimal, in the sense that, 
for all $(W',\theta')\in \cF$, with no long edges, if $h=_H R'\theta'(W'){R'}^{-1}$ then $|R|\le |R'|$. 

Let $p$ be a vertex on $F$ such that $|d(p,\i(F)-d(p,\t(F)|\le 1$. First suppose that  $p$ lies on $F_\i$ or $F_\t$, so there
is a 
vertex $q$ on $R$ or $R^{-1}$ such that $d(p,q)\le \d$. Without loss of generality,  we may take $q$ to lie on $R^{-1}$ 
and we take $d$ to be the label of a geodesic from $q$ to $p$. 
Then $R=R_0R_1$, where $\t(R_0)=\i(R_1)=q$ and $F=F_0F_1$, where $\t(F_0)=\i(F_1)=p$, as in Figure \ref{fig:bdr2}. This means  
$|R_1|\ge |F_1|-\d$ and  as $|F_1|\ge (|F|-1)/2$, we have $|R_1|> 6l+M/2+\d+2$.  
However $|\theta(E)|\le 12l+M+4$, for all $E\in \supp(W)$, so 
$p$ is distance at most $6l+M/2+2$ from the initial vertex of $\theta(E^{\pm 1})$, for some  $E\in \supp(W)$. Thus $a^{-1}F_1F_0a=\theta(W')$, for
some cyclic permutation $W'$ of $W$ and some $a\in F(X)$ with $|a|\le 6l+M/2+2$ (see  Figure \ref{fig:bdr2}). Let $F'=\theta(W')$. 
As $R_1=dF_1$, we have $h=_H R_0d F_1F_0d^{-1}R_0^{-1}= R_0da F'a^{-1}d^{-1}R_0^{-1}$ and $|R_0da|\le |R_0|+|a|+|d|\le |R_0|+ 6l+M/2+\d+2<|R_0|+|R_1|=|R|$. 
As $W'$ is a Wicks form we may replace $(W,\theta)$ and $F$, in the conclusion of  Theorem \ref{Thexp} or Theorem \ref{Thexp-}, 
by  $(W',\theta)$ and $F'$ and then replace $R$ by $R_0da$. However, this contradicts the minimality of $|R|$. Therefore we may
assume that $p$ lies on $F_\g$. 

As $p\in F_\g$ there is a vertex $q$ on $h$ such that  $d(p,q)\le 2\d$, and we now take 
 $d$ to be the label of a geodesic from such a vertex $q$ on $h$ to the vertex $p$ on $F$, as in Figure \ref{fig:bdr3}. 
Let $h=h_0h_1$, where $\t(h_0)=q=\i(h_1)$; so $h=h_2h_1h_0h_2^{-1}$, for some $h_2$ such that $|h_2|\le |h|/2$ (i.e. $h_2=h_0$ or $h_1^{-1}$). 
Then $h_1h_0=dF_1F_0d^{-1}$ so 
$h=h_2dF_1F_0d^{-1}h_2^{-1}=h_2da F' a^{-1} d^{-1}h_2^{-1}$, with $a$ and $F'$ as before. Now $|h_2da|\le |h|/2 + 2\d+ 6l+M/2+2$ and by hypothesis 
$|R|\le |h_2da|$, so in all cases we have $|R| \le |h|/2+6l+3M/2+7/2+2\d$.

Now assume that $F$ satisfies \ref{it:exp2} of Theorem \ref{Thexp} or Theorem \ref{Thexp-}. 
If  $R$ is an $H$-minimal element of $F(X)$ such that 
$h=_H R F R^{-1}$, where $F=\theta(W)$,  then we say that $R$ is a \emph{conjugator} for 
$(W,\theta)$. Since we assumed that $W$ ends in
a long edge we call $R$ a \emph{minimal conjugator} if $R$ is a conjugator for $(W,\theta)$ and, 
for every cyclic permutation $W'$ of 
$W$ which ends in a long edge, and every conjugator $R'$ of $(W',\theta)$,
we have $|R|\le |R'|$. Assume then that $R$ is a minimal conjugator. Then 
 $F$ has a partition $\a_1,\ldots,\a_f$, such that \ref{it:chouv1} and \ref{it:chouv2} of Lemma \ref{lem:chouv}  
 hold (with $\a_i$ in place of $F_i$). 
Let $P=\{\t(\a_i)\,:\, 1\le i\le f-1\}$. 
Then, from Lemma \ref{lem:chouv} again, $P\neq \emptyset$. 
Suppose that $p$ lies on $F_\i$ or $F_\t$, for all $p\in P$.  
Then, without loss of generality we may assume that $p$ lies on $F_\t$, for some $p\in P$; so
 there exists a point $q$ on $R^{-1}$ such that $d(p,q)\le \d$. 
Factorise $R=R_0R_1$ and $F=F_0F_1$,  as in the previous case. This time 
 $F_1F_0$ is the label of a Hamiltonian cycle on a graph corresponding to the extension given by the theorem, 
and the point $p$ is, by definition $v(U_i)$ for some long edge $U_i$. Thus $F_1F_0$ is the  Hamiltonian cycle
 of an extension constructed from a permutation $W'$ of $W$ ending with a long edge. Hence 
 we may replace $F$ by $F_1F_0$ in the conclusion of the theorem. 
We have $|F_1|\ge 6l+2$, since $\a_f$ is a subinterval of $F_1$, so $|R_1|\ge  6l+2-\d$,  
 and  $h=_H R_0dF_1F_0d^{-1}R_0$, where  $|d|\le \d<|R_1|$, contrary
to the choice of $|R|$ as a minimal conjugator. This contradiction shows that $p$ lies on $F_\g$, for 
all $p\in P$.  

Now choose $p\in P$, so the geodesic $d$  
from some point $q$ on $h^{-1}$ to $p$ has  length at most $2\d$. Now $d$ is a conjugator, for a cyclic permutation
of $W$ ending in a long edge, and by minimality of $R$ we have 
$|R|\le |h|/2 +2\d$, as required.
\label{endthm}
\end{proof}

%
\bibliography{qpe}{}  

\begin{thebibliography}{10}

\bibitem{AlonsoBrady}
J.~M. Alonso, T.~Brady, D.~Cooper, V.~Ferlini, M.~Lustig, M.~Mihalik,
  M.~Shapiro, and H.~Short.
\newblock Notes on word hyperbolic groups.
\newblock In {\em Group theory from a geometrical viewpoint ({T}rieste, 1990)},
  pages 3--63. World Sci. Publ., River Edge, NJ, 1991.
\newblock Edited by Short.

\bibitem{BacherVdovina}
Roland Bacher and Alina Vdovina.
\newblock Counting 1-vertex triangulations of oriented surfaces.
\newblock {\em Discrete Math.}, 246(1-3):13--27, 2002.
\newblock Formal power series and algebraic combinatorics (Barcelona, 1999).

\bibitem{BridsonHowie05}
Martin~R. Bridson and James Howie.
\newblock Conjugacy of finite subsets in hyperbolic groups.
\newblock {\em {IJAC}}, 15(4):725--756, 2005.

\bibitem{BridsonHaefliger}
M.R. Bridson and A.~H{\"a}fliger.
\newblock {\em Metric Spaces of Non-Positive Curvature}.
\newblock Die Grundlehren der mathematischen Wissenschaften in
  Einzeldarstellungen. Springer, 1999.

\bibitem{ComerfordEdmunds81}
Leo~P.\ Comerford and Charles~C.\ Edmunds.
\newblock Quadratic equations over free groups and free products.
\newblock {\em J.\ Algebra}, pages 276--297, 1981.

\bibitem{ComerfordEdmunds89}
Leo~P. Comerford, Jr. and Charles~C. Edmunds.
\newblock Solutions of equations in free groups.
\newblock In {\em Group theory ({S}ingapore, 1987)}, pages 347--356. de
  Gruyter, Berlin, 1989.

\bibitem{Culler81}
M.\ Culler.
\newblock Using surfaces to solve equations in free groups.
\newblock {\em Topology}, pages 237--300, 1981.

\bibitem{Culler78}
Marc~Edward Culler.
\newblock {\em G{enus} {of} {elements} {in} {a} {free} {group}}.
\newblock ProQuest LLC, Ann Arbor, MI, 1978.
\newblock Thesis (Ph.D.)--University of California, Berkeley.

\bibitem{DG10}
François Dahmani and Vincent Guirardel.
\newblock Foliations for solving equations in groups: free, virtually free, and
  hyperbolic groups.
\newblock {\em J Topology}, 3(2):343--404, 2010.

\bibitem{edmunds75}
Charles~C. Edmunds.
\newblock On the endomorphism problem for free groups.
\newblock {\em Comm. Algebra}, 3:1--20, 1975.

\bibitem{edmunds79}
Charles~C. Edmunds.
\newblock On the endomorphism problem for free groups. {II}.
\newblock {\em Proc. London Math. Soc. (3)}, 38(1):153--168, 1979.

\bibitem{ghys90}
{\'E}.~Ghys and P.~de~La~Harpe.
\newblock {\em Sur les groupes hyperboliques d'apr{\`e}s Mikhael Gromov:}.
\newblock Progress in mathematics. Birkh{\"a}user, 1990.

\bibitem{GoldsteinTurner79}
R.Z.\ Goldstein and E.C.\ Turner.
\newblock Applications of topological graph theory to group theory.
\newblock {\em Math.\ Z.}, 165:1--10, 1979.

\bibitem{GrigorchukLysionok}
R.~I.~\ Grigorchuk and I.~G.\ Lysionok.
\newblock A description of solutions of quadratic equations in hyperbolic
  groups.
\newblock {\em Internat.\ J.\ Algebra Comput.}, 2(3):237--274, 1992.

\bibitem{GK92}
R.I. Grigorchuk and P.F. Kurchanov.
\newblock On quadratic equations in free groups.
\newblock In Kostrikin~A.I Bokut'~L.A., Ershov Yu~L., editor, {\em Proceedings
  of the International Conference on Algebra Dedicated to the Memory of A.I.
  Mal'cev}, volume~1 of {\em Contemporary Mathematics - American Mathematical
  Society}, pages 159--171. American Mathematical Society, 1992.

\bibitem{Gromov87}
M.~Gromov.
\newblock Hyperbolic groups.
\newblock In S.M.\ Gersten, editor, {\em Essays in group theory}, volume~8 of
  {\em MSRI Publications}, pages 75--263. Springer, 1987.

\bibitem{Gromov93}
M.~Gromov.
\newblock Asymptotic invariants of infinite groups.
\newblock In {\em Geometric group theory, {V}ol.\ 2 ({S}ussex, 1991)}, volume
  182 of {\em London Math. Soc. Lecture Note Ser.}, pages 1--295. Cambridge
  Univ. Press, Cambridge, 1993.

\bibitem{Comerford81}
Leo P~Comerford Jr.
\newblock Quadratic equations over small cancellation groups.
\newblock {\em Journal of Algebra}, 69(1):175 -- 185, 1981.

\bibitem{KMTV}
O.~{Kharlampovich}, A.~{Mohajeri}, A.~{Taam}, and A.~{Vdovina}.
\newblock {Quadratic Equations in Hyperbolic Groups are NP-complete}.
\newblock {\em ArXiv e-prints}, June 2013.

\bibitem{KV}
O.~{Kharlampovich} and A.~{Vdovina}.
\newblock Linear estimates for solutions of quadratic equations in free groups.
\newblock {\em International Journal of Algebra and Computation},
  22(01):1250004, 2012.

\bibitem{LysenokMyasnikov}
Igor Lysenok and Alexei Myasnikov.
\newblock A polynomial bound on solutions of quadratic equations in free
  groups.
\newblock {\em Proceedings of the Steklov Institute of Mathematics},
  274(1):136--173, 2011.

\bibitem{lysionok89}
I.G.\ Lysionok.
\newblock On some algorithmic properties of hyperbolic groups.
\newblock {\em Math. USSR Izv.}, 35:145--163, 1990.
\newblock Translated from Izv. Akad. Nauk. SSSR Ser. Math, vol 54(3), 1989.

\bibitem{Makanin82}
G.S.\ Makanin.
\newblock Equations in a free group. (russian).
\newblock {\em Izv.\ Akad.\ Nauk SSSR Ser.\ Mat.}, 46(6):1199--1273, 1982.
\newblock English translation in: Math. USSR--Izv. 21 (1983), no. 3, 546--582.

\bibitem{ols89}
A.~Yu. Ol'shanskii.
\newblock Diagrams of homomorphisms of surface groups.
\newblock {\em Sibirsk. Mat. Zh.}, 30(6):150--171, 1989.

\bibitem{Schupp79}
Schupp P.E.
\newblock Quadratic equations in groups, cancellation diagrams on compact
  surfaces, and automorphisms of surface groups.
\newblock In Boone W.~W. Adian S.~I., Higman~G., editor, {\em Word problems II
  : the Oxford book}. North-Holland Amsterdam, 1979.

\bibitem{razborov85}
A~A Razborov.
\newblock On systems of equations in a free group.
\newblock {\em Mathematics of the USSR-Izvestiya}, 25(1):115, 1985.

\bibitem{RipsSela95}
E.\ Rips and Z.\ Sela.
\newblock Canonical representatives and equations in hyperbolic groups.
\newblock {\em Invent. Math.}, 120(3):489--512, 1995.

\bibitem{vdovina97}
Alina Vdovina.
\newblock Products of commutators in free products.
\newblock {\em Internat. J. Algebra Comput.}, 7(4):471--485, 1997.

\bibitem{wicks62}
M.~J. Wicks.
\newblock Commutators in free products.
\newblock {\em J. London Math. Soc.}, 37:433--444, 1962.

\bibitem{wicks73}
M.~J. Wicks.
\newblock The equation $x^2y^2 = g$ over free products.
\newblock In {\em {Proceedings of the {S}econd {C}ongress of the {S}ingapore
  {N}ational {A}cadamy of {S}cience}}, pages 238--248, 1973.

\end{thebibliography}
\bibliographystyle{plain}  
\end{document}